\documentclass[11pt,fullpage,doublespace]{article}
\usepackage{graphicx} % a package for including diagrams
\usepackage{setspace} % provides the doublespace command below
\usepackage{latexsym} % must have basic LaTeX symbols
\usepackage{amsfonts} % another must have AMS fonts
\usepackage{amsmath} % AMS math stuff like matrices and such
\usepackage{amssymb}
\usepackage{accents}
\usepackage{textcomp}
\usepackage{undertilde}
\usepackage{enumerate}
\usepackage{mathrsfs}
\usepackage{bm}
\usepackage{stmaryrd}
\usepackage[margin=1in]{geometry}
\usepackage{amsthm}
\usepackage{ifsym}
\usepackage{amssymb,latexsym,amsmath}
\usepackage{graphics}
\usepackage{titlesec}
\usepackage{changepage}
\usepackage{lipsum}
\usepackage[retainorgcmds]{IEEEtrantools}
\usepackage{thmtools}
\usepackage{comment}
\usepackage[colorlinks=true,linkcolor=blue]{hyperref}
\declaretheorem[style=definition,qed=$\dashv$,numberwithin=section]
{definition}

\declaretheorem[style=plain,sibling=definition]{theorem}
\declaretheorem[style=plain,sibling=definition]{lemma}

\declaretheorem[style=plain,sibling=definition]{question}
\declaretheorem[style=plain,sibling=definition]{conjecture}
\declaretheorem[style=definition,sibling=definition]{remark}
\declaretheorem[style=plain,sibling=definition]{claim}
\declaretheorem[style=plain,sibling=definition]{claim*}

\titleformat{\section}{\normalsize\centering}{\thesection.}{1em}{}
\titleformat{\subsection}{\normalsize\centering}{\thesubsection.}{1em}{}
\titleformat{\subsubsection}{\normalsize}{\thesubsubsection.}{1em}{}
\numberwithin{equation}{section}

\newcommand{\gTheta}{\mathsf{G}}

\newcommand{\rest}{\restriction}

\newcommand{\param}{\mathfrak{P}}
\newcommand{\dom}{{\rm dom}}

\newcommand{\lh}{{\rm lh}}

\def\a{\alpha}

\def\l{\lambda}

\renewcommand{\models}{\vDash}
\newcommand{\powerset}{{\wp }}

\def\P{{\mathcal{P} }}
\def\W{{\mathcal{W} }}
\def\Q{{\mathcal{ Q}}}

\def\R{{\mathcal R}}

\def\M{{\mathcal{M}}}
\def\N{{\mathcal{N}}}
\def\F{{\mathcal{F}}}
\def\T {{\mathcal{T}}}
\def\U{{\mathcal{U}}}
\def\S{{\mathcal{S}}}
\def\G{{\mathcal{G}}}

\def\VT{{\vec{\mathcal{T}}}}

\newcommand{\ins}{\trianglelefteq}

\newcommand{\pins}{\triangleleft}

\newcommand{\crit}{\mathrm{crit}}
\newcommand{\RR}{\mathbb{R}}

\newcommand{\OR}{\text{OR}}
\newcommand{\J}{\mathcal J}
\newcommand{\un}{\cup}
\newcommand{\sub}{\subseteq}
\newcommand{\om}{\omega}
\newcommand{\ZF}{\mathsf{ZF}}
\newcommand{\AD}{\mathsf{AD}}
\newcommand{\ZFC}{\mathsf{ZFC}}
\newcommand{\CC}{\mathbb{C}}
\newcommand{\es}{\mathbb{E}}
\newcommand{\Tt}{\mathcal{T}}

\newcommand{\Uu}{\mathcal{U}}

\newcommand{\conc}{\ \widehat{\ }\ }

\newcommand{\Lp}{\mathrm{Lp}}
\newcommand{\sats}{\vDash}
\newcommand{\cut}{\backslash}

\newcommand{\core}{\mathcal{C}}
\newcommand{\pow}{\mathcal{P}}
\newcommand{\inter}{\cap}

\newcommand{\BB}{\mathbb{B}}
\newcommand{\Coll}{\mathrm{Col}}

\renewcommand{\l}{\mathit{l}}

\renewcommand{\OR}{\textit{o}}
\newcommand{\nth}{\text{th}}
\newcommand{\Ord}{\text{Ord}}

\newcommand{\com}{\circ}
\newcommand{\HC}{\mathrm{HC}}

\newcommand{\BBB}{\mathfrak{B}}
\newcommand{\Ll}{\mathcal{L}}
\newcommand{\all}{\forall}
\newcommand{\ex}{\exists}

\newcommand{\id}{\mathrm{id}}
\newcommand{\cross}{\times}
\newcommand{\Mbar}{\bar{\M}}
\newcommand{\gammabar}{\bar{\gamma}}
\newcommand{\Ttbar}{\bar{\Tt}}
\newcommand{\bbar}{\bar{b}}
\newcommand{\Bbar}{\bar{B}}
\newcommand{\Rbar}{\bar{\R}}

\newcommand{\C}{\mathcal{C}}
\newcommand{\her}{\mathcal{H}}

\newcommand{\rank}{\mathrm{rank}}

\newcommand{\MFsharp}{\mathfrak{M}}
\newcommand{\Jop}{\J^{\mathrm{m}}}
\newcommand{\Pbar}{\bar{\P}}
\newcommand{\abar}{\bar{a}}

\newcommand{\Fop}{\mathcal{F}}

\newcommand{\g}{\mathrm{g}}

\newcommand{\trancl}{\mathrm{trancl}}
\newcommand{\ahat}{\hat{a}}

\newcommand{\wh}{\mathrm{wh}}

\newcommand*{\TitleFont}{%
      \usefont{\encodingdefault}{\rmdefault}{b}{n}%
      \fontsize{12}{16}%
      \selectfont}

 \input xy
 \xyoption{all}
\begin{document}
\title{\TitleFont $\textsf{PFA}$ AND GUESSING MODELS}
\author{\fontsize{11}{13} \selectfont NAM TRANG\\
Department of Mathematics\\ 
University of California, Irvine\\
Email: ntrang@math.uci.edu.}

\date{}
\maketitle
\begin{abstract}
This paper explores the consistency strength of The Proper Forcing Axiom ($\textsf{PFA}$) and the theory (T) which involves a variation of the Viale-Wei$\ss$ guessing hull principle. We show that (T) is consistent relative to a supercompact cardinal. The main result of the paper is Theorem \ref{main_technical_theorem}, which implies that the theory ``$\sf{AD}$$_\mathbb{R} + \Theta$ is regular" is consistent relative to (T) and to $\textsf{PFA}$. This improves significantly the previous known best lower-bound for consistency strength for (T) and $\textsf{PFA}$, which is roughly ``$\sf{AD}$$_\mathbb{R} + \textsf{DC}$". 
\end{abstract}

Suppose $\kappa < \gamma$ are uncountable cardinals and $X\prec H_\gamma$ is such that $\kappa \leq |X|$. Let $\powerset_\kappa(X) = \{a\in \powerset(X) \ | \ |a| < \kappa\}$. We say that $X$ is \textit{$\kappa$-guessing} if for all $a \in X$ and for all $b\subseteq a$ such that $c\cap b \in X$ for all $c\in X \cap \powerset_{\kappa}(X)$ then $b$ is \textit{$X$-guessed}, i.e. there is some $c\in X$ such that $c\cap X = b\cap X$. Such a $b$ satisfying the hypothesis of the previous sentence is called \textit{$\kappa$-approximated by $X$}. So a hull $X$ is $\kappa$-guessing if whenever $a\in X$ and whenever $b\subseteq a$ is $\kappa$-approximated by $X$, then $b$ is $X$-guessed. 
%Suppose $\kappa \in X$ and $X\cap \kappa =_{\textrm{def}} \mu \in \kappa$ (so $\mu$ is the critical point of the uncollapse map $\pi_X: M_X\rightarrow X$ associated to $X$); we say that $X$ is \textit{amenably closed at $\kappa$} if whenever $A\subseteq \mu$ is such that $A\cap \alpha \in X$ (equivalently $A\cap\alpha\in M_X$) for all $\alpha<\mu$, then $A\in M_X$. We will simply say $X$ is amenably closed from now on.

In this paper, we study the strength of the following theories
\begin{itemize}
\item The Proper Forcing Axiom ($\textsf{PFA})$;
\item (T): there is a cardinal $\lambda\geq 2^{\aleph_2}$ such that the set $\{X \prec (H_{\lambda^{++}},\in) \ | \ |X| = \aleph_2$, $X^\omega \subseteq X$, $\omega_2\subset X$, and $X$ is $\omega_2$-guessing $\}$ is stationary.
\end{itemize}

\indent Guessing models in \cite{viale2011consistency} are $\omega_1$-guessing in the above notations. It's not clear that the theory (T) is consistent with $\textsf{PFA}$ (in contrast to Viale-Wei$\ss$ principle $\sf{ISP}(\omega_2)$, which asserts the existence of stationary many $\omega_1$-guessing models of size $\aleph_1$ of $H_{\lambda}$ for all sufficiently large $\lambda$). However, it's possible that (T) is a consequence of or at least consistent with a higher analog of $\sf{PFA}$. 

%Amenably closed hulls of size $\omega_1$ were first shown to exist by Woodin in $\mathbb{P}_{\rm{max}}$ extensions of models of ``$\sf{AD}_\mathbb{R}$$ + \Theta$ is regular''. The principle ``there are stationarily many amenably closed $X\prec H_{\omega_3}$ such that $|X|=\omega_1$'' when combined with other (relatively weak) combinatorial principles implies $\sf{AD}$$^{L(\mathbb{R})}$ (see \cite{PFA}); it's not clear how to get more strength significantly beyond  $\sf{AD}$$^{L(\mathbb{R})}$ in the situation of Theorem 0.3 of \cite{PFA}. It's easy to see that the principle ``there are stationarily many amenably closed $X\prec H_{\kappa^{++}}$ such that $|X|=\omega_2$ and $X^\omega\subseteq X$'' follows from (T) (see Foonote 5). Furthermore, the results in this paper actually show that when combined with instances of failures of squares, which are also consequences of $(T)$, this yields a model of ``$\textsf{AD}_\mathbb{R} + \Theta$ is regular''.

The outline of the paper is as follows. In section \ref{BasicAD+Facts}, we review some $\sf{AD}$$^+$ facts that we'll be using in this paper. In section \ref{upperbound}, using a Mitchell-style forcing, we prove
\begin{theorem}
\label{FromASpct}
Con($\sf{ZFC}$ + there is a supercompact cardinal) $\Rightarrow$ Con$\rm{(T)}$.
\end{theorem} 
Of course, it is well-known that $\sf{PFA}$ is consistent relative to the existence of a supercompact cardinal. Theorem \ref{PFA_Guess} suggests that it's reasonable to expect $\sf{PFA}$ and (T) are equiconsistent.

Recall, for an infinite cardinal $\lambda$, the principle $\square_\lambda$ asserts the existence of a sequence $\langle C_\alpha \ | \ \alpha < \lambda^+ \rangle$ such that for each $\alpha < \lambda^+$,
\begin{itemize}
\item $C_\alpha$ is club in $\alpha$;
\item for each limit point $\beta$ of $C_\alpha$, $C_\beta = C_\alpha\cap \beta$;
\item the order type of $C_\alpha$ is at most $\lambda$.
\end{itemize}
The principle $\square(\lambda)$ asserts the existence of a sequence $\langle C_\alpha \ | \ \alpha < \lambda \rangle$ such that
\begin{enumerate}[(1)]
\item for each $\alpha < \lambda$,
\begin{itemize}
\item each $C_\alpha$ is club in $\alpha$;
\item for each limit point $\beta$ of $C_\alpha$, $C_\beta = C_\alpha \cap \beta$; and
\end{itemize}
\item there is no thread through the sequence, i.e., there is no club $E\subseteq \lambda$ such that $C_\alpha = E\cap \alpha$ for each limit point $\alpha$ of $E$.
\end{enumerate}

Note that $\square_\lambda$ implies $\square(\lambda^+)$ (equivalently, $\neg \square(\lambda^+)$ implies  $\neg \square_\lambda$). The main technical theorem of the paper, proved in Section \ref{lowerbound}, is the following.

\begin{theorem}\label{main_technical_theorem}
Supppose $\kappa$ is a cardinal such that $\kappa^\omega=\kappa$. Suppose for each cardinal $\alpha\in [\kappa^+,(2^\kappa)^+]$, $\neg \square(\alpha)$. Then letting $G\subseteq Col(\omega,\kappa)$ be $V$-generic, in $V[G]$, there is a transitive $M$ containing $\rm{OR}\cup \mathbb{R}$ such that $M\vDash ``\sf{AD}$$_\mathbb{R} + \Theta$ is regular".
\end{theorem}
\begin{remark}\label{firstRmk}
We remark that the proof of Theorems \ref{main_technical_theorem} actually shows that we can get in $V$ a model $M$ containing $\rm{OR}\cup \mathbb{R}$ that satisfies ``$\sf{AD}_\mathbb{R} + \Theta$ is regular", see Remark \ref{InV}. This is because we actually construct a hod pair $(\P,\Sigma)$ that generates a model of ``$\sf{AD}_\mathbb{R} + \Theta$ is regular" with the property that $(\P,\Sigma\rest V)\in V$.
\end{remark}

Hence, as a corollary, we establish the following theorem, which improves upon the conclusion of Corollary 0.2 of \cite{JSSS}. %See Section \ref{lowerbound} for a discussion that the hypothesis of Theorem \ref{PFA_Guess} implies the hypothesis of Theorem \ref{main_technical_theorem}.
\begin{theorem}\label{PFA_Guess}
Suppose $S$ is one of the following theories:
\begin{enumerate}
\item $\sf{PFA}$,
\item $\rm{(T)}$,
\item there is a strongly compact cardinal.
\end{enumerate}
If $S$ holds, then there is a transitive model $M$ containing $\mathbb{R} \cup \rm{OR}$ such that $M\vDash ``\sf{AD}_\mathbb{R} + \Theta$ is regular". \footnote{To the best of the author's knowledge, the result in this paper gives the best lower-bound obtained from any combinatorial principle not augmented by large cardinal assumptions.}
\end{theorem}
\begin{proof}
We apply Theorem \ref{main_technical_theorem} and Remark \ref{firstRmk}. We just need to verify that $S$ implies the hypothesis of Theorem \ref{main_technical_theorem}. If $S$ is either $\sf{PFA}$ or (T), we take $\kappa = \aleph_2$. It's well-known that the hypothesis in Theorem \ref{main_technical_theorem} regarding threadability follows from $S$. \footnote{In the case $S$ is (T), see \cite{viale2011consistency} and \cite{weiss2012combinatorial}.} Otherwise, take $\kappa$ to be a signular, strong limit cardinal of uncountable cofinality above a strongly compact cardinal. By \cite{solovay1974strongly}, the hypothesis of Theorem \ref{main_technical_theorem} holds at $\kappa$.
\end{proof}
%A consequence of the proof of Theorem \ref{main_technical_theorem} is the following (compare this with Theorem 0.1 of \cite{JSSS}).
%\begin{theorem}\label{consecutiveSquareFailures}
%Suppose $\kappa<\delta$ are cardinals such that:
%\begin{itemize}
%\item $\delta = \kappa^+$;
%\item $\kappa^\omega = \kappa$ and $2^\kappa = \delta$;
%\item $\neg \square(\delta)$ and $\neg \square(\delta^+)$.\footnote{Suppose there is a strongly compact cardinal. Let $\kappa$ be a singular, strong limit cardinal of uncountable cofinality above a strongly compact cardinal. Let $\delta=\kappa^+$. Then it's easy to see that $\kappa$ and $\delta$ satisfy the hypothesis of Theorem \ref{consecutiveSquareFailures}.}
%\end{itemize}
%Then letting $G\subseteq Col(\omega,\kappa)$ be $V$-generic, in $V[G]$, there is a transitive $M$ containing $\rm{OR}\cup \mathbb{R}$ such that $M\vDash ``\sf{AD}$$_\mathbb{R} + \Theta$ is regular".\footnote{This theorem arises from conversations with John Steel.}
%\end{theorem}

In Subsection \ref{framework}, we lay out the framework for the core model induction which allows us to construct models of ``$\sf{AD}$$_\mathbb{R} + \Theta$ is regular" from the hypothesis of Theorem \ref{main_technical_theorem}. The actual construction of such models is carried out in Subsections \ref{M1sharp}-\ref{thetareg}. 

We note that the previous best known lower-bound for (T) as well as for $\textsf{PFA}$ is the sharp for a proper class model with a proper class of strong cardinals and a proper class of Woodin cardinals (see \cite[Corollary 0.2]{JSSS}), which is just a bit stronger than ``$\sf{AD}$$_\mathbb{R} + \textsf{DC}$" but is weaker than ``$\sf{AD}$$_\mathbb{R} + \Theta$ is regular". The method used in this paper is the core model induction method, which can be used to further improve the lower bounds for (T) and for $\textsf{PFA}$, as opposed to the method in \cite{JSSS}, which seems hard to generalize.  In fact, it's possible to improve the lowerbound consistency strength for (T) and for $\textsf{PFA}$ in Theorem \ref{PFA_Guess} to ``$\sf{AD}$$_\mathbb{R} + \Theta$ is measurable" and beyond. These results involve a combination of the core model induction and techniques for constructing hod mice beyond those developed in \cite{ATHM} and hence will appear in a future publication. 

Let $\textsf{LSA}$ denote the theory ``$\sf{AD}$$^+ + \Theta = \theta_{\alpha+1} + \theta_\alpha$ is the largest Suslin cardinal." $\sf{LSA}$ was first isolated by Woodin in \cite{Woodin} and is very recently shown to be consistent by G. Sargsyan. It is one of the strongest determinacy theories known to be consistent. We conjecture that  
\begin{conjecture}\label{ConLSA}
Con$\rm{(T)}$ $\Rightarrow$ Con$(\sf{LSA})$ and Con$(\sf{PFA})$ $\Rightarrow$ Con$(\sf{LSA})$.
\end{conjecture}

\noindent We are hopeful that methods used in this paper and their extensions can be used to settle the conjecture.
\\
\\
\noindent \textbf{Acknowledgement.} The author would like to thank John Steel and Trevor Wilson for many useful conversations regarding Chapter 3 of the paper. We would like to thank Christoph Wei$\ss$  for helpful discussions regarding his work on guessing hulls principle and related concepts. We would also like to thank Martin Zeman and the referee for pointing out several typos and mistakes in an earlier version of the paper.
\section{BASIC FACTS ABOUT $\sf{AD}$$^+$}
\label{BasicAD+Facts}
We start with the definition of Woodin's theory of $\sf{AD}$$^+$. In this paper, we identify $\mathbb{R}$ with $\omega^{\omega}$. We use $\Theta$ to denote the sup of ordinals $\alpha$ such that there is a surjection $\pi: \mathbb{R} \rightarrow \alpha$. Under $\textsf{AC}$, $\Theta$ is just the successor cardinal of the continuum. In the context of $\sf{AD}$, $\Theta$ is shown to be the supremum of $w(A)$\footnote{$w(A)$ is the Wadge rank of $A$.} for $A\subseteq \mathbb{R}$. The definition of $\Theta$ relativizes to any determined pointclass (with sufficient closure properties). We denote $\Theta^\Gamma$ for the sup of $\alpha$ such that there is a surjection from $\mathbb{R}$ onto $\alpha$ coded by a set of reals in $\Gamma$.
\begin{definition}
\label{AD+}
$\sf{AD}$$^+$ is the theory $\sf{ZF} + \sf{AD}$$ + \sf{DC}_{\mathbb{R}}$ and 
\begin{enumerate}
\item for every set of reals $A$, there are a set of ordinals $S$ and a formula $\varphi$ such that $x\in A \Leftrightarrow L[S,x] \vDash \varphi[S,x]$. $(S,\varphi)$ is called an $\infty$-Borel code\index{$\infty$-Borel code} for $A$;
\item for every $\lambda < \Theta$, for every continuous $\pi: \lambda^\omega \rightarrow \omega^{\omega}$, for every $A \subseteq \mathbb{R}$, the set $\pi^{-1}[A]$ is determined.
\end{enumerate}
\end{definition}
\noindent $\sf{AD}$$^+$ is equivalent to ``$\sf{AD}$$\ +\ $the set of Suslin cardinals is closed". Another, perhaps more useful, characterization of $\sf{AD}$$^+$ is ``$\sf{AD}$$ + \Sigma_1$ statements reflect into Suslin co-Suslin sets'' (see \cite{steelderived} for the precise statement). 
\\
\indent Let $A\subseteq \mathbb{R}$, we let $\theta_A$\index{$\theta_A$} be the supremum of all $\alpha$ such that there is an $OD(A)$ surjection from $\mathbb{R}$ onto $\alpha$. If $\Gamma$ is a determined pointclass, and $A\in \Gamma$, we write $\Gamma\rest A$ for the set of $B\in\Gamma$ which is Wadge reducible to $A$. If $\alpha < \Theta^\Gamma$, we write $\Gamma\rest \alpha$ for the set of $A\in \Gamma$ with Wadge rank strictly less than $\alpha$.
\begin{definition}[\sf{AD}$$$^+$]
\label{Solovaysequence}
The \textbf{Solovay sequence} is the sequence $\langle\theta_\alpha \ | \ \alpha \leq \lambda\rangle$ where
\begin{enumerate}
\item $\theta_0$ is the supremum of ordinals $\beta$ such that there is an $OD$ surjection from $\mathbb{R}$ onto $\beta$;
\item if $\alpha>0$ is limit, then $\theta_\alpha = \sup\{\theta_\beta \ | \ \beta<\alpha\}$;
\item if $\alpha =\beta + 1$ and $\theta_\beta < \Theta$ (i.e. $\beta < \lambda$), fixing a set $A\subseteq \mathbb{R}$ of Wadge rank $\theta_\beta$, $\theta_\alpha$ is the sup of ordinals $\gamma$ such that there is an $OD(A)$ surjection from $\mathbb{R}$ onto $\gamma$, i.e. $\theta_\alpha = \theta_A$.
\end{enumerate}
\end{definition}
Note that the definition of $\theta_\alpha$ for $\alpha = \beta+1$ in Definition \ref{Solovaysequence} does not depend on the choice of $A$. For a pointclass $\Gamma$ that satisfies $\sf{AD}^+$ and is sufficiently closed, we can also define the Solovay sequence $\langle \theta_\alpha^\Gamma \ | \ \alpha\leq \lambda \rangle$ of $\Gamma$ like above. For $\alpha\leq \lambda$, we say $\Gamma\rest \theta^\Gamma_\alpha$ is a \textit{Solovay initial segment} of $\Gamma$.

Roughly speaking the longer the Solovay sequence is, the stronger the associated $\sf{AD}$$^+$-theory is. For instance the theory $\sf{AD}$$_\mathbb{R} + \textsf{DC}$ is strictly stronger than $\sf{AD}$$_\mathbb{R}$ since by \cite{solovay1978independence}, $\textsf{DC}$ implies cof$(\Theta)>\omega$ while the minimal model\footnote{From here on, whenever we talk about ``models of $\textsf{AD}^+$", we always mean those $M$ that contain $\rm{OR}\cup \mathbb{R}$ and satisfy $\sf{AD}^+$.} of $\sf{AD}$$_\mathbb{R}$ satisfies $\Theta = \theta_\omega$. $\sf{AD}$$_\mathbb{R} + \Theta$ is regular is much stronger still as it implies the existence of many models of $\sf{AD}$$_\mathbb{R} + \textsf{DC}$. We end this section with a theorem of Woodin, which produces models with Woodin cardinals in $\sf{AD}$$^+$. The theorem is important in the HOD analysis of such models.
\begin{theorem}[Woodin, see \cite{koellner2010large}]
Assume $\sf{AD}$$^+$. Let $\langle \theta_\alpha \ | \ \alpha \leq \Omega\rangle$ be the Solovay sequence. Suppose $\alpha = 0$ or $\alpha = \beta+1$ for some $\beta < \Omega$. Then $\rm{HOD} $ $ \vDash \theta_\alpha$ is Woodin.
\end{theorem}

\section{UPPER-BOUND CONSISTENCY STRENGTH OF (T)}
\label{upperbound}
In this section, we prove Theorem \ref{FromASpct}. We follow closely the construction of Section 3 in \cite{krueger2008general}\footnote{For the reader's convenience, our $\omega_1$ will play the role of $\mu$ in Section 3 of \cite{krueger2008general}, our $\omega_2$ will play the role of $\kappa$ there, and finally our $\alpha$ plays the same role as the $\alpha$ in \cite{krueger2008general}.}. We use Even and Odd to denote the classes of even ordinals and odd ordinals respectively. We assume the ground model $V$ satisfies 
\begin{center}
$\omega_1^\omega=\omega_1 + \omega_2^\omega = \omega_2 + \alpha$ is supercompact. 
\end{center}
Consider the following forcing iteration
\begin{center}
\indent \indent $\langle \mathbb{P}_i, \dot{\mathbb{Q}}_j \ | \ i\leq \alpha, j<\alpha\rangle$,
\end{center}
with two partial orderings $\leq$ and $\leq^*$, where $\leq$ is the standard partial ordering on posets and $\leq^*$ is defined by letting $p\textasteriskcentered \dot{r} \leq^* q\textasteriskcentered \dot{s}$ if $p\textasteriskcentered \dot{r} \leq q\textasteriskcentered \dot{s}$ and $p=q$. Inductively, we ensure that the following hold.
\begin{enumerate}
\item If $i<\alpha$ is even then $\mathbb{P}_i$ forces $\dot{\mathbb{Q}}_i = \textrm{ADD}(\omega_1)$, where $\textrm{ADD}(\omega_1)$ is the standard forcing for adding a Cohen subset of $\omega_1$ with countable conditions.
\item If $j<\alpha$ is odd, $\mathbb{P}_{j-1}$ forces $\langle \dot{Q}_{j-1}\textasteriskcentered \dot{Q}_j,\leq^*\rangle = \langle \textrm{ADD}(\omega_1)\textasteriskcentered \textrm{Col}(\omega_2,j),\leq^* \rangle$, where $\textrm{Col}(\omega_2,j)$ is the standard forcing that collapses $j$ to $\omega_2$ using conditions of size at most $\aleph_1$. So for all $i<\alpha$, $\mathbb{P}_i$ forces $\dot{\mathbb{Q}}_i$ is $\omega_1$-closed and furthermore, if $i$ is odd, then $\mathbb{P}_{i-1}$ forces $\langle \dot{\mathbb{Q}}_{i-1}\textasteriskcentered \dot{\mathbb{Q}}_i, \leq^*\rangle$ is $\omega_2$-strategically closed.
\item If $i\leq \alpha$ is a limit ordinal, then $\mathbb{P}_i$ consists of all partial functions $p:i\rightarrow V$ such that $p\rest j \in \mathbb{P}_j$ for $j<i$, $|\textrm{dom}(p)\cap \textrm{Even}|<\omega_1$, and $|\textrm{dom}(p)\cap \textrm{Odd}|<\omega_2$.
\item For $i\leq \alpha$ and $p,q\in \mathbb{P}_i$, $q\leq p$ in $\mathbb{P}_i$ iff for all $\gamma$ in the domain of $p$, $\gamma$ is in the domain of $q$ and $q\rest \gamma \Vdash q(\gamma) \leq p(\gamma)$.
\item For $i\leq \alpha$ and $p,q\in \mathbb{P}_i$, $q\leq^* p$ in $\mathbb{P}_i$ iff $q\leq p$, $\textrm{dom}(p)\cap \textrm{Even} = \textrm{dom}(q)\cap \textrm{Even}$, and for every $\gamma$ in $\textrm{dom}(p)\cap \textrm{Even}$, $q\rest \gamma\Vdash q(\gamma) = p(\gamma)$. 
\end{enumerate}
By results in Section 3 of \cite{krueger2008general}, we get the following.
\begin{enumerate}[(a)]
\item $\langle\mathbb{P}_\beta,\leq\rangle$ is $\omega_1$-closed for all $\beta\leq \alpha$. In particular $\textrm{ADD}(\omega_1)$ is the same in the ground model and in any intermediate extension. 
\item $\mathbb{P}_\alpha$ preserves $\omega_1, \omega_2$ (as well as stationary subsets of $\omega_1$ and $\omega_2$). 
\item $\forall \eta\leq \alpha$ such that $\eta$ is inaccessible, $\mathbb{P}_\eta$ is $\eta$-cc.
\item $\mathbb{P}_\alpha$ forces $\alpha=\omega_3$.
\item $\forall \eta \leq \alpha$ such that $\eta$ is inaccessible, letting $\mathbb{P}_\alpha=\mathbb{P}_\eta\textasteriskcentered \dot{\mathbb{Q}}$, then $\Vdash_{\mathbb{P}_\eta} \dot{\mathbb{Q}}$ satisfies the $\omega_2$-approximation property, that is, whenever $G\textasteriskcentered H$ is $V$-generic for $\mathbb{P}_\eta\textasteriskcentered \dot{\mathbb{Q}}$, then if for all $x\in V[G][H]$, $x\subseteq V[G]$, it holds that if $x\cap z\in V[G]$ for all $z\in \powerset_{\omega_2}(V[G])^{V[G]}$ then $x\in V[G]$.
\end{enumerate}
%In $V$, let $\mu$ be a normal fine measure on $\powerset_{\alpha}(H_{\alpha^{+++}})$ and $j:V\rightarrow M$ be the ultrapower embedding given by $\mu$. Let $G$ be $\mathbb{P}_\alpha$-generic over $V$. The following lemma completes the proof of Theorem \ref{upperbound}.
\begin{lemma}
\label{VGsatisfiesT}
In $V[G]$, for any regular cardinal $\lambda > \omega_2$, the set $\{X\prec (H_\lambda,\in) \ | \ |X|=\omega_2, \omega_2\subset X,  X^\omega \subseteq X, \textrm{ and } X \textrm{ is } \omega_2\textrm{-guessing}\}$ is stationary. In particular, $\rm{(T)}$ holds in $V[G]$.
\end{lemma}
\begin{proof}
Since $\mathbb{P}_\alpha$ is $\omega_1$-closed, $\omega_2^\omega = \omega_2$ in $V$, and $\omega_2^V = \omega_2^{V[G]}$, 
\begin{center}
$(\omega_2^\omega)^{V[G]} = (\omega_2^\omega)^V = \omega_2^V = \omega_2^{V[G]}$.
\end{center}
It's also clear from (a), (c), and the fact that $\alpha$ is an inaccessible limit of inaccessibles that $\omega_2^\omega = \omega_2$ and $2^{\omega_2} = \omega_3$ in $V[G]$.

Working in $V[G]$, fix a regular cardinal $\lambda>\omega_2$. Let
\begin{center}
$T = \{X \prec H_{\lambda} \ | \ |X| = \omega_2 \wedge \omega_2\subset X \wedge X^\omega \subseteq X \wedge X \textrm{ is } \omega_2\textrm{-guessing}\}$.
\end{center}
We show $T$ is stationary. In $V$, let $j: V\rightarrow M$ witness that $\alpha$ is $H_{\lambda}$-supercompact. Let $G\textasteriskcentered H$ be $\mathbb{P}_{j(\alpha)}$-generic over $V$. $j$ canonically lifts to $j^+: V[G] \rightarrow M[G\textasteriskcentered H]$, where $j^+(\tau_G) = j(\tau)_{G\textasteriskcentered H}$. Let $\mathcal{F}$ be the normal filter defined from $j^+$, that is for all $A\in \powerset_{\omega_3}(H_{\lambda})^{V[G]}$,
\begin{center}
$A\in \mathcal{F} \Leftrightarrow \mathbb{P}_{j(\alpha)}$ forces over $V[G]$ that $j^+[H_{\lambda}^{V[G]}]\in j(A)$.
\end{center}
It's easy to check that $\mathcal{F}$ is a normal filter in $V[G]$. We now check that whenever $G\textasteriskcentered H$ is $V$-generic, then $H_{\lambda}^{V[G]}\in j^+(T)$. Fix such $G,H$ and let $X = j^+[H_{\lambda}^{V[G]}]$. To simplify the notation, we also use $j$ to denote $j^+$. Note that $X\prec H^{M[G][H]}_{j(\lambda)}$.

We first show in $M[G][H]$, $X^\omega\subseteq X$. Let $a\in X^\omega$ and note that $j^{-1}[a] \in V[G]$. This is because $j^{-1}[a] \subseteq V[G]$ and is a countable sequence in $V[G][H]$ and hence is in $V[G]$ since by construction, in $V[G]$, $\Vdash_{\mathbb{P}_\alpha} ``\dot{\mathbb{Q}}$ is $\omega_1$-closed'', where $\mathbb{P}_{j(\alpha)}=\mathbb{P}_\alpha\textasteriskcentered \dot{\mathbb{Q}}$. This easily implies that $j^{-1}[a]\in H^{V[G]}_{\lambda}$. Hence $j(j^{-1}[a]) = a \in X$. 

Now suppose $b\subseteq z\in X$ is such that $b\subseteq X$ and whenever $d\in X\cap \powerset_{\omega_2}(X)$, $d\cap b \in X$. We want to show there is a $c\in X$ such that $b\cap X = c\cap X$. To this end, note that $j^{-1}[b]\in V[G]$. This uses (e) and the assumption on $b$. Let $c = j(j^{-1}[b])$. Since $j^{-1}[b]\subseteq j^{-1}(z)\in H^{V[G]}_{\lambda}$, $j^{-1}[b]\in H^{V[G]}_{\lambda}$; this gives $c\in X$. It's easy then to check that $c\cap X = b \cap X$ ($c$ need not equal $b$ though). This completes the proof of the lemma.
\end{proof}
\begin{proof}[Proof of Theorem \ref{FromASpct}] Immediate from Lemma \ref{VGsatisfiesT}.
\end{proof}
\begin{remark} We note that the $\omega_2$-approximation property in (e) is crucial in the proof of Lemma \ref{VGsatisfiesT}. It's used to show that the hull $X$ is $\omega_2$-guessing.
\end{remark} 
%By ``$\sf{AD}$$_\mathbb{R} + \Theta$ is measurable", we mean $\sf{AD}$$^+ + \sf{AD}$$_\mathbb{R}$ and there is an $\mathbb{R}$-complete nonprincipal measure on $\Theta$. ``$\sf{AD}$$_\mathbb{R} + \Theta$ is measurable" implies the existence of many models of $\sf{AD}$$_\mathbb{R} + \Theta$ is regular.

\section{LOWER-BOUND CONSISTENCY STRENGTH OF (T)} 
\label{lowerbound}
In this section, we prove Theorem \ref{main_technical_theorem} and hence Theorem \ref{PFA_Guess}. The next several subsections are dedicated to setting up the core model induction, constructing hod pairs with nice properties that generate $\sf{AD}$$^+$ models. Fix a $V$-generic $G \subseteq \textrm{Col}(\omega,\kappa)$. Let $\lambda=2^\kappa$. For $X\prec H_{\lambda^{++}}$ such that $|X| = \kappa$, $\kappa\subseteq X$, let $\pi_X: M_X \rightarrow H_{\lambda^{++}}$ be the uncollapse map. $\pi_X$ naturally extends to a map, which we also call $\pi_X$ from $M_X[G]$ to $H_{\lambda^{++}}[G]$. The core model induction will occur in $V[G]$. Our smallness assumption throughout this paper is: 
\\
\begin{adjustwidth}{2cm}{2cm}

$(\dag): \ \ \ $ in $V[G]$, there is no model $M$ containing all reals and ordinals such that $M \vDash ``\sf{AD}$$_\mathbb{R} + \Theta$ is regular".

\end{adjustwidth}

Among other things, $(\dag)$ implies:
\begin{itemize}
\item There are no $\textsf{AD}^+$ models $M, N$ such that $\mathbb{R}\cup \textrm{OR} \subseteq M, N$ and $\powerset(\mathbb{R})\cap M \cap N$ is strictly contained in $\powerset(\mathbb{R})\cap M$ and in $\powerset(\mathbb{R})\cap N$ (see \cite{ATHM}). This implies that all $\textsf{AD}^+$ models constructed in the core model induction will end-extend one another.
\item If $M$ is an $\textsf{AD}^+$ model, then \textit{Strong Mouse Capturing} ($\textsf{SMC}$) (see Footnote \ref{MC} for definition of $\sf{SMC}$) and \textit{Generation of Mouse Full Pointclasses} (see \cite[Section 6.1]{ATHM}) hold in $M$ . This fact allows us to use the hod analysis in \cite{ATHM} to construct hod mice. 
\end{itemize}

Using the first consequence of $(\dag)$, we define

\begin{definition}[Maximal pointclass of $\textsf{AD}^+$]\label{MaximalModel}
In $V[G]$, let 
\begin{center}$\Omega = \bigcup \{\powerset(\mathbb{R})\cap M \ | \  \mathbb{R}\cup \textrm{OR} \subset M \wedge M \vDash \textsf{AD}^+\}$.\end{center}
\end{definition}

The rest of the paper is dedicated to analyzing $\Omega$. In particular, we show that in $V[G]$,
\begin{itemize}
\item $\Omega \neq \emptyset$.
\item Letting $\langle \theta_\alpha^\Omega\ | \ \alpha \leq \gamma\rangle$ be the Solovay sequence of $\Omega$, then $\gamma$ is a limit ordinal.
\end{itemize}
We will then deduce that there is indeed a model $M$ of ``$\textsf{AD}_\mathbb{R} + \Theta$ is regular". This contradicts $(\dag)$.

\subsection{FRAMEWORK FOR THE CORE MODEL INDUCTION}
\label{framework}
This section, consisting of several subsections, develops some terminology and framework for the core model induction. The first subsection gives a brief summary of the theory of hod mice developed in \cite{ATHM}. In the next three subsections, we briefly introduce the notions of $\F$-premice, strategy premice, and $(\Theta)$-$g$-organized $\F$-premice developed in \cite{trang2013}. For a full development of these concepts as well as proofs of lemmas stated below, the reader should consult \cite{trang2013}. These subsections summarize the theory and results in \cite{trang2013} to make the paper self-contained. The reader who wishes to see the main argument can skip them on the first read, and go back when needed. The next subsection discusses the $S$-constructions, which allow us to translate hybrid mice over a set $a$ to hybrid mice over a set $b$ where $a$ and $b$ are closely related. The last subsection defines core model induction operators, which are operators that we construct during the course of the core model induction in this paper.
\subsubsection{A BRIEF INTRODUCTION TO HOD MICE}
\label{briefHodMice} 
In this paper, a hod premouse $\P$ is one defined as in \cite{ATHM}. The reader is advised to consult \cite{ATHM} for basic results and notations concerning hod premice and mice. Let us mention some basic first-order properties of a hod premouse $\P$. There are an ordinal $\lambda^\P$ and sequences $\langle(\P(\alpha),\Sigma^\P_\alpha) \ | \ \alpha < \lambda^\P\rangle$ and $\langle \delta^\P_\alpha \ | \ \alpha \leq \lambda^\P  \rangle$ such that 
\begin{enumerate}
\item $\langle \delta^\P_\alpha \ | \ \alpha \leq \lambda^\P  \rangle$ is increasing and continuous and if $\alpha$ is a successor ordinal then $\P \vDash \delta^\P_\alpha$ is Woodin;
\item $\P(0) = Lp_\omega(\P|\delta_0)^\P$; for $\alpha < \lambda^\P$, $\P(\alpha+1) = (Lp_\omega^{\Sigma^\P_\alpha}(\P|\delta_\alpha))^\P$; for limit $\alpha\leq \lambda^\P$, $\P(\alpha) = (Lp_\omega^{\oplus_{\beta<\alpha}\Sigma^\P_\beta}(\P|\delta_\alpha))^\P$;
\item $\P \vDash \Sigma^\P_\alpha$ is a $(\omega,o(\P),o(\P))$\footnote{This just means $\Sigma^\P_\alpha$ acts on all stacks of $\omega$-maximal, normal trees in $\P$.}-strategy for $\P(\alpha)$ with hull condensation;
\item if $\alpha < \beta < \lambda^\P$ then $\Sigma^\P_\beta$ extends $\Sigma^\P_\alpha$.
\end{enumerate}
We will write $\delta^\P$ for $\delta^\P_{\lambda^\P}$ and $\Sigma^\P=\oplus_{\beta<\lambda^\P}\Sigma^\P_{\beta}$. Note that $\P(0)$ is a pure extender model. Suppose $\P$ and $\Q$ are two hod premice. Then $\P\trianglelefteq_{hod}\Q$\index{$\trianglelefteq_{hod}$} if there is $\a\leq\l^\Q$ such that $\P=\Q(\a)$. We say then that $\P$ is a \textit{hod initial segment} of $\Q$. $(\P,\Sigma)$ is a \textit{hod pair} if $\P$ is a hod premouse and $\Sigma$ is a strategy for $\P$ (acting on countable stacks of countable normal trees) such that $\Sigma^\P \subseteq \Sigma$ and this fact is preserved under $\Sigma$-iterations. Typically, we will construct hod pairs $(\P,\Sigma)$ such that $\Sigma$ has hull condensation, branch condensation, and is $\Gamma$-fullness preserving for some pointclass $\Gamma$. 

Suppose $(\Q,\Sigma)$ is a hod pair such that $\Sigma$ has hull condensation. $\P$ is a $(\Q,\Sigma)$-hod premouse if there are ordinal $\lambda^\P$ and sequences $\langle(\P(\alpha),\Sigma^\P_\alpha) \ | \ \alpha < \lambda^\P\rangle$ and $\langle \delta^\P_\alpha \ | \ \alpha \leq \lambda^\P  \rangle$ such that 
\begin{enumerate}
\item $\langle \delta^\P_\alpha \ | \ \alpha \leq \lambda^\P  \rangle$ is increasing and continuous and if $\alpha$ is a successor ordinal then $\P \vDash \delta^\P_\alpha$ is Woodin;
\item $\P(0) = Lp^{\Sigma}_\omega(\P|\delta_0)^\P$ (so $\P(0)$ is a $\Sigma$-premouse built over $\Q$); for $\alpha < \lambda^\P$, $\P(\alpha+1) = (Lp_\omega^{\Sigma\oplus\Sigma^\P_\alpha}(\P|\delta_\alpha))^\P$; for limit $\alpha\leq \lambda^\P$, $\P(\alpha) = (Lp_\omega^{\oplus_{\beta<\alpha}\Sigma^\P_\beta}(\P|\delta_\alpha))^\P$;
\item $\P \vDash \Sigma\cap \P$ is a $(\omega,o(\P),o(\P))$\-strategy for $\Q$ with hull condensation;
\item $\P \vDash \Sigma^\P_\alpha$ is a $(\omega,o(\P),o(\P))$\-strategy for $\P(\alpha)$ with hull condensation;
\item if $\alpha < \beta < \lambda^\P$ then $\Sigma^\P_\beta$ extends $\Sigma^\P_\alpha$.
\end{enumerate}
Inside $\P$, the strategies $\Sigma^\P_\alpha$ act on stacks above $\Q$ and every $\Sigma^P_\alpha$ iterate is a $\Sigma$-premouse. Again, we write $\delta^\P$ for $\delta^\P_{\lambda^\P}$ and $\Sigma^\P=\oplus_{\beta<\lambda^\P}\Sigma^\P_{\beta}$. $(\P,\Lambda)$ is a $(\Q,\Sigma)$-hod pair if $\P$ is a $(\Q,\Sigma)$-hod premouse and $\Lambda$ is a strategy for $\P$ such that $\Sigma^P\subseteq \Lambda$ and this fact is preserved under $\Lambda$-iterations. The reader should consult \cite{ATHM} for the definition of $B(\Q,\Sigma)$, and $I(\Q,\Sigma)$. Roughly speaking, $B(\Q,\Sigma)$ is the collection of all hod pairs which are strict hod initial segments of a $\Sigma$-iterate of $\Q$ and $I(\Q,\Sigma)$ is the collection of all $\Sigma$-iterates of $\Sigma$. In the case $\lambda^\Q$ is limit, $\Gamma(\Q,\Sigma)$ is the collection of $A\subseteq \mathbb{R}$ such that $A$ is Wadge reducible to some $\Psi$ for which there is some $\R$ such that $(\R,\Psi)\in B(\Q,\Sigma)$. See \cite{ATHM} for the definition of $\Gamma(\Q,\Sigma)$ in the case $\lambda^\Q$ is a successor ordinal.

\cite{ATHM} constructs under $\textsf{AD}^+$ and the hypothesis that there are no models of ``$\textsf{AD}_\mathbb{R}+\Theta$ is regular"  hod pairs that are fullness preserving, positional, commuting, and have branch condensation (see \cite{ATHM} for a full discussion of these notions). Such hod pairs are particularly important for our computation as they are points in the direct limit system giving rise to \textrm{HOD} of $\textsf{AD}^+$ models. Under $\sf{AD}^+$, for hod pairs $(\M_\Sigma, \Sigma)$, if $\Sigma$ is a strategy with branch condensation and $\VT$ is a stack on $\M_\Sigma$ with last model $\N$, $\Sigma_{\N, \VT}$ is independent of $\VT$. Therefore, later on we will omit the subscript $\VT$ from $\Sigma_{N, \VT}$ whenever $\Sigma$ is a strategy with branch condensation and $\M_\Sigma$ is a hod mouse. In a core model induction, we don't quite have, at the moment $(\M_\Sigma,\Sigma)$ is constructed, an $\textsf{AD}^+$-model $M$ such that $(\M_\Sigma,\Sigma)\in M$ but we do know that every $(\R,\Lambda)\in B(\M_\Sigma,\Sigma)$ belongs to such a model. We then can show (using our hypothesis) that $(\M_\Sigma,\Sigma)$ belongs to an $\textsf{AD}^+$-model.

\begin{definition}[Hod pair below $\kappa$]\label{hodpairbelowomega2} $(\P, \Sigma)$ is a hod pair below $\kappa$ if $\P\in V$, $|\P|^V\leq \kappa$, $\Sigma$ is a $(\omega_1,\omega_1)$ in $V[G]$\footnote{Technically, this should be a $(k,\omega_1,\omega_1)$-strategy, where $k$ is the degree of soundness of $\P$. But we suppress this parameter throughout our paper.}-strategy with branch condensation, and is commuting, positional, and $\Omega$-fullness preserving, and for all $(\Q,\Lambda)\in B(\Q,\Sigma)$, $\Lambda\rest HC \in \Omega$. Furthermore, $\Sigma\rest V\in V$.
\end{definition}

\subsubsection{$\F$-PREMICE}
\begin{definition}
\label{model}
Let $\mathcal{L}_0$ be the language of set theory expanded by unary predicate
symbols $\dot{E}, \dot{B}, \dot{S}$, and constant symbols $\dot{a}$, 
$\dot{\param}$. Let 
$\Ll_0^-=\Ll_0\cut\{\dot{E},\dot{B}\}$.

Let $a$ be transitive. Let $\varrho:a\to\rank(a)$ be the rank function. 
We write $\ahat=\trancl(\{(a,\varrho)\})$. Let $\param\in\J_1(\ahat)$.

A \textbf{$\J$-structure over $a$ (with parameter $\param$) (for $\Ll_0$)} 
is a structure $\M$ for $\Ll_0$ such that $a^\M=a$, ($\param^\M=\param$), and 
there is
$\lambda\in[1,\Ord)$ such that $|\M|=\J_\lambda^{S^\M}(\ahat)$.

Here we also let $\l(\M)$ denote 
$\lambda$, the \textbf{length} of $\M$, and let $\ahat^\M$ denote $\ahat$.

For $\alpha\in[1,\lambda]$ let $\M_\alpha=\J_\alpha^{S^\M}(\ahat)$.
We say that $\M$ is \textbf{acceptable} iff for each 
$\alpha<\lambda$ and $\tau<\OR(\M_\alpha)$, if 
\[ 
\pow(\tau^{<\om}\cross\ahat^{<\om})\inter\M_\alpha\neq\pow(\tau^{<\om}
\cross\ahat^{<\om})\inter\M_{\alpha+1},\]
then there is a surjection 
$\tau^{<\om}\cross\ahat^{<\om}\to\M_\alpha$ in $\M_{\alpha+1}$.

A \textbf{$\J$-structure (for $\Ll_0$)} is a $\J$-structure over $a$, for some 
$a$.
\end{definition}

As all $\J$-structures we consider will be for $\Ll_0$, we will omit the 
phrase ``for $\Ll_0$''. We also often omit the phrase ``with parameter 
$\param$''. Note that if $\M$ is a $\J$-structure over $a$ then 
$|\M|$ is transtive and rud-closed, $\hat{a}\in M$ and $\OR\inter 
M=\rank(M)$. This last point is because we construct from $\hat{a}$ instead of 
$a$.

$\Fop$-premice will be $\J$-structures of the following form.

\begin{definition}\label{dfn:model}
A \textbf{$\J$-model over $a$ (with parameter $\param$)} 
is an acceptable
$\J$-structure over $a$ (with parameter $\param$), of the form
\[ \M = (M; E, B, S, a, \param) \]
where $\dot{E}^\mathcal{M}=E$, etc, and letting $\lambda=\l(\M)$, the following 
hold.
\begin{enumerate}
 \item $\M$ is amenable.
 \item $S = \langle S_\xi \ | \ \xi\in[1,\lambda) \rangle$
is a sequence of $\J$-models over $a$ (with parameter $\param$).
\item For each $\xi\in[1,\lambda)$, $\dot{S}^{S_\xi} = S\rest \xi$ and $\M_\xi = |S_\xi|$.
%\item If $\lambda$ is a successor ordinal then $\l=\lambda$; otherwise $\l=0$.
\item\label{item:extender} Suppose $E\neq\emptyset$. Then $B=\emptyset$ and 
there is an extender $F$ over $\M$ which is
$\hat{a}\times\gamma$-complete for all $\gamma < \crit(F)$ and such that 
the premouse axioms \cite[Definition 2.2.1]{wilson2012contributions} hold for $(\M,F)$, and
$E$ codes $\tilde{F}\un\{G\}$ where: (i) $\tilde{F}\sub M$ is the 
amenable code for $F$ (as in \cite{steel2010outline}); and (ii) if $F$ is not 
type 2 then $G=\emptyset$, and otherwise $G$ is the ``longest'' non-type Z 
proper segment of $F$ in $\M$.\footnote{We use $G$ explicitly, instead of the 
code $\gamma^\M$ used for $G$ in \cite[Section 2]{FSIT}, because $G$ does not depend 
on which (if there is any) wellorder of $\M$ we use. This ensures that certain 
pure mouse operators are forgetful.}
\qedhere
\end{enumerate}
\end{definition}

Our notion of a ``$\J$-model over $a$" is a bit different from the notion of ``model with parameter $a$" in \cite{CMI} or \cite[Definition 2.1.1]{wilson2012contributions} in that we build into our notion some fine structure and we do not have the predicate $l$ used in \cite[Definition 2.1.1]{wilson2012contributions}. Note that with notation as above, if $\lambda$ is a successor ordinal then 
$M=J(S^\M_{\lambda-1})$, and otherwise, 
$M=\bigcup_{\alpha<\lambda}|S_\alpha|$.
The predicate $\dot{B}$ will be used to code extra information (like a (partial) branch of a tree in $M$).

\begin{definition}
\label{SomeNotations}
Let $\M$ be a $\J$-model over $a$ (with parameter $\param$). Let $E^\M$ denote $\dot{E}^\M$, 
etc. Let $\lambda=\l(\M)$, $S^\M_0=a$, $S^\M_\lambda=\M$, and
$\M|\xi=S^\M_\xi$ for all $\xi\leq\lambda$.
An \textbf{(initial) segment} of $\M$ is just a structure of 
the form $\M|\xi$ for some $\xi\in[1,\lambda]$. We write $\P\ins\M$ iff $\P$ is 
a segment of $\M$, and $\P\pins\M$ iff $\P\ins\M$ and $\P\neq\M$.
Let $\M||\xi$ be the structure having the same universe
and predicates as $\M|\xi$, except that $E^{\M||\xi}=\emptyset$.
%***Is this definition of $\rho(\M)$ assuming that $\M$ is sound?***
We say that $\M$ is \textbf{$E$-active} iff
$E^\M\neq\emptyset$, and \textbf{$B$-active} iff
$B^\M\neq\emptyset$. \textbf{Active} means either
$E$-active or $B$-active; \textbf{$E$-passive} means not
$E$-active; \textbf{$B$-passive} means not $B$-active; and \textbf{passive}
means not active.

Given a $\J$-model $\M_1$ over $b$ and a $\J$-model $\M_2$ over
$\M_1$, we write $\M_2\downarrow b$ for the $\J$-model $\M$ over $b$,
such that $\M$ is ``$\M_1\conc\M_2$''. That is,
$|\M|=|\M_2|$, $a^\M=b$, $E^\M=E^{\M_2}$, $B^\M=B^{\M_2}$,
and $\P\pins\M$ iff $\P\ins\M_1$ or there is $\Q\pins\M_2$ such that 
$\P=\Q\downarrow b$, when such an $\M$ exists. Existence depends on whether the $\J$-structure $\M$ is acceptable.
\end{definition}

In the following, the variable $i$ should be interpreted as 
follows. When $i=0$, we ignore history, and so $\P$ is treated as a coarse 
object when determining $\Fop(0,\P)$. When $i=1$ we respect the history 
(given it exists).

\begin{definition}\label{operatorWithParama}
An \textbf{operator $\Fop$ with domain $D$} is a function with domain 
$D$, such that for some cone $C=C_\Fop$, possibly self-wellordered (sword)\footnote{$C$ is a cone if there are a cardinal $\kappa$ and a transitive set $a\in H_\kappa$ such that $C$ is the set of $b\in H_\kappa$ such that $a\in L_1(b)$; $a$ is called the base of the cone. A set $a$ is self-wellordered if there is a well-ordering of $a$ in $L_1(a)$. A set $C$ is a self-wellordered cone if $C$ is the restriction of a cone $C'$ to its own self-wellordered elements}, $D$ is the set of 
pairs $(i,X)$ such that either:
\begin{itemize}
 \item $i=0$ and $X\in C$, or
 \item $i=1$ and $X$ is a $\J$-model over $X_1\in C$,
\end{itemize}
and for each $(i,X)\in D$, $\Fop(i,X)$ is a $\J$-model over $X$ 
such that for each $\P\ins \Fop(i,X)$, $\P$ is fully 
sound. (Note that $\P$ is a $\J$-model over $X$, so soundness is in this 
sense.)

Let $\Fop,D$ be as above. We say $\Fop$ is \textbf{forgetful} iff 
$\Fop(0,X)=\Fop(1,X)$ whenever $(0,X),(1,X)\in D$, and whenever $X$ is a 
$\J$-model over $X_1$, and $X_1$ is a $\J$-model over $X_2\in C$, we have 
$\Fop(1,X)=\Fop(1,X\downarrow X_2)$. 
Otherwise we say $\Fop$ is \textbf{historical}. Even when $\Fop$ is historical, 
we often just write $\Fop(X)$ instead of $\Fop(i,X)$ when the nature of $\Fop$ is clear from the context.
We say $\Fop$ is \textbf{basic} iff for all 
$(i,X)\in D$ and $\P\ins\Fop(i,X)$, we have $E^\P=\emptyset$.
We say $\Fop$ is \textbf{projecting} iff for all $(i,X)\in D$, we have 
$\rho_\om^{\Fop(i,X)}=X$.
\end{definition}

Here are some illustrations. Strategy 
operators (to be explained in more detail later) are basic, and as usually 
defined, projecting and historical. Suppose we have an iteration strategy $\Sigma$ and we want to build a $\J$-model $\N$ (over some $a$) that codes a fragment of $\Sigma$ via its predicate $\dot{B}$. We feed $\Sigma$ into $\N$ by always providing $b = \Sigma(\T)$, for the $<$-$\N$-least tree $\T$ for which this information is required. So given a reasonably closed level $\P\lhd \N$, the choice of which tree $\T$ should be processed next will usually depend on the information regarding $\Sigma$ already encoded in $\P$ (its history). Using an operator $\Fop$ to build $\N$, then $\Fop(i,\P)$ will be a structure extending $\P$ and over which $b = \Sigma(\T)$ is encoded. The variable $i$ should be interpreted as follows. When $i = 1$, we respect the history of $\P$ when selecting $\T$. When $i = 0$ we ignore history when selecting $\T$ . The operator $\Fop(X)=X^\#$ is forgetful 
and projecting, and not basic; here $\Fop(X) = \Fop(0,X)$.

\begin{definition}
For any $P$ and any ordinal $\alpha\geq 1$, the operator $\Jop_\alpha(\ \cdot\ ;P)$ 
is defined as follows.\footnote{The ``$\mathsf{m}$'' is 
for 
``model''.} For $X$ such that $P\in\J_1(\hat{X})$, let
$\Jop_{\alpha}(X;P)$ be the $\J$-model $\M$ over $X$, with parameter $P$, such 
that
$|\M|=\J_\alpha(\hat{X})$ and for each $\beta\in[1,\alpha]$, $\M|\beta$ is 
passive. Clearly $\Jop_\alpha(\ \cdot\ ;P)$ is basic and forgetful.
If $P=\emptyset$ or we wish to supress $P$, we just 
write $\Jop_\alpha(\ \cdot\ )$.

\end{definition}

\begin{definition}[Potential $\Fop$-premouse, 
$\C_\Fop$]\label{potentialJpremouse}
Let $\Fop$ be an operator with domain $D$ of self-wellordered sets. Let $b\in C_\Fop$, so there is a well-ordering of $b$ in $L_1[b]$. A \textbf{potential
$\Fop$-premouse over $b$} is an acceptable $\J$-model $\M$ over $b$
such that there is an ordinal $\iota>0$ and an increasing, closed sequence
$\langle \zeta_\alpha \rangle_{\alpha\leq \iota}$ of ordinals such that for each
$\alpha\leq \iota$, we have:
\begin{enumerate}
\item $0=\zeta_0\leq\zeta_\alpha\leq\zeta_\iota=l(\M)$ (so $\M|\zeta_0=b$ and
$\M|\zeta_\iota=\M$).
\item If $1<\iota$ then $\M|\zeta_1=\Fop(0,b)$.
\item If $1=\iota$ then $\M\ins\Fop(0,b)$.
\item If $1<\alpha+1<\iota$ then
$\M|\zeta_{\alpha+1}=\Fop(1,\M|\zeta_\alpha)\downarrow b$.
\item If $1<\alpha+1=\iota$, then $\M\ins
\Fop(1,\M|\zeta_\alpha)\downarrow b$.
\item Suppose $\alpha$ is a limit. Then $\M|\zeta_\alpha$ is $B$-passive, and 
if $E$-active, then $\crit(E^{\M|\zeta_\alpha})>\rank(b)$.
\end{enumerate}
We say that $\M$ is \textbf{($\Fop$-)whole} iff $\iota$ is a limit or else,
$\iota=\alpha+1$ and $\M = \Fop(\M|\zeta_\alpha)\downarrow b$.

A \textbf{(potential) $\Fop$-premouse} is a (potential) $\Fop$-premouse over 
$b$, for some $b$.
\end{definition}
\begin{definition}
Let $\Fop$ be an operator and $b\in C_\Fop$. Let $\N$ be a whole
$\Fop$-premouse over $b$. A \textbf{potential continuing $\Fop$-premouse over 
$\N$} is a $\J$-model $\M$ over $\N$ such that $\M\downarrow b$ is a potential 
$\Fop$-premouse over $b$. (Therefore $\N$ is a whole strong cutpoint of $\M$.)

We say that $\M$ (as above) is \textbf{whole} iff $\M\downarrow 
b$ is whole.

A \textbf{(potential) continuing $\Fop$-premouse} is a (potential) continuing 
$\Fop$-premouse over $b$, for some $b$.
\end{definition}

\begin{definition}
\label{Lps}
$\Lp^\Fop(a)$ denotes the stack of all 
countably $\Fop$-iterable $\Fop$-premice $\M$ over $a$ such that $\M$ is fully sound 
and 
projects to $a$.\footnote{Countable substructures of $\M$ are $(\omega,\omega_1+1)$-$\Fop$-iterable, i.e. all iterates are $\Fop$-premice. See \cite[Section 2]{trang2013} for more details on $\Fop$-iterability.}

Let $\N$ be a whole $\Fop$-premouse over $b$, for $b\in C_\Fop$. Then $\Lp^\Fop_+(\N)$ denotes the stack 
of all countably $\Fop$-iterable (above $o(\N)$) continuing $\Fop$-premice $\M$ over $\N$ such 
that $\M\downarrow b$ is fully sound and projects to $\N$.

We say that $\Fop$ is
\textbf{uniformly $\Sigma_1$} iff there are $\Sigma_1$ formulas $\varphi_1$ and 
$\varphi_2$ in $\Ll_0^-$ such that whenever $\M$ is a (continuing) 
$\Fop$-premouse, 
then
the set of whole proper segments of $\M$ is defined over $\M$ by $\varphi_1$ 
($\varphi_2$). For 
such an operator $\Fop$, let $\varphi^\Fop_\wh$ denote the least such 
$\varphi_1$.
\end{definition}
\begin{definition}[Mouse operator]\label{MouseOperator}
Let $Y$ be a projecting, uniformly $\Sigma_1$ operator.
A \textbf{$Y$-mouse operator $\Fop$ with domain $D$}
is an operator with domain $D$ such for each $(0,X)\in D$, 
$\Fop(0,X)\pins\Lp^Y(X)$, and for each $(1,X)\in 
D$, $\Fop(1,X)\pins\Lp^Y_+(X)$.\footnote{This restricts the usual notion defined in \cite{CMI}.} (So any $Y$-mouse 
operator is an operator.) A $Y$-mouse operator $\Fop$ is called \textbf{first order} if there are formulas $\varphi_1$ and $\varphi_2$ in the language of $Y$-premice such that $\Fop(0,X)$ ($\Fop(1,X)$) is the first $\M\lhd \textrm{Lp}^Y(X)$ ($\textrm{Lp}^Y_+(X)$) satisfying $\varphi_1$ ($\varphi_2$).

A \textbf{mouse operator} is a $\Jop_1$-mouse operator.
\end{definition}

We can then define $\Fop$-solidity, the $L^{\Fop}[\es]$-construction etc. as usual (see \cite{trang2013} for more details). We now define the kind of condensation that mouse operators need to satisfy to ensure the $L^{\Fop}[\es]$ converges.

\begin{definition}
Let $\M_1,\M_2$ be $k$-sound $\J$-models over $a_1,a_2$ and 
let $\pi:\M_1\to\M_2$. Then $\pi$
is \textbf{(weakly, nearly) $k$-good} iff $\pi\rest a_1=\id$, $\pi(a_1)=a_2$,\ and $\pi$ is 
a (weak, near)
$k$-embedding (as in \cite{FSIT}).\end{definition}

\begin{definition}\label{DropDownSequence}
Given a $\J$-model $\N$ over $a$, and $\M\pins\N$ such that $\M$ is fully 
sound, the 
\textbf{$\M$-drop-down sequence} of $\N$ is the sequence of
pairs $\left<(\Q_n,m_n)\right>_{n<k}$ of maximal length such that $\Q_0=\M$ and 
$m_0=\om$ and for each $n+1<k$:
\begin{enumerate}
 \item $\M\pins\Q_{n+1}\ins\N$ and $\Q_n\ins\Q_{n+1}$,
 \item every proper segment of $\Q_{n+1}$ is fully sound,
 \item $\rho_{m_n}(\Q_n)$ is an $a$-cardinal of $\Q_{n+1}$,
 \item $0<m_{n+1}<\om$,
 \item $\Q_{n+1}$ is $(m_{n+1}-1)$-sound,
 \item 
$\rho_{m_{n+1}}(\Q_{n+1})<\rho_{m_n}(\Q_n)\leq\rho_{m_{n+1}-1}(\Q_{n+1})$. 
 \qedhere
\end{enumerate}
\end{definition}

\begin{definition}\label{dfn:Fop_con}
Let $\Fop$ be an operator and let $C$ be some class of $E$-active 
$\Fop$-premice. Let $b$ 
be transitive. A \textbf{($C$-certified) $L^\Fop[\es,b]$-construction} is a 
sequence $\left<\N_\alpha\right>_{\alpha\leq\lambda}$ with the 
following properties. We omit the phrase ``over $b$''.

We have $\N_0=b$ and $\N_1=\Fop(0,b)$.

Let $\alpha\in(0,\lambda]$. Then $\N_\alpha$ is an $\Fop$-premouse, and 
if $\alpha$ is a limit then $\N_\alpha$ is the lim inf of the $\N_\beta$ 
for $\beta<\alpha$. Now suppose 
that $\alpha<\lambda$.
Then either:
\begin{itemize}
\item $\N_\alpha$ is passive and is a limit of whole proper segments and
$\N_{\alpha+1}=(\N_\alpha,G)$ for some extender $G$ (with $\N_{\alpha+1}\in 
C$); or
\item\label{item:apply_Fop} $\N_\alpha$ is $\om$-$\Fop$-solid. Let $\M_\alpha=\core_\om(\N_\alpha)$.
Let $\M$ be the largest whole segment of $\M_\alpha$. So either $\M_\alpha=\M$ or 
$\M_\alpha\downarrow\M\ins\Fop_1(\M)$.
Let $\N\ins\Fop_1(\M)$ be least such that either $\N=\Fop_1(\M)$ or for some $k+1<\om$, 
$(\N\downarrow b,k+1)$ is 
on the $\M_\alpha$-drop-down sequence of $\N\downarrow b$. Then $\N_{\alpha+1}=\N\downarrow b$. (Note 
$\M_\alpha\pins\N_{\alpha+1}$.)\qedhere
\end{itemize}
\end{definition}

\begin{definition}\label{dfn:condenses_coarsely}
 Let $Y$ be an operator. We say that $Y$ 
\textbf{condenses coarsely} iff 
for all $i\in\{0,1\}$ and $(i,\bar{X}),(i,X)\in\dom(Y)$, and all $\J$-models 
$\M^+$ over $\bar{X}$, if $\pi:\M^+\to Y_i(X)$ is fully elementary, fixes the parameters in the definition of $Y$, then 
\begin{enumerate}
 \item if $i=0$ then $\M^+\ins Y_0(\bar{X})$; and
 \item if $i=1$ and $X$ is a sound whole $Y$-premouse, then $\M^+\ins 
Y_1(\bar{X})$.\qedhere
\end{enumerate}
\end{definition}

\begin{definition}\label{dfn:condenses_finely}
Let $Y$ be a projecting, uniformly $\Sigma_1$ operator. We say that $Y$ 
\textbf{condenses finely} iff $Y$ condenses 
coarsely and we have the following. Let $k<\om$. Let $\M^*$ 
be a $Y$-premouse over $a$, with a largest whole proper segment $\M$, 
such that $\M^+=\M^*\downarrow\M$ is sound and $\rho_{k+1}(\M^+)=\M$. Let 
$\P^*,\abar,\P,\P^+$ be likewise. Let $\N$ be a sound whole $Y$-premouse over 
$\abar$. Let
$G\sub\Coll(\om,\P\un\N)$ be $V$-generic. Let $\N^+,\pi,\sigma\in V[G]$,
with $\N^+$ a sound $\J$-model over $\N$ such that
$\N^*=\N^+\downarrow\abar$ is defined (i.e. acceptable). Suppose $\pi:\N^*\to\M^*$ is such that
$\pi(\N)=\M$ and
either:
\begin{enumerate}
 \item $\M^*$ is $k$-sound and $\N^*=\core_{k+1}(\M^*)$; or
 \item $(\N^*,k+1)$ is in the 
$\N$-dropdown sequence of $\N^*$, and likewise $(\P^*,k+1),\P$, and 
either:
\begin{enumerate}
\item\label{item:condense_k-good} $\pi$ is $k$-good, or

\item\label{item:condense_fully} $\pi$ is fully elementary, or

\item\label{item:condense_realize} $\pi$ is a weak $k$-embedding, 
$\sigma:\P^*\to\N^*$ is 
$k$-good, $\sigma(\P)=\N$ and $\pi\com\sigma\in V$ is a near $k$-embedding.
\end{enumerate}
\end{enumerate}
Then $\N^+\ins Y_1(\N)$.

We say that $Y$ \textbf{almost condenses finely} iff $\N^+\ins Y_1(\N)$ whenever the hypotheses 
above hold with $\N^+,\pi,\sigma\in V$.
\end{definition}

In fact, the two notions above are equivalent.

\begin{lemma}\label{lem:almost_condenses_finely}
Let $Y$ be an operator on a cone with base in $\HC$. Suppose that $Y$ almost condenses 
finely. Then $Y$ condenses finely.
\end{lemma}

We end this section with the following lemma (proved in Section 2 of \cite{trang2013}), which states that the $L^{\Fop}[\es]$-construction (relative to some class of background extenders) runs smoothly for a certain class of operators. In the following, if $(\N,G)\in C$, then $G$ is backgrounded as in \cite{FSIT} or as in \cite{CMIP} (we additionally demand that the structure $N$ in \cite[Definition 1.1]{CMIP} is closed under $\Fop$).

\begin{lemma}\label{lem:le_converges}
Let $\Fop$ be a projecting, uniformly $\Sigma_1$ operator which condenses 
finely. Suppose $\Fop$ is defined on a cone with bases in $\rm{HC}$. Let $\CC=\left<\N_\alpha\right>_{\alpha\leq\lambda}$ be the ($C$-certified) 
$L^\Fop[\es,b]$-construction for $b\in C_\Fop$.
Then (a) $\N_\lambda$ is $0$-$\Fop$-solid (i.e., is an 
$\Fop$-premouse).

Now suppose 
that $\N_\lambda$ is $k$-$\Fop$-solid.

Suppose that for a club of countable elementary $\pi:\M\to\core_k(\N_\lambda)$, 
there is an $\Fop$-putative, $(k,\om_1,\om_1+1)$-iteration 
strategy $\Sigma$ for $\M$, such that every tree $\Tt$ via 
$\Sigma$ is $(\pi,\CC)$-realizable.\footnote{See \cite[Section 2]{trang2013} for a precise definition of $(\pi,\CC)$-realizability. Roughly speaking this means that models along the tree $\T$ are embedded into the $\N_\alpha$'s.}

Then (b) $\N_\lambda$ is $(k+1)$-$\Fop$-solid.
\end{lemma}

\begin{lemma}\label{lem:strong_condenses_coarsely}
Let $Y,\Fop$ be uniformly $\Sigma_1$ operators defined on a cone over some $\her_\kappa$, with bases in $\HC$.\footnote{We also say ``operator over $\her_\kappa$ with bases in $\HC$" for short.} Suppose that $Y$ 
condenses finely. Suppose that $\Fop$ is a whole continuing $Y$-mouse operator. Then  
$\Fop$ 
condenses finely. 
\end{lemma}

The following lemma gives a stronger condensation property than fine condensation in certain circumstances. So if $\Fop$ satisfies the hypothesis of Lemma \ref{lem:strong_condenses_coarsely} (particularly, if $\Fop$ is one of the operators constructed in our core model induction) then the $L^\Fop[\es]$-construction converges by Lemma \ref{lem:le_converges}. 

\begin{lemma}\label{lem:strong_condenses_coarsely}
Let $Y,\Fop$ be uniformly $\Sigma_1$ operators with bases in $\HC$. Suppose that $Y$ 
condenses finely. Suppose that $\Fop$ is a whole continuing $Y$-mouse operator. Then (a) 
$\Fop$ 
condenses finely. Moreover, (b) let 
$\M$ be an $\Fop$-whole $\Fop$-premouse. Let $\pi:\N\to\M$ be fully elementary with $a^\N\in 
C_\Fop$. Then $\N$ is an 
$\Fop$-whole 
$\Fop$-premouse. So regarding $\Fop$, the conclusion of \ref{dfn:condenses_coarsely} may be 
modified by 
replacing ``$\ins$'' with ``$=$''.
\end{lemma}

\begin{remark}
In the context of the core model induction of this paper (and elsewhere), we often construct mouse operators $\Fop$ defined over some $\her_\kappa$ with base $a \notin \HC$. So given an $\Fop$-premouse $\N$, $\pi:\N^*\rightarrow \N$ elementary, and $\N^*$ countable, $\N^*$ may not be an $\Fop$ premouse. We have to make some changes for the theory above to work for these $\Fop$. For instance, in Lemma \ref{lem:le_converges}, with the notation as there, we can modify the hypothesis of the lemma in one of two ways:
\begin{enumerate}
\item We can either require that $a\in \M$, $|\M| = |a|$, and the $(\pi,\CC)$-realizable strategy $\Sigma$ is $(k,|a|^+,|a|^+ + 1)$-iterable.
\item We can still require $\M$ is countable but the strategy $\Sigma$ is a $(k,\omega_1,\omega_1+1)$-$\Fop^\pi$-strategy, where $\Fop^\pi$ is the $\pi$-pullback operator of $\Fop$.\footnote{For instance, if $\Fop$ corresponds to a strategy $\Sigma$, then $\Fop^\pi$ corresponds to $\Sigma^\pi$, the $\pi$-pullback of $\Sigma$. If $\Fop$ is a first order mouse operator defined by ($\varphi, a$), then $\Fop^\pi$ is defined by ($\varphi,\pi^{-1}(a)$).}
\end{enumerate}
\end{remark}

\subsubsection{STRATEGY PREMICE}
We now proceed to defining $\Sigma$-premice, for an iteration strategy
$\Sigma$. We first define the operator to be used to feed in $\Sigma$.

\begin{definition}[$\BBB(a,\Tt,b)$, $b^\N$]\label{dfn:B(M,T,b)}
Let $a,\P$ be transitive, with $\P\in\J_1(\ahat)$. Let
$\lambda>0$ and let
$\Tt$ be an iteration tree\footnote{We formally take an \emph{iteration tree} to include
the entire sequence $\left<M^\Tt_\alpha\right>_{\alpha<\lh(\Tt)}$ of models. So it is 
$\Sigma_0(\Tt,\param)$ to assert that ``$\Tt$ is an iteration tree on
$\param$''.} on $\P$, of length $\om\lambda$, with $\Tt\rest\beta\in a$ for 
all $\beta\leq\om\lambda$. Let $b\sub\om\lambda$. We
define
$\N=\BBB(a,\Tt,b)$ recursively on $\lh(\Tt)$, as the $\J$-model $\N$ over
$a$, with parameter $\P$,\footnote{$\P=M^\Tt_0$ is determined by $\Tt$.} such that:
\begin{enumerate}
 \item $\l(\N)=\lambda$,
 \item for each $\gamma\in(0,\lambda)$,
$\N|\gamma=\BBB(a,\Tt\rest\om\gamma,[0,\om\gamma]_\Tt)$,
 \item $B^\N$ is the set of ordinals $\OR(a)+\gamma$ such that $\gamma\in
b$,
\item $E^\N=\emptyset$.
\end{enumerate}
We also write $b^\N=b$.
\end{definition}

It is easy to see that every initial segment of $\N$ is sound, so $\N$ is 
acceptable and is indeed a $\J$-model (not just a $\J$-structure). 

In the context of a $\Sigma$-premouse $\M$ for an iteration strategy $\Sigma$,
if $\Tt$ is the $<_\M$-least tree for which $\M$ lacks instruction regarding
$\Sigma(\Tt)$,
then $\M$ will already have been instructed regarding $\Sigma(\Tt\rest\alpha)$
for all
$\alpha<\lh(\Tt)$. Therefore if $\lh(\Tt)>\om$ then
$\BBB(\M,\Tt,\Sigma(\Tt))$ codes redundant information (the branches already in
$\Tt$) before coding $\Sigma(\Tt)$. This redundancy seems to allow one to
prove slightly stronger condensation properties, given that $\Sigma$ has nice
condensation properties (see Lemma \ref{StategyCondensation}). It also
simplifies the definition.

\begin{definition}
 Let $\Sigma$ be a partial iteration strategy. Let
$C$ be a class of iteration trees, closed under
initial segment. We say that $(\Sigma,C)$ is \textbf{suitably condensing} iff
for every $\Tt\in C$ such that $\Tt$ is via $\Sigma$ and $\lh(\Tt)=\lambda+1$
for some limit $\lambda$, either (i) $\Sigma$ has hull condensation with respect
to $\Tt$, or (ii) $b^\Tt$ does not drop and $\Sigma$ has branch
condensation with respect to $\Tt$, that is, any hull $\U^\smallfrown c$ of $\T^\smallfrown b$ is according to $\Sigma$.
\end{definition}

When $C$ is the class of all iteration trees according to $\Sigma$, we simply omit it from our notation. 

\begin{definition}\label{dfn:J_premice}
 Let $\varphi$ be an $\Ll_0$-formula. Let 
$\P$
be transitive. Let $\M$ be a $\J$-model (over some $a$), with parameter $\P$. 
Let $\Tt\in\M$. We
say that $\varphi$ \textbf{selects $\Tt$ for $\M$}, and write 
$\Tt=\Tt^\M_\varphi$, iff
\begin{enumerate}[(a)]
 \item $\Tt$ is the unique 
$x\in\M$ such that $\M\sats\varphi(x)$,
 \item $\Tt$ is an iteration tree on $\P$ of
limit length,
 \item for every $\N\pins\M$, we have $\N\not\sats\varphi(\Tt)$, and
 \item for every limit $\lambda<\lh(\Tt)$, there is $\N\pins\M$ such that
$\N\sats\varphi(\Tt\rest\lambda)$.\qedhere
\end{enumerate}
\end{definition}

One instance of $\phi(\P,\T)$ is, in the case $a$ is self-wellordered, the formula ``$\T$ is the least tree on $\P$ that doesn't have a cofinal branch", where least is computed with respect to the canonical well-order of the model.

\begin{definition}[Potential $\P$-strategy-premouse, $\Sigma^\M$]
\label{PotStrPremouse}
Let $\varphi\in\Ll_0$. Let $\P,a$ be transitive with $\P\in\J_1(\ahat)$. A
\textbf{potential $\P$-strategy-premouse (over $a$, of type
$\varphi$)} is a $\J$-model $\M$ over $a$, with parameter $\P$, 
such that
the $\BBB$ operator is used to 
feed in
an iteration strategy for trees on $\P$, using the sequence of trees
naturally determined by $S^\M$ and selection by $\varphi$. We let $\Sigma^\M$ 
denote the
partial strategy coded by the predicates $B^{\M|\eta}$, for
$\eta\leq\l(\M)$.

In more detail, there is an
increasing, closed sequence of ordinals
$\left<\eta_\alpha\right>_{\alpha\leq\iota}$ with the following properties.
We will also define $\Sigma^{\M|\eta}$ for all $\eta\in[1,\l(\M)]$ and
$\Tt_\eta=\Tt^\M_\eta$ for
all $\eta\in[1,\l(\M))$.\begin{enumerate}
 \item $1=\eta_0$ and 
$\M|1=\Jop_1(a;\P)$ and 
$\Sigma^{\M|1}=\emptyset$.
 \item $\l(\M)=\eta_\iota$, so $\M|\eta_\iota=\M$.
 \item Given $\eta\leq\l(\M)$ such that $B^{\M|\eta}=\emptyset$, we set
$\Sigma^{\M|\eta}=\bigcup_{\eta'<\eta}\Sigma^{\M|\eta'}$.
\end{enumerate}

Let $\eta\in[1,\l(\M)]$. Suppose
there is $\gamma\in[1,\eta]$ and $\Tt\in\M|\gamma$ such that
$\Tt=\Tt^{\M|\gamma}_\varphi$, and $\Tt$ is via 
$\Sigma^{\M|\eta}$, but no
proper extension of $\Tt$ is via
$\Sigma^{\M|\eta}$. Taking $\gamma$ minimal such, let
$\Tt_\eta=\Tt^{\M|\gamma}_\varphi$.
Otherwise let $\Tt_\eta=\emptyset$.
\begin{enumerate}\setcounter{enumi}{3}
 \item Let $\alpha+1\leq\iota$. Suppose $\Tt_{\eta_\alpha}=\emptyset$.
Then $\eta_{\alpha+1}=\eta_\alpha+1$ and
$\M|\eta_{\alpha+1}=\Jop_1(\M|\eta_\alpha;\P)\downarrow a$.
\item Let $\alpha+1\leq\iota$.
Suppose $\Tt=\Tt_{\eta_\alpha}\neq\emptyset$.
Let $\om\lambda=\lh(\Tt)$.
Then for some $b\sub\om\lambda$, and
$\mathcal{S}=\BBB(\M|\eta_\alpha,\Tt,b)$, we have:
\begin{enumerate}
\item $\M|\eta_{\alpha+1}\ins\mathcal{S}$.
\item If $\alpha+1<\iota$ then $\M|\eta_{\alpha+1}=\mathcal{S}$.
\item If $\mathcal{S}\ins\M$ then $b$ is a $\Tt$-cofinal branch.\footnote{We allow
$\M^\Tt_b$ to be illfounded, but then
$\Tt\conc b$ is not an iteration tree, so is not continued by $\Sigma^\M$.}
\item For $\eta\in[\eta_\alpha,\l(\M)]$ such that
$\eta<\l(\mathcal{S})$,
$\Sigma^{\M|\eta}=\Sigma^{\M|\eta_\alpha}$.
\item If $\mathcal{S}\ins\M$ then
then $\Sigma^{\mathcal{S}}=\Sigma^{\M|\eta_\alpha}\un\{(\Tt,b^\mathcal{S})\}$.
\end{enumerate}
\item For each limit $\alpha\leq\iota$, $B^{\M|\eta_\alpha}=\emptyset$.\qedhere
\end{enumerate}
\end{definition}

\begin{definition}[Whole]\label{dfn:whole_strategy}
Let $\M$ be a potential $\P$-strategy-premouse of type $\varphi$. We say $\P$
is \textbf{$\varphi$-whole} (or just \textbf{whole} if $\varphi$ is fixed) iff
for every $\eta<\l(\M)$, if
$\Tt_{\eta}\neq\emptyset$ and $\Tt_\eta\neq\Tt_{\eta'}$ for all
$\eta'<\eta$, then for some $b$,
$\BBB(\M|\eta,\Tt_{\eta},b)\ins\M$.\footnote{\emph{$\varphi$-whole} depends on
$\varphi$ as the definition of $\Tt_\eta$ does.}
\end{definition}

\begin{definition}[Potential $\Sigma$-premouse]
Let $\Sigma$ be a (partial) iteration strategy for a transitive
structure $\P$.
A \textbf{potential $\Sigma$-premouse (over $a$, of type $\varphi$)}
is a
potential
$\P$-strategy premouse $\M$ (over $a$, of type $\varphi$) such that
$\Sigma^\M\sub\Sigma$.\footnote{If $\M$ is a model all of whose proper segments
are potential $\Sigma$-premice, and the rules for potential $\P$-strategy
premice
require that $B^\M$ code a $\Tt$-cofinal branch, but $\Sigma(\Tt)$ is not
defined,
then $\M$ is not a potential $\Sigma$-premouse, whatever its predicates are.}
\end{definition}

\begin{definition}
 Let $\R,\M$ be $\J$-structures for $\Ll_0$, $a=a^\R$ and $b=a^\M$. 
Suppose that $a,b$ code $\P,\Q$ respectively. Let $\pi:\R\to\M$ (or
\[ \pi:\OR(\R)\un\P\un\{\P\}\to\OR(\M)\un\Q\un\{\Q\} \]
respectively). Then $\pi$ is a
\textbf{$(\P,\Q)$-weak $0$-embedding} (resp., \textbf{$(\P,\Q)$-very weak 
$0$-embedding}) iff
$\pi(\P)=(\Q)$ and with respect to the language 
$\Ll_0$, $\pi$ is
$\Sigma_0$-elementary, and there is an $X\sub\R$ (resp., $X\sub\OR(\R)$) such 
that $X$ is cofinal in $\in^\R$ and 
$\pi$ is
$\Sigma_1$-elementary on parameters in $X\un\P\un\{\P\}$. If also $\P=\Q$ and 
$\pi\rest\P\un\{\P\}=\id$,
then we just say that $\pi$ is a \textbf{$\P$-weak $0$-embedding} (resp., 
\textbf{$\P$-very weak $0$-embedding}).
\end{definition}

Note that, for $(\P,\Q)$-weak $0$-embeddings, we can in fact take 
$X\sub\OR(\R)$. The following lemma is again proved in \cite[Section 3]{trang2013}.

\begin{lemma}\label{StategyCondensation}
 Let $\M$ be a $\P$-strategy premouse over $a$, of type $\varphi$, where $\varphi$ is $\Sigma_1$. Let
$\R$
be a $\J$-structure for $\Ll_0$ and $a'=a^\R$, and let $\P'$ be a 
transitive structure coded
by $a'$.

\begin{enumerate}
\item\label{item:piRtoM} Suppose $\pi:\R\to\M$ is a partial map such that 
$\pi(\P')=\P$ and either:
\begin{enumerate}[(a)]
 \item\label{item:piRtoMP'Pweak} $\pi$ is a $(\P',\P)$-weak $0$-embedding, or
 \item\label{item:piRtoMveryweak} $\pi$ is a $(\P',\P)$-very weak 
$0$-embedding, and if $E^\R\neq\emptyset$ then item \ref{item:extender} of 
\ref{dfn:model} holds for 
$E^\R$.
\end{enumerate}
Then $\R$ is a $\P'$-strategy premouse of type $\varphi$.
Moreover, if $\pi\rest\{\P'\}\un\P'=\id$ and
if $\M$ is a $\Sigma$-premouse, where
$(\Sigma,\dom(\Sigma^\M))$ is suitably condensing, then
$\R$ is also a $\Sigma$-premouse.

\item\label{item:piMtoRgen} Suppose there is $\pi:\M\to\R$ is such that 
$\pi(a,\P)=(a',\P')$ and
either
\begin{enumerate}[(a)]
\item $\pi$ is $\Sigma_2$-elementary; or
\item $\pi$ is cofinal and $\Sigma_1$-elementary, and $B^\M=\emptyset$.
\end{enumerate}
Then $\R$ is a $\P'$-strategy premouse of type $\varphi$, and $\R$ is whole iff
$\M$
is whole.

\item\label{item:piMtoRBactive} Suppose $B^\M\neq\emptyset$. Let 
$\Tt=\Tt^\M_\eta$ where
$\eta<\l(\M)$ is largest such that $\M|\eta$ is whole. Let $b=b^\M$ and
$\om\gamma=\bigcup b$. So $\M\ins\BBB(\M|\eta,\Tt,b)$.
Suppose there is $\pi:\M\to\R$ such that
$\pi(\P)=\P'$ and $\pi$ is cofinal and
$\Sigma_1$-elementary. Let $\om\gamma'=\sup\pi``\om\gamma$.
\begin{enumerate}[(a)]
\item\label{item:piMtoRBactive(a)} $\R$ is a $\P'$-strategy premouse of type 
$\varphi$ iff we have either
(i) $\om\gamma'=\lh(\pi(\Tt))$,
or (ii) $\om\gamma'<\lh(\pi(\Tt))$ and 
$b^\R=[0,\om\gamma']_{\pi(\Tt)}$.
\item\label{item:piMtoRBactive(b)} If either $b^\M\in\M$ or $\pi$ is continuous 
at $\lh(\Tt)$
then $\R$ is a $\P'$-strategy premouse of type $\varphi$.
\end{enumerate}
\end{enumerate}
\end{lemma}

\begin{remark}
The preceding lemma left open the possibility that $\R$ fails to be a
$\P$-strategy premouse under certain circumstances (because $B^\R$ should
be coding a branch that has in fact already been coded at some
proper segment of $\R$, but codes some other branch instead). In the main
circumstance we are interested in, this does not arise, for a couple of reasons.
Suppose that $\Sigma$ is an iteration
strategy for $\P$ with hull condensation, $\M$ is a $\Sigma$-premouse, and
$\Lambda$ is a $\Sigma$-strategy for $\M$. Suppose $\pi:\M\to\R$ is a degree $0$
iteration embedding and $B^\M\neq\emptyset$ and $\pi$ is discontinuous at
$\lh(\Tt)$. Then \cite[Section 3]{trang2013} shows that $b^\M\in\M$. (It's not relevant whether $\pi$
itself is via $\Lambda$.) It then follows from \ref{item:piMtoRBactive(b)} of Lemma \ref{StategyCondensation} that $\R$ is a $\Sigma$-mouse.

The other reason is that, supposing $\pi:\M\to\R$ is via $\Lambda$ (so $\pi\rest \P\cup\{\P\} = \rm{id}$), then
trivially, $B^\R$ must code branches according to $\Sigma$. We
can obtain such a $\Lambda$ given that we can realize iterates of $\M$ back
into a fixed $\Sigma$-premouse (with $\P$-weak $0$-embeddings as realization
maps).
\end{remark}

\begin{definition}\label{dfn:strategy_op}
Let $\P$ be transitive and $\Sigma$ a partial iteration strategy for $\P$.
Let $\varphi\in\Ll_0$. Let $\Fop=\Fop_{\Sigma,\varphi}$ be 
the operator such that:
\begin{enumerate}
 \item $\Fop_0(a)=\Jop_1(a;\P)$, for all transitive 
$a$ such that $\P\in\J_1(\ahat)$;
 \item Let $\M$ be a sound branch-whole $\Sigma$-premouse of 
type $\varphi$. Let $\lambda=\l(\M)$ and with notation as in \ref{PotStrPremouse}, let 
$\Tt=\Tt_\lambda$. If $\Tt=\emptyset$ then $\Fop_1(\M)=\Jop_1(\M;\P)$. If 
$\Tt\neq\emptyset$ then $\Fop_1(\M)=\BBB(\M,\Tt,b)$ where $b=\Sigma(\Tt)$.
\end{enumerate}
We say that $\Fop$ is a \textbf{strategy operator}.
\end{definition}

\begin{lemma}\label{lem:FSigma_condenses}
Let $\P$ be countable and transitive. Let $\varphi$ be a formula of $\Ll_0$.
Let $\Sigma$ be a partial strategy for $\P$. Let $D_\varphi$ be the class of 
iteration trees $\Tt$ on $\P$ such that for some $\J$-model $\M$, with 
parameter $\P$, we have $\Tt=\Tt^\M_\varphi$. Suppose that $(\Sigma,D_\varphi)$ 
is suitably condensing. Then 
$\Fop_{\Sigma,\varphi}$ is uniformly $\Sigma_1$, projecting, and condenses finely.
\end{lemma}

\begin{definition}\label{dfn:MFsharp}
Let $a$ be transitive and let $\Fop$ be an operator. 
We say that 
\textbf{$\M_1^{\Fop,\#}(a)$ exists}
iff there is a $(0,|a|,|a|+1)$-$\Fop$-iterable, non-$1$-small 
$\Fop$-premouse over $a$. We write 
$\M_1^{\Fop,\#}(a)$ for the least such 
sound structure. For $\Sigma,\P,a,\varphi$ as in \ref{dfn:strategy_op}, we 
write 
$\M_1^{\Sigma,\varphi,\#}(a)$ for $\M_1^{\Fop_{\Sigma,\varphi},\#}(a)$.

Let $\Ll_0^+$ be the language $\Ll_0\un\{\dot{\prec},\dot{\Sigma}\}$, 
where $\dot{\prec}$ is the binary relation defined by ``$\dot{a}$ is 
self-wellordered, with ordering $\prec_{\dot{a}}$, and $\dot{\prec}$ is the 
canonical wellorder of the universe extending $\prec_{\dot{a}}$'', and 
$\dot{\Sigma}$ is the partial function defined ``$\dot{\param}$ is a transitive 
structure and the universe is a potential $\dot{\param}$-strategy premouse over 
$\dot{a}$ and $\dot{\Sigma}$ is the associated partial putative iteration 
strategy for $\dot{\param}$''. Let $\varphi_\textrm{all}(\Tt)$ be the 
$\Ll_0$-formula 
``$\Tt$ is the $\dot{\prec}$-least limit length iteration tree $\Uu$ on 
$\dot{\param}$ such that $\Uu$ is via $\dot{\Sigma}$, but no proper extension 
of $\Uu$ is via $\dot{\Sigma}$''. Then for $\Sigma,\P,a$ as in 
\ref{dfn:strategy_op}, we sometimes write $\M_1^{\Sigma,\#}(a)$ for 
$\M_1^{\Fop_{\Sigma,\varphi_{\textrm{all}}},\#}(a)$.

Let $\kappa$ be a cardinal and suppose that $\MFsharp=\M_1^{\Fop,\#}(a)$ 
exists and is $(0,\kappa^++1)$-iterable.
We write $\Lambda_\MFsharp$ for the unique $(0,\kappa^++1)$-iteration strategy 
for $\MFsharp$ (given that $\kappa$ is fixed).
\end{definition}

\begin{definition}\label{dfn:suitable}
We say that $(\Fop,\Sigma,\varphi,D,a,\param)$ is 
\textbf{suitable} iff $a$ is transitive and 
$\M_1^{\Fop,\#}(a)$ exists, where either
\begin{enumerate}
 \item\label{item:mouse_op} $\Fop$ is a projecting, uniformly $\Sigma_1$ operator, $C_\Fop$ is the 
(possibly swo'd) cone above $a$, $D$ is the set of pairs $(i,X)\in\dom(\Fop)$ such 
that either $i=0$ or $X$ is a sound whole $\Fop$-premouse,  and 
$\Sigma=\varphi=0$, or
 \item $\P,\Sigma,\varphi,D_\varphi$ are as in \ref{lem:FSigma_condenses}, $\Fop=\Fop_{\Sigma,\varphi}$, $D_\varphi\sub D$, $D$ is a class of limit length 
iteration trees on $\P$, via $\Sigma$, $\Sigma(\Tt)$ is defined for 
all $\Tt\in D$, $(\Sigma,D)$ is suitably condensing and $\P\in\J_1(\ahat)$.
\end{enumerate}
We write 
$\G_\Fop$ for the function 
with domain $C_\Fop$, such that for all $x \in C_\Fop$, $\G_\Fop(x)=\Sigma(x)$ in case (ii), and in case (i), 
$\G_\Fop(0,x)=\Fop(0,x)$ and $\G_\Fop(1,x)$ is the least $\R\ins\Fop_1(x)\downarrow a^x$ such that 
either $\R=\Fop_1(X)\downarrow a^X$ or $\R$ is unsound. 

\end{definition}

\begin{lemma}\label{lem:Sigma_N_condensation}
Let $\Fop$ be as in \ref{dfn:suitable} and $\MFsharp=\M_1^{\Fop,\#}$. Then 
$\Lambda_\MFsharp$ has branch condensation and hull condensation.
\end{lemma}

\subsubsection{G-ORGANIZED $\mathcal{F}$-PREMICE}\label{g_organized}
Now we give an outline of the general treatment of \cite{trang2013} on $\Fop$-premice over an arbitrary set; following the terminology of \cite{trang2013}, we will call these $g$-organized $\Fop$-premice and $\Theta$-$g$-organized $\Fop$-premice. For $(\Theta)$-$g$-organized $\Fop$-premice to be useful,
we need to assume that the following absoluteness property holds of the operator
$\Fop$. We then show that if $\Fop$ is the operator for a nice enough iteration
strategy, then it does hold. We write $\MFsharp$ for $\MFsharp_\Fop$ and fix $a, \param, \Fop, \P, C$ as in the previous subsection. In the following, $\delta^\MFsharp$ denotes the 
Woodin cardinal of $\MFsharp$. Again, the reader should see \cite{trang2013} for proofs of lemmas stated here.

\begin{definition}\label{dfn:determines}
Let $(\Fop,\Sigma,\varphi,C,a,\param)$ be suitable.
We say that $\M_1^{\F,\sharp}(a)$ \textbf{generically interprets $\F$}\footnote{In \cite{trang2013}, this notion is called $\F$ determines itself on generic extensions. In this paper, ``determines itself on generic extensions" will have a different meaning, as defined later.} iff, writing 
$\MFsharp=\M_1^{\Fop,\#}(a)$, there
are formulas $\Phi,\Psi$ in
$\Ll_0$ such that there is some $\gamma > \delta^\MFsharp$ such that 
$\MFsharp|\gamma
\vDash\Phi$ and for any non-dropping
$\Sigma_\MFsharp$-iterate $\N$ of
$\MFsharp$, via a countable iteration tree $\Tt$, any $\N$-cardinal $\delta$, 
any
$\gamma\in\Ord$ such that $\N|\gamma\models\Phi\ \&\ $``$\delta$ is
Woodin'', and any $g$ which is set-generic over $\N|\gamma$ (with $g\in V$),
then $(\N|\gamma)[g]$
is closed under $\G_\Fop$, and $\G_\Fop\rest(\N|\gamma)[g]$ is defined 
over
$(\N|\gamma)[g]$ by
$\Psi$. We say such a pair $(\Phi,\Psi)$ \textbf{generically determines 
$(\Fop,\Sigma,\varphi,C,a)$} (or just $\Fop$).

We say an operator $\Fop$ is \textbf{nice} iff for some 
$\Sigma,\varphi,C,a,\param$, 
$(\Fop,\Sigma,\varphi,C,a,\param)$ is suitable and $\M_1^{\F,\sharp}$ generically interprets $\Fop$.

Let $\P\in\HC$, let $\Sigma$ be an iteration strategy for $\P$ and let $C$ be 
the class of all limit length trees via $\Sigma$. Suppose $\M_1^{\Sigma,\#}(\P)$ 
exists, $(\Sigma,C)$ is suitably condensing. We say that $\M_1^{\Sigma,\#}(\P)$ 
\textbf{generically interprets $\Sigma$} iff  some $(\Phi,\Psi)$ generically determines
$(\Fop_{\Sigma,\varphi_\mathrm{all}},\Sigma,\varphi_\mathrm{all},C,\P)$. (Note then that
the latter is suitable.)
\end{definition}

\begin{lemma}\label{lem:determines_extend}
Let $\N,\delta$, etc, be as in \ref{dfn:determines}, except that we allow 
$\Tt$ to have uncountable length, and allow $g$ to be in a set-generic extension of $V$. Then 
$(\N|\gamma)[g]$ is closed 
under $\G_\Fop$ 
and letting $\G'$ be the interpretation of $\Psi$ over $(\N|\gamma)[g]$,
$\G'\rest C=\G_\Fop\rest(\N|\gamma)[g]$.
\end{lemma}

We fix a nice $\Fop$, $\MFsharp$, $\Lambda_{\MFsharp} = \Lambda$, $(\Phi,\Psi)$ for the rest of the section. We define $\M_1^\Sigma$ from $\MFsharp$ in the standard way.

\begin{comment}
\begin{definition}
We say a (hod) premouse $\M$ is \textbf{reasonable} iff under $\ZF+\AD$, $\M$
satisfies the
first-order properties which are consequences of
$(\om,\om_1,\om_1)$-iterability, or under $\ZFC$, $\M$ satisfies the first-order
properties which are consequences of $(\om,\om_1,\om_1+1)$-iterability.
\end{definition}

For a premouse, an important consequence of reasonability is condensation. For a
hod premouse, reasonability gives us condensation in between windows of the form
$[\delta,\gamma)$, where $\delta$ is either $0$ or a Woodin cardinal or a limit 
of Woodin
cardinals of $\M$ and $\gamma$ is the least Woodin cardinal above $\delta$ or
$\gamma = o^\M$ if there is no Woodin cardinal above $\delta$.
\end{comment}
See \cite[Section 4]{trang2013} for a proof that if $\Sigma$ is a strategy (of a hod mouse, a suitable mouse) with branch condensation and is fullness preserving with respect to mice in some sufficiently closed, determined pointclass $\Gamma$ or if $\Sigma$ is the unique strategy of a sound ($Y$)-mouse for some mouse operator $Y$ that is projecting, uniformly $\Sigma_1$, $\M_1^{Y,\sharp}$ generically interprets $Y$, and condenses finely then $M_1^{\F,\sharp}$ generically interprets $\F$.

\begin{comment}
\begin{lemma}
\label{GenericInt}
Let $(\P,\Sigma)$ be such that either (a)
$P$ is an $\Fop$-premouse and $\Sigma$ is the unique normal $\Fop$-$\Ord$-iteration
strategy for $P$, where $\Fop$ is some operator that condenses finely; or (b) $P$ is a hod premouse, $(\P,\Sigma)$ is a
hod pair which is $\Gamma$-fullness preserving (for some inductive-like pointclass $\Gamma$) and has branch condensation.  Assume that
$\M_1^{\Sigma,\sharp}$ exists
and is fully iterable. Then $\M_1^{\Sigma,\sharp} generically interprets $\Sigma$.
\end{lemma}

\begin{remark}
In the above lemma, we can replace the $\Ord$-iterability of $\Sigma$ by
$\xi^+ + 1$-iterability and require that card$(\P)\leq \xi$ for some $\xi$. In this case,
by $\M_1^\Sigma$, we mean $\M|\xi^+$, where $\M$ is
the $(\xi^+)^\nth$ iterate of $\M_1^{\Sigma,\sharp}$ via its top extender.
\end{remark}
\end{comment}
Now we are ready to define $g$-organized $\F$-premice.

\begin{definition}[Sargsyan, \cite{ATHM}]\label{genGenTree}
Let $M$ be a transitive structure. Let $\dot{G}$ be the name for the generic 
$G\subseteq\Coll(\omega,M)$ and let $\dot{x}_{\dot{G}}$ be the
canonical name for the real coding $\{(n,m) \ | \ G(n) \in G(m)\}$, where we
identify $G$ with $\bigcup G$. The 
\textbf{tree
$\Tt_M$ for making $M$
generically generic}, is the iteration tree $\Tt$ on $\MFsharp$ of maximal 
length such
that:
\begin{enumerate}
\item $\Tt$ is via $\Lambda$ and is everywhere non-dropping.
\item $\Tt\rest\OR(M)+1$ is the tree given by linearly iterating the first total 
measure of $\MFsharp$ and its images.
\item Suppose $\lh(\Tt)\geq\OR(M)+2$ and let $\alpha+1\in(\OR(M),\lh(\Tt))$.
Let $\delta=\delta(\M^\Tt_\alpha)$ and let $\BB=\BB(M^\Tt_\alpha)$ be the 
extender algebra of
$M^\Tt_\alpha$ at $\delta$. Then $E^\mathcal{T}_\alpha$ is the extender $E$ with
least index in $M^\mathcal{T}_\alpha$ such that for some condition
$p\in\Coll(\omega,M)$, $p \Vdash$``There is a $\mathbb{B}$-axiom induced
by $E$ which fails for $\dot{x}_{\dot{G}}$''.
\end{enumerate}
Assuming that $\MFsharp$ is sufficiently iterable, then $\Tt_M$ exists and has 
successor length.
\end{definition}

Sargsyan noticed that one can feed in 
$\Fop$ into a structure $\N$ indirectly, by feeding in the branches for 
$\Tt_\M$, for various $\M\ins\N$. The operator ${^\g\Fop}$, defined below, and 
used in building g-organized $\Fop$-premice, feeds in 
branches for such $\Tt_\M$. We will also ensure that being such a structure is 
first-order - other than wellfoundedness and the correctness of the branches - 
by allowing sufficient spacing between these branches.

In the following, we let $\N^\T$ denote the last model of the tree $\T$.
\begin{definition} \label{dfn:P^Phi}Given a formula $\Phi$. Given a successor length, 
nowhere dropping tree $\Tt$ on $\MFsharp$, let
$P^\Phi(\Tt)$
be the least $P\ins\N^\Tt$ such that for some cardinal $\delta'$ of
$\N^\Tt$, we have
$\delta'<\OR(P)$ and $P\sats\Phi+$``$\delta'$ is Woodin''. Let
$\lambda=\lambda^\Phi(\Tt)$ be
least such that $P^\Phi(\Tt)\ins M^\Tt_\lambda$. Then $\delta'$ is a
cardinal of $M^\Tt_\lambda$. Let $I^\Phi=I^\Phi(\Tt)$ be
the set of limit ordinals $\leq\lambda$.
\end{definition}

We can now define the operator used for g-organization:

\begin{definition}[${^\g\Fop}$]
\label{modelOpGenerala2}
We define the forgetful operator ${^\g\Fop}$, for $\Fop$ such that $\M_1^{\F,\sharp}$ generically interprets $\F$ as witnessed by a pair $(\Phi,\Psi)$.
Let $b$ be a transitive structure with $\MFsharp\in\J_1(\hat{b})$.
%footnote

%end footnote
We define $\M={^\g\Fop}(b)$, a $\J$-model over $b$, with parameter $\MFsharp$, 
as follows.

For each $\alpha\leq\l(\M)$, $E^{\M|\alpha}=\emptyset$.

Let $\alpha_0$ be the least $\alpha$ such that $\J_\alpha(b)\sats\ZF$. Then 
$\M|\alpha_0=\Jop_{\alpha_0}(b;\MFsharp)$.

Let $\Tt=\Tt_{\M|\alpha_0}$.
We use the notation $P^\Phi=P^\Phi(\Tt)$, $\lambda=\lambda^\Phi(\Tt)$, etc, 
as in \ref{dfn:P^Phi}.
The predicates
$B^{\M|\gamma}$ for $\alpha_0<\gamma\leq\l(\M)$ will be used to feed in branches
for $\Tt\rest\lambda+1$, and therefore $P^\Phi$ itself, into $\M$. Let 
$\left<\xi_\alpha\right>_{\alpha<\iota}$ enumerate 
$I^\Phi\un\{0\}$.

There is a closed, increasing sequence of ordinals
$\left<\eta_\alpha\right>_{\alpha\leq\iota}$ and an increasing sequence of
ordinals $\left<\gamma_\alpha\right>_{\alpha\leq\iota}$ such that:

\begin{enumerate}
 \item $\eta_1=\gamma_0=\eta_0=\alpha_0$.
 \item For each $\alpha<\iota$, 
$\eta_\alpha\leq\gamma_\alpha\leq\eta_{\alpha+1}$, and if
$\alpha>0$ then $\gamma_\alpha<\eta_{\alpha+1}$.
 \item $\gamma_\iota=\l(\M)$, so $\M=\M|\gamma_\iota$.
 \item\label{item:gamma_alpha} Let $\alpha\in(0,\iota)$. Then $\gamma_\alpha$ 
is the least ordinal of
the form $\eta_\alpha+\tau$ such that
$\Tt\rest\xi_\alpha\in\J_\tau(\M|\eta_\alpha)$ and if 
$\alpha>\alpha_0$ then
$\delta(\Tt\rest\xi_\alpha)<\tau$. (We explain below why such $\tau$ 
exists.) And $\M|\gamma_\alpha= \Jop_\tau(\M|\eta_\alpha;\MFsharp)\downarrow b$.
 \item Let $\alpha\in(0,\iota)$. Then
$\M|\eta_{\alpha+1}=\BBB(\M|\gamma_\alpha,\Tt\rest\xi_\alpha,
\Lambda(\Tt\rest\xi_\alpha))\downarrow b$.
 \item Let $\alpha<\iota$ be a limit. Then $\M|\eta_\alpha$ is passive.
 \item $\gamma_\iota$ is the least ordinal of the form $\eta_\iota+\tau$ such
that $\Tt\rest\lambda+1\in\J_{\eta_\iota+\tau}(\M|\eta_\iota)$ and
$\tau>\OR(M^\Tt_\lambda)$; with this $\tau$, $\M=\Jop_\tau(\M|\eta_\iota;\MFsharp)\downarrow 
b$ and furthermore, $^g\Fop(b)$ is acceptable and every strict segment of $^g\Fop(b)$ is sound.\qedhere
\end{enumerate}
\end{definition}

\begin{remark}
It's not hard to see (cf. \cite{trang2013}) $\Mbar\ins\M={^\g\Fop}(b)$, the sequences
$\left<\M|\eta_\alpha\right>_{\alpha\leq\iota}\inter\Mbar$ and
$\left<\M|\gamma_\alpha\right>_{\alpha\leq\iota}\inter\Mbar$ and
$\left<\Tt\rest\alpha\right>_{\alpha\leq\lambda+1}\inter\Mbar$ are
$\Sigma_1^{\Mbar}$ in $\Ll_0^-$, uniformly in $b$ and $\Mbar$.
\end{remark}

\begin{definition}
\label{PotJPremice2}
Let $b$ be transitive with $\MFsharp\in\J_1(\hat{b})$.
A \textbf{potential g-organized $\Fop$-premouse over $b$}
is a potential ${^\g\Fop}$-premouse over 
$b$, with parameter $\MFsharp$.
\end{definition}

\begin{lemma}\label{lem:varphi_g}
There is a formula 
$\varphi_\g$ in $\Ll_0$, such that for any transitive $b$ with 
$\MFsharp\in\J_1(\hat{b})$, and any
$\J$-structure $\M$ over $b$, $\M$ is a potential 
g-organized $\Fop$-premouse over $b$ iff $\M$ is a potential 
$\Lambda_\MFsharp$-premouse over $b$, of type $\varphi_\g$.
\end{lemma}

\begin{lemma}\label{lem:gF_props}
${^\g\Fop}$ is projecting, uniformly $\Sigma_1$, basic, and condenses finely. \end{lemma}

\begin{definition}\label{dfn:whole_reorganized}
Let $\M$ be a g-organized $\Fop$-premouse over $b$. We say $\M$ is 
\textbf{$\Fop$-closed} iff $\M$ is a limit 
of ${^\g\Fop}$-whole proper segments.
\end{definition}

Because $\M_1^{\F,\sharp}$ generically interprets $\Fop$,
$\Fop$-closure ensures closure under $\G_\Fop$:

\begin{lemma}
\label{ClosedUnderJ}
Let $\M$ be an $\Fop$-closed g-organized $\Fop$-premouse over $b$. Then
$\M$ is closed under $\G_\Fop$. In fact, for any set generic extension $\M[g]$ 
of $\M$, with $g\in V$, $\M[g]$ is closed under
$\G_\Fop$ and $\G_\Fop\rest\M[g]$ is definable over $\M[g]$, via a 
formula in $\Ll_0^-$, uniformly in $\M,g$.\end{lemma}

The analysis of scales in $\Lp^{^\g\Fop}(\RR)$ runs into a problem (see \cite[Remark 6.8]{trang2013} for an explanation). Therefore we will analyze scales in a slightly 
different hierarchy.

\begin{definition}\label{dfn:self-scaled}
Let $X\sub\mathbb{R}$. We say that 
$X$ is \textbf{self-scaled} iff there are scales on $X$ and $\RR\cut 
X$ which are analytical (i.e., $\Sigma^1_n$ for some $n<\om$) in $X$.
\end{definition}

\begin{definition}\label{theta_g_organized}
 Let $b$ be transitive with $\MFsharp\in\J_1(\hat{b})$.
 
 Then 
${^\gTheta\F}(b)$ denotes the least $\N\ins{^\g\Fop}(b)$ such that either 
$\N={^\g\Fop}(b)$ or $\J_1(\N)\sats$``$\Theta$ does not exist''. (Therefore 
$\Jop_1(b;\MFsharp)\ins{^\gTheta\Fop}(b)$.)

We say that $\M$ is a 
\textbf{potential 
$\Theta$-g-organized $\Fop$-premouse over $X$} iff $\MFsharp\in\HC^\M$ and for some $X\sub\HC^\M$, 
$\M$ is a
potential ${^\gTheta\Fop}$-premouse over $(\HC^\M,X)$ with 
parameter $\MFsharp$ and $\M\sats$``$X$ is self-scaled''. We write $X^\M=X$.
\end{definition}

In our application to core model induction, we will be most interested in the cases that either 
$X=\emptyset$ or $X=\Fop\rest\HC^\M$.
Clearly $\Theta$-g-organized $\Fop$-premousehood is not first order. Certain aspects of the 
definition, however, are:

\begin{definition} Let ``I am a $\Theta$-g-organized premouse over $X$'' be the 
$\Ll_0$ formula $\psi$ such that for all $\J$-structures $\M$ and $X\in\M$ we 
have
$\M\sats\psi(X)$ iff (i) $X\sub\HC^\M$; (ii) $\M$ is a $\J$-model over 
$(\HC^\M,X)$; (iii) $\M|1\sats$``$X$ is self-scaled''; (iv) every proper segment of 
$\M$ is sound; and (v) for every $\N\ins\M$:
\begin{enumerate}
 \item if $\N\sats$``$\Theta$ exists'' then $\N\downarrow(\N|\Theta^\N)$ is 
a $\param^\N$-strategy premouse of type $\varphi_\g$;
 \item if $\N\sats$``$\Theta$ does not exist'' then $\N$ is passive.\qedhere
\end{enumerate}
\end{definition}

\begin{lemma}\label{lem:charac_Tg-org}
Let $\M$ be a $\J$-structure and $X\in\M$. Then the following are equivalent:
(i) $\M$ is 
a $\Theta$-g-organized $\Fop$-premouse over $X$; (ii) 
$\M\sats$``I am a $\Theta$-g-organized premouse over $X$'' and $\param^\M=\MFsharp$ and 
$\Sigma^\M\sub\Lambda_\MFsharp$; (iii) $\M|1$ is a $\Theta$-g-organized premouse over 
$X$ and every proper segment of $\M$ is sound and for every 
$\N\ins\M$,
\begin{enumerate}
 \item if $\N\sats$``$\Theta$ exists'' then $\N\downarrow(\N|\Theta^\N)$ 
is a g-organized $\Fop$-premouse;
 \item if $\N\sats$``$\Theta$ does not exist'' then $\N$ is 
passive.
\end{enumerate}
\end{lemma}

\begin{lemma}\label{lem:Th-g_very_con}
${^\gTheta\Fop}$ is basic and condenses finely. \end{lemma}

\begin{definition}\label{Lp}
Suppose $\Fop$ is a nice operator and is an iteration strategy and $X\subseteq\mathbb{R}$ is self-scaled. We define 
$\Lp^{^\gTheta\Fop}(\mathbb{R},X)$ as the stack of all $\Theta$-g-organized $\Fop$-mice $\N$ 
over  
$(H_{\om_1},X)$ (with parameter $\MFsharp$). We also say ($\Theta$-g-organized) $\Fop$-premouse \textbf{over 
$\mathbb{R}$} to in fact
mean over $H_{\omega_1}$.\end{definition}
\begin{remark}\label{SamePR}
It's not hard to see that for any such $X$ as in Definition \ref{Lp}, $\powerset(\mathbb{R})\cap \textrm{Lp}^{^\g\Fop}(\mathbb{R},X) = \powerset(\mathbb{R})\cap \textrm{Lp}^{^\gTheta\Fop}(\mathbb{R},X)$. Suppose $\M$ is an initial segment of the first hierarchy and $\M$ is $E$-active. Note that $\M\vDash ``\Theta$ exists" and $\M|\Theta$ is $\Fop$-closed. By induction below $\M|\Theta^\M$, $\M|\Theta^\M$ can be rearranged into an initial segment $\N'$ of the second hierarchy. Above $\Theta^\M$, we simply copy the $E$ and $B$-sequence from $\M$ over to obtain an $\N\lhd \textrm{Lp}^{^\gTheta\Fop}(\mathbb{R},X)$ extending $\N'$.
\end{remark}

In core model induction applications, we often have a pair $(\P,\Sigma)$
where $\P$ is a hod premouse and $\Sigma$ is $\P$'s strategy with branch
condensation and is fullness preserving (relative to mice in some pointclass) or
$\P$ is a sound (hybrid) premouse projecting to some countable set $a$ and
$\Sigma$ is the unique (normal) $\omega_1+1$-strategy for $\P$.  Let $\Fop$ be the operator corresponding to 
$\Sigma$ (using the formula $\varphi_{\rm{all}}$) and suppose $\M_1^{\Fop,\sharp}$ exists. \cite[Lemma 4.8]{trang2013} shows that $\Fop$
condenses finely and $\M_1^{\Fop,\sharp}$ generically interprets $\F$. Also, the core model induction will give us that $\Fop\rest \mathbb{R}$ is self-scaled. \footnote{We abuse notation here, and will continue to do so in the future. Technically, we should write $\Fop\rest$HC.}Thus, we can 
define $\Lp^{^\gTheta\Fop}(\mathbb{R},\Fop\rest\mathbb{R})$ as above (assuming sufficient iterability of $\M_1^{\Fop,\sharp}$). A core model induction 
is then used to prove that
$\Lp^{^\gTheta\Fop}(\mathbb{R},\Fop\rest\mathbb{R})\vDash \textsf{AD}^+$. What's needed
to prove this is the scales analysis of $\Lp^{^\gTheta\Fop}(\mathbb{R},\Fop\rest\mathbb{R})$,
from the optimal hypothesis (similar to those used by Steel; see
\cite{K(R)} and \cite{Scalesendgap}).\footnote{Suppose $\P=\M_1^\sharp$ 
and $\Sigma$ is $\P$'s unique iteration strategy. Let $\Fop$ be the operator corresponding to $\Sigma$. Suppose 
$\rm{Lp}^{^\gTheta\Fop}(\mathbb{R},\Fop\rest \mathbb{R})\vDash \sf{AD}^+ + \sf{MC}$. Then in fact 
$\rm{Lp}^{^\gTheta\Fop}(\mathbb{R})\cap \powerset(\mathbb{R}) = 
\rm{Lp}(\mathbb{R})\cap \powerset(\mathbb{R})$. This is because in 
$L(\rm{Lp}^{^\gTheta\Fop}(\mathbb{R},\Fop\rest\mathbb{R}))$, $L(\powerset(\mathbb{R})) \vDash \sf{AD}^+ \ 
+ \ $$\Theta = \theta_0 + \sf{MC}$ and hence by \cite{sargsyan2014Rmice}, in 
$L(\rm{Lp}^{^\gTheta\Fop}(\mathbb{R},\Fop\rest\mathbb{R}))$, $\powerset(\mathbb{R})\subseteq 
\rm{Lp}(\mathbb{R})$.  Therefore, even though the hierarchies 
$\rm{Lp}(\mathbb{R})$ and $\rm{Lp}^{^\gTheta\Fop}(\mathbb{R},\Fop\rest\mathbb{R})$ are different, as far 
as sets of reals are concerned, we don't lose any information by analyzing the 
scales pattern in $\rm{Lp}^{^\gTheta\Fop}(\mathbb{R},\Fop\rest\mathbb{R})$ instead of that in 
$\rm{Lp}(\mathbb{R})$.} This is carried out in \cite{trang2013}; we will not go into details here, though we simply note that for the scales analysis to go through under optimal hypotheses, we need to work with the $\Theta$-g-organized hierarchy, instead of the g-organized hieararchy.
\subsubsection{BRIEF REMARKS ON $S$-CONSTRUCTIONS}
Suppose $\Fop$ is a nice operator (with parameter $\param$) and suppose $\M$ is a $\mathcal{G}$-mouse (over some transitive $a$), where $\mathcal{G}$ is either $^\g\Fop$ or $^\gTheta\Fop$. Suppose $\delta$ is a cutpoint of $\M$ and suppose $\N$ is a transitive structure such that $\delta\subseteq \N\subseteq \M|\delta$, $\param\in \N$. Suppose $\mathbb{P}\in \J_\omega[\N]$ is such that $\M|\delta$ is $\mathbb{P}$-generic over $\J_\omega[\N]$ and suppose whenever $\Q$ is a $\mathcal{G}$-mouse over $\N$ such that $H^\Q_\delta = \N$ then $\M|\delta$ is $\mathbb{P}$-generic over $\Q$. Then the $S$-constructions (or $P$-constructions) from \cite{schindler2009self} gives a $\mathcal{G}$-mouse $\R$ over $\N$ such that $\R[\M|\delta] = \M$. The $S$-constructions give the sequence $(\R_\alpha : \delta < \alpha\leq \lambda)$ of $\mathcal{G}$-premice over $\N$, where 
\begin{enumerate}[(i)]
\item $\R_{\delta+1} = \Jop_\omega(\N)$;
\item if $\alpha$ is limit then let $\R_\alpha^* = \bigcup_{\beta<\alpha}\R_\beta$. If $\M|\alpha$ is passive, then let $\R_\alpha = \R_\alpha^*$. So $\R_\alpha$ is passive. If $B^{\M|\alpha}\neq \emptyset$, then let $\R_\alpha = (|\R_\alpha^*|; \emptyset, B^{\M|\alpha}, \bigcup_{\beta<\alpha} S^{\R_\beta}, \N ,\param)$. Suppose $E^{\M|\alpha}\neq \emptyset$; let $E^* = E^{\M|\alpha}\cap |\R_\alpha^*|$, then we let $\R_\alpha=(|\R_\alpha^*|; E^*, \emptyset, \bigcup_{\beta<\alpha} S^{\R_\beta}, \N ,\param)$. By the hypothesis, we have $\R_\alpha[\M|\delta]=\M|\alpha$.
\item Suppose we have already constructed $\R_\alpha$ and (by the hypothesis) maintain that $\R_\alpha[\M|\delta]=\M|\alpha$. Then $\R_{\alpha+1} = \Jop_\omega(\R_\alpha)$.
\item $\lambda$ is such that $\R_\lambda[\M|\delta]=\M$. We set $\R_\lambda = \R$.
\end{enumerate} 	
We note that the full constructions from \cite{schindler2009self} does not require that $\delta$ is a cutpoint of $\M$ but we don't need the full power of the $S$-constructions in our paper. Also, the fact that $\M$ is g-organized (or $\Theta$-g-organized) is important for our constructions above because it allows us to get past levels $\M|\alpha$ for which $B^{\M|\alpha}\neq \emptyset$. Because of this fact, in this paper, hod mice are reorganized into the g-organized hierarchy, that is if $\P$ is a hod mouse then $\P(\alpha+1)$ is a g-organized $\Sigma_{\P(\alpha)}$-premouse for all $\alpha<\lambda^\P$. The $S$-constructions are also important in many other contexts. One such context is the local HOD analysis of levels of Lp$^{^\gTheta\Fop}(\mathbb{R},\Fop\rest\mathbb{R})$, which features in the scales analysis of Lp$^{^\gTheta\Fop}(\mathbb{R},\Fop\rest\mathbb{R})$ (cf. \cite{trang2013}).
\subsubsection{CORE MODEL INDUCTION OPERATORS}
To analyze $\Omega$, we adapt the framework for the core model induction developed above and the scales analysis in \cite{trang2013}, \cite{Scalesendgap}, and \cite{K(R)}. We are now in a position to introduce the core model induction operators that we will need in this paper. These are particular kinds of (hybrid) mouse operators that are constructed during the course of the core model induction. These operators can be shown to satisfy the sort of condensation described above and determine themselves on generic extensions. 

Suppose $\Fop$ is a nice operator and $\Gamma$ is an inductive-like pointclass that is determined. Let $\MFsharp = \M_1^{\Fop,\sharp}$. Lp$^{^g\Fop}(x)$ is defined as in the previous section. We write Lp$^{^g\Fop,\Gamma}(x)$ for the stack of $^g\Fop$-premice $\M$ over $x$ such that every countable, transitive $\M^*$ embeddable into $\M$ has an $\omega_1$-$^g\Fop$-iteration strategy in $\Gamma$.

\begin{definition}\label{dfn:k-suitable}
  Let $t\in\HC$ with $\MFsharp\in\J_1(t)$. Let $1\leq k<\om$. A 
premouse $\N$ over $t$ is
\emph{$\Fop$-$\Gamma$-$k$-suitable} (or just \emph{$k$-suitable} if $\Gamma$ and $\Fop$ are clear from the context) iff there is a strictly increasing sequence
$\left<\delta_i\right>_{i<k}$ such that
\begin{enumerate}
 \item $\all\delta\in\N$, $\N\sats$``$\delta$ is Woodin'' if and only if $\ex 
i<k(\delta=\delta_i)$.
 \item $\OR(\N)=\sup_{i<\om}(\delta_{k-1}^{+i})^\N$.
\item\label{item:cutpoint} If $\N|\eta$ is a ${^\g\Fop}$-whole strong 
cutpoint of $\N$ then
$\N|(\eta^+)^\N=\Lp^{^g\Fop,\Gamma}(\N|\eta)$.\footnote{Literally we should 
write ``$\N|(\eta^+)^\N=\Lp^{\Gamma}(\N|\eta)\downarrow t$'', but we 
will be lax about this from now on.}
 \item\label{item:Qstructure} Let $\xi<\OR(\N)$, where $\N\sats$``$\xi$ 
is 
not Woodin''. Then $C_\Gamma(\N|\xi)\sats$``$\xi$ is not Woodin''.
 \end{enumerate}
We write $\delta^\N_i=\delta_i$; also let $\delta_{-1}^\N=0$ and $\delta_k^\N=\OR(\N)$.
\end{definition}

%We identify a core model induction operator $J$, which is defined on a cone over $H_{\omega_1}^{V[G]}$ above some $a$ with its extension on $H_{\omega_2}^{V[G]}$.
%If $\Gamma$ is inductive-like such that $\Sigma \in \bf{\Delta}_\Gamma$ and $\mathcal{A} = \{A_n \ | \ n<\omega\}$ is a self-justifying system in $Env(\Gamma(x))$ for some $x\in\mathbb{R}$ such that $A_0$ is the universal $\Gamma$ set, then $J_{\Sigma,\mathcal{A}}$ is the mouse operator defined as: $J_{\Sigma,\mathcal{A}}(\M) = (\M^+,\in, B)$, where $\M^+ = \textrm{Lp}_\omega^{\Gamma,\Sigma}(\M)$\index{$\M^+$} and $B$ is the term relation for $\mathcal{A}$ (see Definition 4.3.9 of \cite{wilson2012contributions}). $J_{\Sigma,\mathcal{A}}$ is called a \textit{term relation hybrid mouse operator}\index{term relation hybrid mouse operator}.

\begin{definition}[relativizes well]\label{relativizeWell}
Let $\Fop$ be a $Y$-mouse operator for some operator $Y$. We say that $\Fop$ \textbf{relativizes well} if there is a formula $\phi(x,y,z)$ such that for any $a,b\in \textrm{dom}(\Fop)$ such that $a\in L_1(b)$ and have the same cardinality, whenever $N$ is a transitive model of $\textsf{ZFC}^-$ such that $N$ is closed under $Y$, $\Fop(b)\in N$ then $\Fop(a)\in N$ and is the unique $x\in N$ such that $N\vDash \phi[x, a, \Fop(b)]$.
\end{definition}
\begin{definition}[determines itself on generic extensions]\label{detGenExts}
Suppose $\F$ is a $Y$-mouse operator for some operator $Y$. We say that $\Fop$ \textbf{determines itself on generic extensions} if there is a formula $\phi(x,y,z)$, a parameter $a$ such that for almost all transitive structures $N$ of $\sf{ZFC}^-$ such that $\omega_1\subset N$, $N$ contains $a$ and is closed under $\Fop$, for any generic extension $N[g]$ of $N$ in $V$, $\Fop\cap N[g]\in N[g]$ and is definable over $N[g]$ via $(\phi,a)$, i.e. for any $x\in N[g]\cap \dom(\Fop)$, $\Fop(a)=b$ if and only if $b$ is the unique $c\in N[g]$ such that $N[g]\vDash \phi[x,c,a]$.\footnote{By ``almost all", we mean for all such $N$ with the properties listed above and $N$ satisfies some additional property. In practice, this additional property is: $N$ is closed under $\M_1^{\Fop,\sharp}$.}
\end{definition}

The following definition gives examples of ``nice model operators''. This is not a standard definition and is given here for convenience more than anything. These are the kind of model operators that the core model induction in this paper deals with. We by no means claim that these operators are all the useful model operators that one might consider. Recall we fixed a $V$-generic $G\subseteq Col(\omega,\kappa)$. 

\begin{definition}[Core model induction operators]\label{cmi operator} \index{core model induction operators}Suppose $(\P, \Sigma)$ is a hod pair below $\kappa$; assume furthermore that $\Sigma$ is a $(\lambda^+,\lambda^+)$-strategy. Let $\Fop = \Fop_{\Sigma,\varphi_{\rm{all}}}$ (note that $\Fop$, $^g\Fop$ are basic, projecting, uniformly $\Sigma_1$, and condenses finely). Assume $\Fop\rest\mathbb{R}$ is self-scaled. We say $J$ is a $\Sigma$ core model induction operator or just a $\Sigma$-cmi operator if in $V[G]$, one of the following holds:
\begin{enumerate}
 \item $J$ is a projecting, uniformly $\Sigma_1$, first order $\Fop$-mouse operator (or $^\g\Fop$-mouse operator) defined on a cone of $(H_{\omega_1})^{V[G]}$ above some $a\in (H_{\omega_1})^{V[G]}$. Furthermore, $J$ relativizes well. 
 
\item  For some $\a\in \mathrm{OR}$ such that $\a$ ends either a weak or a strong gap in the sense of \cite{K(R)} and \cite{trang2013}, letting $M=\textrm{Lp}^{^\gTheta\Fop}(\mathbb{R},\Fop\rest\mathbb{R})||\a$ and $\Gamma = (\Sigma_1)^M$, $M\models \sf{AD}$$^++\sf{MC}$$(\Sigma)$\footnote{\label{MC}$\textsf{MC}(\Sigma)$ stands for the Mouse Capturing relative to $\Sigma$ which says that for $x, y\in \mathbb{R}$, $x$ is $\mathrm{OD}(\Sigma, y)$ (or equivalently $x$ is $\mathrm{OD}(\Fop, y)$) iff $x$ is in some $^\g\Fop$-mouse over $y$. $\sf{SMC}$ is the statement that for every hod pair $(\P,\Sigma)$ such that $\Sigma$ is fullness preserving and has branch condensation, then $\textsf{MC}(\Sigma)$ holds.}. For some transitive $b\in H^{V[G]}_{\omega_1}$ and some g-organized $\Fop$-premouse $\Q$ over $b$, $J=\Fop_\Lambda$, where $\Lambda$ is an $(\omega_1, \omega_1)$-iteration strategy for a $1$-suitable (or more fully $\Fop$-$\Gamma$-$1$-suitable) $\Q$ which is $\Gamma$-fullness preserving, has branch condensation and is guided by some self-justifying-system (sjs) $\vec{A}=(A_i: i<\omega)$ such that $\vec{A}\in OD_{b, \Sigma, x}^M$ for some real $x$ and $\vec{A}$ seals the gap that ends at $\alpha$\index{seal a gap}\footnote{This implies that $\vec{A}$ is Wadge cofinal in $\bf{Env}(\Gamma)$, where $\Gamma = \Sigma_1^{M}$. Note that $\bf{Env}(\Gamma) = \powerset(\mathbb{R})^M$ if $\alpha$ ends a weak gap and $\bf{Env}(\Gamma) = \powerset(\mathbb{R})^{\textrm{Lp}^\Sigma(\mathbb{R})|(\alpha+1)}$ if $\alpha$ ends a strong gap.}. %We call $J$ \textit{non-plain}.
%\item  For some $H$, $H$ satisfies a or b above and for some $n<\omega$, $F$ is the $x\rightarrow \M^{\#, H}_n(x)$ operator or for some $b\in HC$, $F$ is the $\omega_1$-iteration strategy of $\M_n^{\#, H}(b)$.
\end{enumerate}
\end{definition}
\begin{remark}\label{cmi-op-remark}
1) The $\Sigma$-cmi operators $J$ we construct in this paper also determine themselves on generic extensions. If $J$ is defined as in (1) and determines itself on generic extensions then so does the ``next operator" $\M_1^{J,\sharp}$. If $J$ is defined as in (2), then \cite{trang2013} shows that $M_1^{J,\sharp}$ generically interprets $J$; from this, the proof of Lemma \ref{ClosedUnderJ} (see \cite{trang2013}[Lemma 4.21]) shows that $J$ determines itself on generic extensions.\\
\indent 2) Suppose $J$ is defined on a cone over $(H_{\omega_1})^{V[G]}$ above some transitive $a\in H^V_{\kappa^+}$ and $J\rest V\in V$. During the course of construction, we show that knowing $J$ on $V$ is sufficient to determine $J$ on $V[G]$. During the course of the core model induction, we'll be first constructing these $\Sigma$-cmi operators $J$'s on $H^V_{\kappa^+}$ (above some $a$); then we show how to extend $J$ to $HC^{V[G]}$; we then lift $J$ to $H^V_{\lambda^+}$, which then extend $J$ to $H_{\lambda^+}^{V[G]}$.\\
\indent 3) By results in \cite{ATHM}, under $(\dag)$, if $(\P,\Sigma)$ is a hod pair such that $\Sigma$ has branch condensation, then $\Sigma$ has hull condensation. The same is true for $(\Q,J)$ in Definition \ref{cmi operator}. This implies that $\Sigma$ ($J$) is suitably condensing.
\end{remark}
%If $M,(\P,\Sigma)$ are as in Definition \ref{cmi operator} and $(\Q,F)$ is as in clause 2) then letting $\Gamma = (\Sigma_1)^M$, $\Q$ is a $1$-suitable premouse relative to $\Gamma$, where a $k$-suitable premouse relative to $\Gamma$ (or a $k$-$\Gamma$-suitable premouse) is defined as follows (the strategy $\Sigma$ is clear from the context so for brevity, we won't display it in the notation).  

\subsection{GETTING $M_1^{J,\sharp}$ AND LIFTING}
\label{M1sharp}
We assume the hypothesis of Theorem \ref{main_technical_theorem}. We fix a $V$-generic $G\subseteq Col(\omega,\kappa)$ and recall we that we let $\lambda = 2^{\kappa}$. Suppose $(\P^*,\Sigma)$ is a hod pair below $\kappa$ such that $\Sigma$ is an $(\kappa^+,\kappa^+)$-strategy in $V[G]$ and $\Sigma\rest V\in V$ (or $(\P^*,\Sigma) = (\emptyset,\emptyset)$). Suppose $J$ is a $\Sigma$-cmi-operator. As part of the induction, we assume $J$ is defined on a cone in $H_{\lambda^+}^{V[G]}$ above some $x\in H_{\kappa^+}^V$ and $J\rest V\in V$. \footnote{We note the specific requirement that the cone over which $J$ is defined is above some $x\in V$. These are the $\Sigma$-cmi-operators that we will propagate in our core model induction. We will not deal with all $\Sigma$-cmi-operators.} We first show $\M_1^{J,\sharp}(a)$ exists (and is $(\kappa^+,\kappa^+)$-iterable) for $a\in H^V_{\kappa^+}\cap \textrm{dom}(J)$. We then show that $\M_1^{J,\sharp}(a)$ is defined on $H_{\lambda^+}^V$ and $\M_1^{J,\sharp}$ is $(\lambda^+,\lambda^+)$-iterable for all $a\in H_{\lambda^+}^V$. Finally, we get that  $\M_1^{J,\sharp}$ is a $\Sigma$-cmi operator defined on a cone in $H_{\lambda^+}^{V[G]}$.
%Throughout this section, we fix a $B\subseteq\kappa$ coding $H_\kappa^V$. 

Let $\Fop = \Fop_{\Sigma,\varphi_{\rm{all}}}$. Let $A$ code $\P^*$ and $\textrm{Lp}_*^\Fop(A)$ be the union of all $\N$ such that $\N$ is $\omega$-sound above $A$, $\N$ is a countably iterable $\Sigma$-premouse over $A$ and $\rho_{\omega}(\N) \leq \textrm{sup}(A)$. This means whenever $\pi:\N^*\rightarrow \N$ is elementary, $\N^*$ is countable, transitive, then $\N^*$ is $(\omega,\omega_1+1)$ iterable via a unique strategy $\Lambda$ such that whenever $\M$ is a $\Lambda$-iterate of $\N^*$, then $\M$ is a $\Sigma^\pi$-premouse. As a matter of notations, in $V$, for $A$ a bounded subset of $(\lambda^+)^V$, we set
\begin{center}
$\textrm{Lp}_1^\Sigma(A) = \textrm{Lp}_*^\Fop(A)$.
\end{center}
Suppose $\textrm{Lp}_\alpha^\Sigma(A)$ has been defined for $\alpha<\lambda^+$,
\begin{center}
$\textrm{Lp}_{\alpha+1}^\Sigma(A) = \textrm{Lp}^{\Fop}_{*,+}(\textrm{Lp}^\Sigma_\alpha(A))$,\footnote{$\textrm{Lp}^{\Fop}_{*,+}(\textrm{Lp}^\Sigma_\alpha(A))$ is defined similarly to Lp$_*^\Fop$ but here we stack continuing, $\Fop$-sound $\Fop$-premice.} and
\end{center}
for $\xi<\kappa^+$ limit,
\begin{center}
$\textrm{Lp}_\xi^\Sigma(A) = \bigcup_{\alpha<\xi} \textrm{Lp}_{\alpha}^\Sigma(A)$.
\end{center}
We define Lp$^{^\g\Sigma}_\lambda(A)$ and Lp$_\lambda^{^\gTheta\Sigma}(A)$ similarly for $\xi\leq \lambda^+$, in the presence of $\M_1^{\Sigma,\sharp}$. We also write Lp$^\Sigma(A)$ for Lp$_*^\Fop(A)$ and similarly for Lp$^{^\g\Sigma}(A)$ and Lp$^{^\gTheta\Sigma}(A)$. We work in $V$ for a while. %Without loss of generality, we will assume $(\P,\Sigma)=(\emptyset,\emptyset)$ and $J = \textrm{rud}$. We omit $J$ from the notation. The proofs we give below just involve more notations otherwise.
\begin{lemma}
\label{BasicFacts}
Let $A$ be a subset of $\lambda$ coding $\P^*$. Then $\textrm{Lp}^\Sigma_{\lambda^+}(A) \vDash \lambda^+$ exists. Similarly, $\textrm{Lp}^{^\g\Sigma}_{\lambda^+}(A)$, $\textrm{Lp}^{^\gTheta\Sigma}_{\lambda^+}(A) \vDash \lambda^+$ exists.
\end{lemma}
\begin{proof}
Suppose not. This easily implies that we can construct over $\textrm{Lp}^\Sigma_{\lambda^+}(A)$ a $\square_{\lambda}$-sequence\footnote{Squares hold in $\textrm{Lp}^\Sigma_{\kappa^+}(A)$ because $\Sigma$ has hull and branch condensation.}. This contradicts $\neg \square_{\lambda}$ in $V$.%\footnote{The same proof also show that if $A$ is bounded in $\lambda$ and $\xi\in (\textrm{sup}(A),\kappa)$ is a regular cardinal, then $\xi$ is inaccessible in $\textrm{Lp}^\Sigma_{\kappa^+}(A)$.}
\end{proof}

The following gives the main consequence of the failures of squares assumption. It allows us to run covering arguments later.

\begin{lemma}\label{lem:smallCof}
Let $A, \Sigma$ be as in \ref{BasicFacts}. Let $M\in \{Lp^\Sigma_{\lambda^+}(A), \textrm{Lp}^{^\g\Sigma}_{\lambda^+}(A)$, $\textrm{Lp}^{^\gTheta\Sigma}_{\lambda^+}(A)\}$ and $\gamma = (\lambda^+)^M$. Then $cof(\gamma)\leq \kappa$. 
\end{lemma}
\begin{proof}
First note that $\gamma < \lambda^+$ by Lemma \ref{BasicFacts}. Now suppose cof$(\gamma) = \xi$ for some regular cardinal $\xi \in [\kappa^+,\lambda]$. Let $f:\xi \rightarrow \gamma$ be cofinal and continuous. Using $f$ and $\square_\lambda$ in $M$, by a standard argument (see \cite{schimmerling2007coherent}) we can construct a non-threadable sequence of length $\xi$.\footnote{A thread will allows us to construct a $\Sigma$-mouse projecting to $A$ and extends $M|\gamma$ but not in $M$. This is a contradiction to the definition of $M$.}  This contradicts $\neg \square(\xi)$.
\end{proof}

Let $S$ be the set of $X\prec H_{\lambda^{++}}$ such that $\kappa\subseteq X$, $|X|=\kappa$, $X^\omega\subseteq X$, and $X$ is cofinal in the ordinal height of Lp$^\Sigma(B)$, Lp$^{^\gTheta\Sigma}(B)$ and $J, (\P^*\cup\{\P^*\},\Sigma)\in X$.\footnote{This means $(\P^*,\Sigma\rest V)\in X$ and $\Sigma\in X[G]$ but we will abuse notation here.} So $S$ is stationary. As before, we let $\pi_X:M_X\rightarrow H_{\lambda^{++}}$ be the uncollapsed map and $\lambda_X$ be the critical point of $\pi_X$. We first prove some lemmas about ``lifting" operators. In the following, when we write ``Lp$^{^\gTheta\Fop}$", we implicitly assume $\M_1^{\F,\sharp}$ exists and is $(\lambda^+,\lambda^+)$-iterable. We will prove this at the end of the section.
\begin{lemma}
\label{LpClosed}
Suppose $A^*\subseteq \lambda$. Suppose $X\in S$ such that $A^*\in X$ and $X$ is cofinal in Lp$^\Sigma(A^*)$ (there are stationary many such $X$ because cof$(o(Lp^\Sigma(A^*)))\leq \kappa$ by \ref{lem:smallCof}). Let $\pi_X(A)=A^*$.  Then $Lp^\Sigma(A)\subseteq M_X$. The same conclusion holds if we replace $Lp^\Sigma(A)$ by $Lp^{^\gTheta\Sigma}(A)$ or $Lp^{^\g\Sigma}(A)$.
\end{lemma}
\begin{proof}
We just prove the first clause. Suppose not. Then let $\M\lhd \textrm{Lp}^\Sigma(A)$ be the least counterexample. Let $E$ be the $(\lambda_X,\lambda)$-extender derived from $\pi_X$. Let $\N = \textrm{Ult}(\M,E)$. Then any countable transitive $\N^*$ embeddable into $\N$ (via $\sigma$) is embeddable into $\M$ (via $\tau$) such that $i_E\circ \tau = \sigma$ by countable completeness of $E$. So $\N^*$ is $\omega_1+1$ $\Sigma^\sigma$-iterable because $\M\lhd \textrm{Lp}^{\Sigma}(A)$, $\sigma^{-1}(\P^*)=\tau^{-1}(\P^*)$, and $\sigma\rest \sigma^{-1}(\P^*) = \tau\rest \sigma^{-1}(\P^*)$. So $\N\lhd \textrm{Lp}^\Sigma(A^*)$. But since $\pi_X$ is cofinal in Lp$^\Sigma(A^*)$, $\N\notin \textrm{Lp}^\Sigma(A^*)$. Contradiction.
\end{proof}
\begin{lemma}
\label{LiftOperators}
\begin{enumerate}
\item If $H$ is defined by $(\psi,a)$ on $H^{V[G]}_{\omega_1}$ (as in clause 1 of \ref{cmi operator}) with $a\in V$ and $H\rest V\in V$, then $H$ can be extended to an operator $H^+$ defined by $(\psi,a)$ on $H_{\lambda^+}$. Furthermore, $H^+$ relativizes well.
\item If $(\Q,F)$ and $\Gamma$ are as in clause 2 of Definition \ref{cmi operator}, where $F$ plays the role of $\Lambda$ there with $(\Q,F\rest V)\in V$, then $F$ can be extended to a $(\lambda^+,\lambda^+)$-strategy that has branch condensation. Furthermore, there is a unique such extension.
\end{enumerate}
\end{lemma}
\begin{proof}
To prove 1), first let $A^*$ be a bounded subset of $\lambda^+$ (in the cone above $a$) and let $X\in S$ such that $A^*\in X$ and $X$ is cofinal in Lp$^\Sigma(A^*)$. Let $\pi_X(A) = A^*$. We assume $H$ is an $\Fop$-mouse operator. By Lemma \ref{LpClosed}, $H(A)\in M_X$ and hence we can define $H^+(A^*) = \pi_X(H(A))$ (as the first level $\M \lhd \textrm{Lp}^\Sigma(A^*)$ that satisfies $\psi[A^*,a]$). This defines $H^+$ on all bounded subsets of $\kappa^+$. The same proof works for $J$ being a $^\g\Fop$-mouse operator. We can then define $H^+$ on all of $H_{\kappa^+}$ using the fact that $J$ relativizes well and $|H_\kappa|=\kappa$. It's easy to see then that $H^+$ also relativizes well.

We first prove the ``uniqueness" clause of 2). Suppose $F_1$ and $F_2$ are two extensions of $F$ and let $\mathcal{T}$ be according to both $F_1$ and $F_2$. Let $b_1 = F_1(\mathcal{T})$ and $b_2 = F_2(\mathcal{T})$. If $b_1\neq b_2$ then cof$(lh(\mathcal{T})) = \omega$. So letting $\mathcal{T}^*$ be a hull of $\mathcal{T}$ such that $|\mathcal{T}^*|\leq \omega_2$ and letting $\pi:\mathcal{T}^*\rightarrow \mathcal{T}$ be the hull embedding, then $b_1\cup b_2 \subseteq \textrm{rng}(\pi)$. Then $\pi^{-1}[b_1]=F(\mathcal{T}^*) \neq \pi^{-1}[b_2] = F(\mathcal{T}^*)$. Contradiction.

To show existence, let $F_{\kappa^+} = F$. Inductively for each $\kappa^+\leq \xi <\lambda^+$ such that $\xi$ is a limit ordinal, we define a strategy $F_\xi$ extending $F_\alpha$ for $\alpha<\xi$ and $F_\xi$ acts on trees of length $\xi$. For $X\prec Y \prec H_{\lambda^{++}}$, let $\pi_{X,Y} = \pi_Y^{-1}\circ \pi_X$. Let $\mathcal{T}$ be a tree of length $\xi$ such that for all limit $\xi^*<\xi$, $\mathcal{T}\rest \xi^*$ is according to $F_{\xi^*}$. We want to define $F_{\xi}(\mathcal{T})$.

For $X\in S$ such that $X$ is cofinal in Lp$_{*,+}^{^\g\Sigma}(\M(\T))$ (such an $X$ exists by the proof of \ref{lem:smallCof} again)\footnote{Recall Lp$_{*,+}^{^\g\Sigma}(\M(\T))$ is just Lp$_{*,+}^{^\g\Fop}(\M(\T))$.}, let $(\mathcal{T}_X,\xi_X) = \pi_X^{-1}(\mathcal{T},\xi)$ and $b_X = F(\mathcal{T}_X)$. Let $c_X$ be the downward closure of $\pi_X[b_X]$ and $c_{X,Y}$ be the downward closure of $\pi_{X,Y}[b_X]$.
\\
\\
\noindent \textbf{Claim:} For all $\gamma<\xi$, either $\forall^* X\in S\footnote{This means the set of such $X$ is $C\cap S$ for some club $C$.} \ \gamma \in c_X$ or $\forall^* X \in S \ \gamma \notin c_X$.
\begin{proof}
The proof is similar to that of Lemma 2.5 in \cite{PFA} so we only sketch it here. Suppose for contradiction that there are stationarily many $X\in S$ such that $\gamma\in c_X$ and there are stationarily many $Y\in S$ such that $\gamma\notin c_Y$. Suppose first cof$(\xi)\in [\omega_1,\kappa]$. Note that crt$(\pi_X)$, crt$(\pi_Y) > \kappa$. It's easy then to see that $\pi_X[b_X]$ is cofinal in $\xi$ and $\pi_Y[c_Y]$ is cofinal in $\xi$. Hence $c_X = c_Y$. Contradiction.

Now suppose cof$(\xi) = \omega$. Fix a surjection $f:\lambda\twoheadrightarrow \xi$. $\forall^* X\in S \ (f,\xi) \in X$ so let $(f_X,\xi_X) = \pi_X^{-1}(f,\xi)$. For each such $X$, let $\alpha_X$ be least such that $f_X[\alpha_X]\cap b_X$ is cofinal in $\xi_X$. By Fodor's lemma,
\begin{center}
$\exists \alpha \exists U \ (U \textrm{ is stationary} \wedge \forall X \in U \ \alpha_X = \alpha$).
\end{center}
By symmetry and by thinning out $U$, we may assume
\begin{center}
$X\in U \Rightarrow \pi_X^{-1}(\gamma)\in b_X$.
\end{center}
Fix $Y\in S$ such that $\gamma\notin c_Y$ and $\alpha<\lambda_Y$. Since $U$ is stationary, there is some $X\in U$ such that $Y \prec X$, which implies
\begin{center}
$\pi_{Y,X}[f_Y[\alpha]] = f_X[\alpha]$
\end{center}
is cofinal in $b_X$ and hence $\mathcal{T}_Y^\smallfrown \pi^{-1}_{Y,X}[b_X]$ is a hull of $\mathcal{T}_X$. Since $F$ condenses well, $\pi_{Y,X}^{-1}[b_X]=b_Y$. This contradicts the fact that $\pi_X^{-1}(\gamma)\in b_X$ but $\pi_Y^{-1}(\gamma)\notin b_Y$.

Finally, suppose cof$(\xi) \geq \kappa^+$. The case $\mathcal{T}_X$ is maximal is proved exactly as in Lemma 1.25 of \cite{PFA}. Suppose $\T_X$ is short and is according to $F$. Note that lh$(\T_X)$ has uncountable cofinality (in $V$). We claim that $\forall^* X\in S \ b_X = F(\T_X) \in M_X$. Given the claim we get that for any two such $X \prec Y$ satisfying the claim, $\pi_{X,Y}(b_X)$ is cofinal in $\T_Y$ and hence $\pi_{X,Y}(b_X) = b_Y$. This gives $c_{X,Y}$ is an initial segment of $b_Y$, which is what we want to prove.

Now to see $\forall^* X\in S \ b_X = F(\T_X) \in M_X$. We first remind the reader $\Q(\T_X)$ is the least $\Q \lhd \textrm{Lp}^{^\g\Sigma,\Gamma}_{+}(\M(\T_X))$ that defines the failure of Woodinness of $\delta(\T_X)$. Since $\delta(\T)$ has uncountable cofinality (in $V$ and in $V[G]$), by a standard interpolation argument, whenever $\M_0,\M_1\in \textrm{Lp}^{^\g\Sigma,\Gamma}_{+}(\M(\T_X))$ then we have either $\M_0\unlhd \M_1$ or $\M_1 \unlhd \M_0$. So the ``leastness" of $\Q(\T_X)$ is justified in this case. By the same proof as that of Lemma \ref{LpClosed} and the fact that $X$ is cofinal in Lp$^{^\g\Sigma}_{*,+}(\M(\T))$ and $\textrm{Lp}^{^\g\Sigma,\Gamma}_{+}(\M(\T_X))\unlhd \textrm{Lp}^{^\g\Sigma}_{*,+}(\M(\T))$, we get $\Q(\T_X)\in M_X$. 

Now $F(\T_X) = b_X$ is the unique branch $b$ such that $\Q(b,\T_X)$ exists and $(\Q(b,\T_X)^*,\M(\T_X))$\footnote{See \cite{DMATM} for more on $\ast$-translations.$(\Q(b,\T_X)^*,\M(\T_X))$ is fine-structurally equivalent to $\Q(b,\T_X)$ but itself is a $^\gTheta\Sigma$-premouse over $\M(\T_X)$.} is $\Q(\T_X)$. The uniqueness of $b_X$ follows from a standard comparison argument. By an absoluteness argument and the fact that $\Q(\T_X)\in M_X$, $b_X\in M_X$. We're done.
\end{proof}
Using the claim, we can just define 
\begin{center}
$\gamma \in F_\xi(\mathcal{T}) \Leftrightarrow \forall^* X\in S \ \gamma \in c_X$. 
\end{center}
It's easy to verify that with this definition, the unique extension of $F$ to a $(\kappa^+,\kappa^+)$ strategy has branch condensation. This completes the proof sketch of the lemma.
\end{proof}

\begin{lemma}
\label{LiftOperatorsToGenExt}
\begin{enumerate}
\item If $H$ is a defined by $(\psi,a)$ on $H^{V[G]}_{\omega_1}$ as in clause 1 of \ref{cmi operator} with $a\in H_{\kappa^+}^V$, then $H$ can be extended to a first order mouse operator $H^+$ defined by $(\psi,a)$ on $H^{V[G]}_{\lambda^+}$. Furthermore, $H^+$ relativizes well and if $H$ determines itself on generic extensions then so does $H^+$.
\item If $(\Q,\Lambda)$ and $\Gamma$ are as in clause 2 of Definition \ref{cmi operator}, where $F$ plays the role of $J$ there and $(\Q,\Lambda\rest V)\in V$, then $\Lambda$ can be extended to a unique $(\lambda^+,\lambda^+)$-strategy that has branch condensation in $V[G]$.
\end{enumerate}
\end{lemma}
\begin{proof}
For (1), let $b\in H_{\lambda^+}^{V[G]}$ and let $\tau \in H^V_{\lambda^+}$ be a nice Col$(\omega,\kappa)$-name for $b$ (Col$(\omega,\kappa)$ is $\kappa^+$-cc so such a name exists by the choice of $\lambda$).\footnote{In particular, a nice Col$(\omega,\kappa)$-name for a real can be considered a subset of $\kappa$ and hence a nice Col$(\omega,\kappa)$-name for $\mathbb{R}^{V[G]}$ is an element of $H_{\lambda^+}^V$.} Assume $H$ is a $\Sigma$-mouse operator (the other case is proved similarly). Let $X\in S$ be such that $\P^*\cup\{\P^*\},\Sigma, b, \tau, H(\tau) \in X[G]$; here we use Lemma \ref{LiftOperators} to get that $H(\tau)$ is defined. Let $(\bar{b},\bar{\tau}) = \pi_X^{-1}(b,\tau)$. Then $\pi_X^{-1}(H(\tau)) = H(\bar{\tau})\in M_X$ by condensation of $H$. Since $H$ relativizes well, $H(\bar{b})\in M_X[G]$. This means we can define $H^+(b)$ to be $\pi_X(H(\bar{b}))$. We need to see that $H^+(b)$ is countably $\Sigma$-iterable in $V[G]$. So let $\pi: \N\rightarrow H^+(b)$ with $\N$ countable transitive in $V[G]$ and $\pi(b^*)=b$. Let $X\subset Y\in S$ be such that ran$(\pi)\subseteq \textrm{ran}(\pi_Y)$; then $H(\pi_Y^{-1}(b))\in M_Y[G]$ and there is an embedding from $\N$ into $H(\pi_Y^{-1}(b))$, so $\N$ has an $(\omega_1,\omega_1+1)$-$\Sigma$-iteration strategy. The definition doesn't depend on the choice of $X$ and it's easy to see that $H^+$ satisfies the conclusion.

For (2), let $M\in H_{\lambda^+}^{V[G]}$ be transitive and $\tau\in H_{\lambda^+}^V$ be a Col$(\omega,\kappa)$-term for $M$. We define the extension $\Lambda^+$ of $\Lambda$ as follows (it's easy to see that there is at most one such extension). In $N=L_{\lambda^+}^{\Lambda^*}[A,\mathfrak{M}]$, where $A\subseteq\lambda$ codes $tr.cl.(\tau)$ and a well-ordering of $tr.cl.(\tau)$, $\Lambda^*$ is the unique $(\lambda^+,\lambda^+)$-$\Lambda$-strategy for $\mathfrak{M}=\M_1^{\Lambda,\sharp}$ in $V$. $\Lambda^*$ exists by Lemma \ref{LiftOperators}.

Let $\T_{tr.cl.(\tau)}$ be according to $\Lambda^*$ and be defined as in Definition \ref{genGenTree}. Note that $(\lambda^{+})^{N} < (\lambda^+)^V$ by $\neg \square_{\lambda}$ and the fact that $\square_{\lambda}$ holds in $N$, so $\T_{tr.cl.(\tau)}\in N$ and has length less than $o(N)=\lambda^+$. Let $\R$ be the last model of $\T_{tr.cl(\tau)}$ and note that by the construction of $\T_{tr.cl.(\tau)}$, $M$ is generic over $\R$. Let $\U\in M$ be a tree according to $\Lambda^+$ of limit length, then set $\Lambda^+(\U) = b$ where $b$ is given by (the proof of) \cite[Lemma 4.8]{trang2013} by interpreting $\Lambda$ over generic extensions of $\R$. By a simple reflection argument, it's easy to see that $\Lambda^+(\U)$ doesn't depend on $M$.\footnote{Let $M, M^*$ be such that $\U\in M\cap M^*$; let $\tau,\tau^*$ be nice Col$(\omega,\kappa)$-terms for $M, M^*$ respectively. In $V[G]$, let $X[G]$ contain all relevant objects and $X\in S$. Let $\bar{a}=\pi_X^{-1}(a)$ for all $a\in X[G]$. Then letting $b_0,b_1$ be the branches of $\bar{\mathcal{U}}$ given by applying \cite[Lemma 4.8]{trang2013} in $L^{\Lambda^*}[tr.cl.(\bar{\tau}),<_1,\mathfrak{M}], L^{\Lambda^*}[tr.cl.(\bar{\tau^*}),<_2,\mathfrak{M}]$ (built inside $M_X[G]$), where $<_1$ is a well-ordering of $\bar{\tau}$ and $<_2$ is a well-ordering of $\bar{\tau^*}$. Then $b_0 = b_1$ as both are according to $\Lambda$, since $(\mathfrak{M},\Lambda^*)$ generically interprets $\Lambda$ in $V[G]$.} This completes the construction of $\Lambda^+$. It's easy to see that $F^+$ has branch condensation.
\end{proof}

Let $J$ be as above. We now proceed to construct $\M_1^{J,\sharp}$. We denote $\M_0^{J,\sharp}(x)$ for the least $E$-active, sound $J$-mouse over $x$.

\begin{lemma}
\label{SharpsExist}
For every $A$ bounded in $\lambda^+$, $\M_0^{J,\sharp}(A)$ exists.
\end{lemma}
\begin{proof}
By Lemma \ref{LiftOperators}, it's enough to show that if $A$ is a bounded subset of $\kappa^+$, then $\M_0^{J,\sharp}(A)$ exists. Fix such an $A$ and let $X\in S$ such that $\textrm{sup}(A)\cup\{A,\textrm{sup}(A), J\}\subseteq X$ and $X$ is cofinal in the $L^J[A]$-successor of $\kappa^+$, which has cofinality at most $\kappa$. Hence $\pi_X(A) = A$. Let
\begin{center}
$\mu = \{B \subseteq \lambda_X \ | \ \lambda_X\in \pi_X(B) \wedge B\in L^J[A]\footnote{We only build $L^J[A]$ up to $\lambda^+$.}\}$.
\end{center}
\begin{claim}
$\mu$ is a countably complete $L^J[A]$-ultrafilter,
\end{claim} 
\begin{proof}

Let $Q = L^J[A]$ and $P = \pi_X^{-1}(Q)$. Let $\eta=\kappa^+$ and $\xi = \pi_X^{-1}((\eta^+)^Q)$. Let $\kappa_0 = \lambda_X$. Then $\xi = (\kappa_0^+)^Q$. This is because $X$ is cofinal in $(\eta^+)^Q$. So $\mu$ is indeed total over $L^J[A]$. Using the fact that $X^\omega\subset X$, we get that $\mu$ is countably complete.
\begin{comment}
as $\powerset(\lambda_X)\cap L^J[A] \subseteq M_X$. Furthermore, since $M_X^\omega \subseteq M_X$, $E$ is countably complete, hence Ult$(L^J[A],\mu)$ is wellfounded, so we identify it with its transitive collapse $M$. Furthermore, $M = L^J[A]$\footnote{Let $X\prec H_{\kappa^{++}}$ be countable containing all relevant objects. Let $\pi:M\rightarrow X$ be the uncollapse map and for each $a\in X$, let $a^* = \pi^{-1}(a)$. By countable completeness of $\mu$, in $M$, Ult$(L^{J^*}[A^*],\mu^*)$ realizes into $L^J[A]$ and is in fact $L^{J^*}[A^*]$ (in $M$).}.
\end{comment}
\end{proof}

We need to know that when iterating $L^J[A]$ by $\mu$ and its images, the iterates are $L^J[A]$. This follows from a well-known argument by Kunen. The point is that iterates of $L^J[A]$ by $\mu$ and its images can be realized back into $L^J[A]$ and hence since $J$ condenses well, the ultrapowers are $L^J[A]$. We outline the proof here for the reader's convenience (see \cite[Theorem 28]{DFSR} for a similar argument).

Let $\mu_0=\mu$, $\xi_0=\xi$, and $\M_0 = (L_\xi^J[A],\mu_0)$. By the usual Kunen's argument, $\M_0$ is an amenable structure. By induction on $\alpha<\kappa^+$, we define:
\begin{enumerate}
\item $\M_\alpha$, the $\alpha$-th iterate of $\M_0$ by $\mu_0$ and its images,
\item maps $\pi^*_{\beta,\alpha}:\M_\beta\rightarrow \M_\alpha$ for $\beta<\alpha$,
\item maps $\pi_{\beta,\alpha}: L^J[A]\rightarrow L^J[A]$ extending $\pi^*_{\beta,\alpha}$,
\item maps $\tau_\alpha: L^J[A]\rightarrow L^J[A]$ such that $\forall \beta<\alpha$, $\tau_\beta=\tau_\alpha\circ \pi_{\beta,\alpha}$.
\end{enumerate}

For $\alpha=0$, let $\pi^*_{0,1},\pi_{0,1}$ be the $\mu_0$-ultrapower maps and let $E$ be the $(\kappa_0,\omega_3^V)$-extender derived from $\pi_X$ and let
\begin{center}
$\tau_0: L^J[A]\rightarrow \textrm{Ult}(L^J[A],E)$.
\end{center}

It's easy to see that:
\begin{itemize}
\item $\tau_0$ has a stationary set of fixed points.
\item $\textrm{Ult}(L^J[A],E)=L^J[A]$ (similarly, Ult$(L^J[A],\mu_0) = L^J[A]$). Let $Y\prec H_{\lambda^{++}}$ be countable containing all relevant objects and $\pi:M\rightarrow Y$ be the uncollapse map and for each $a\in Y$, let $a^* = \pi^{-1}(a)$. Using countable completeness of $E$, it is easy to check that in $M$, Ult$(L^{J^*}[A^*],E^*)$ realizes into $L^J[A]$ and is in fact $L^{J^*}[A^*]$.
\end{itemize}

If $\alpha$ is limit, let $\M_\alpha$ be the direct limit of the system $(\M_\beta,\pi^*_{\gamma,\beta})_{\gamma<\beta<\alpha}$, $\tau_\alpha = \lim_{\beta<\alpha} \tau_\beta$, and $\pi^*_{\beta,\alpha}, \pi_{\beta,\alpha}$ be natural direct limit maps.

Suppose $\alpha=\beta+1$ and $\M_\beta=(L_{\xi_\beta}^J[A],\mu_\beta)$ is an amenable structure, $\kappa_\beta = \textrm{crt}(\mu_\beta) = \pi_{0,\beta}(\kappa_0)$, and $\mu_\beta = \pi_{0,\beta}[\mu_0]$. Let $\pi^*_{\beta,\alpha},\pi_{\beta,\alpha}$ be $\mu_\beta$-ultrapower maps. For any $f\in L^J[A]$, 
\begin{center}
$\tau_\alpha(\pi_{\beta,\alpha}(f)(\kappa_\beta)) = \tau_\beta(f)(\kappa_\beta)$.
\end{center}
We need to check that $\tau_\alpha$ is elementary. This is equivalent to checking $\mu_\beta$ is derived from $\tau_\beta$, i.e.
\begin{equation}\label{measDerived}
C\in\mu_\beta \Leftrightarrow \kappa_\beta\in \tau_\beta(C).
\end{equation}
To see \ref{measDerived}, let $\nu<\mu_0$ and $W = \mu_0\cap L^J[A]$, $f:\kappa_0\rightarrow \powerset(\kappa_0)\cap L^J_\nu[A]$. Let $c$ be a finite set of fixed points of $\tau_0$ and such that 
\begin{center}
$\forall \xi<\kappa_0 \ f(\xi) = \tau^{L^J[A]}[c](\xi)\cap \kappa_0$.
\end{center}
So
\begin{equation}
L^J[A]\vDash \forall \xi < \kappa_0\ (\tau[c](\xi)\cap\kappa_0\in W \Leftrightarrow \kappa_0\in \tau[c](\xi)).
\end{equation}
This fact is preserved by $\pi_{0,\beta}$ and gives \ref{measDerived}.

One can also show by induction that crt$(\tau_\gamma)=\kappa_\gamma$ for all $\gamma\leq \beta$. This is because $\kappa_\gamma$ is the only generator of $\mu_\gamma$.

So $\M_0^{J,\sharp}$ exists (and is $(\lambda^+,\lambda^+)$-iterable by Lemma \ref{LiftOperators}).
\end{proof}

\begin{lemma}
\label{M1sharpExists}
Suppose $A$ is a bounded subset of $\lambda^+$. Then $\M_1^{J,\sharp}(A)$ exists and is $(\lambda^+,\lambda^+)$-iterable.
\end{lemma}
\begin{proof}

It suffices to show $\M_1^{J,\sharp}(a)$ exists for $a$ a bounded subset of $\kappa^+$ (with $a$ coding $x$). Fix such an $a$ and suppose not. Then the Jensen-Steel core model (cf. \cite{jensenSteelCoreModel}) $K^J(a)$ exists\footnote{By our assumption and the fact that $J$ condenses finely, $K^{c,J}(a)$ (constructed up to $\lambda^+$) converges and is $(\lambda^+,\lambda^+)$-iterable. See Lemma \ref{lem:le_converges}.}. Let $\gamma\geq \kappa^+$ be a successor cardinal in $K^{J}(a)$. Since $\lambda^+ = \textrm{o}(K^J(a)) > \kappa^+$ and is a limit of cardinals in $K^J(a)$ (by the proof of \ref{BasicFacts}), we can take $\gamma<\lambda^+$. Weak covering (cf. \cite[Theorem 1.1 (5)]{jensenSteelCoreModel}) gives us,
\begin{equation}\label{covering}
\rm{cof}(\gamma) \geq |\gamma|\geq \kappa^+.
\end{equation}

Let $\vec{C}$ be a $\square$-sequence in $K^J(a)$ witnessing $\square(\gamma)$. By a standard argument, one can constructs from $\vec{C}$ a sequence witnessing $\square(\rm{cof}(\gamma))$; but cof$(\gamma)\geq \kappa^+$ by \ref{covering} and $\neg \square(\rm{cof}(\gamma))$. Contradiction.
\end{proof}

Lemmas \ref{LiftOperators}, \ref{LiftOperatorsToGenExt}, and \ref{M1sharpExists} allow us to extend $\M_1^{J,\sharp}$ to $H_{\lambda^+}^{V[G]}$.
\begin{lemma}\label{LiftM1Sharp}
Suppose $J$ is defined on a cone above some $a\in H_{\kappa^+}^V$ in $H_{\omega_1}^{V[G]}$ as in case 1 of \ref{cmi operator}. Then for every $b\in H^{V[G]}_{\lambda^+}$ coding $a$, $\M_1^{J,\sharp}(b)$ is defined (and is $(\lambda^+,\lambda^+)$-iterable in $V[G]$). Otherwise, letting $(\Q,\Lambda),x,\Gamma$ be as in 2 of Definition \ref{cmi operator} and $J=\Fop_{\Lambda,\varphi_{\rm{all}}}$, then $\M_1^{J,\sharp}(a)$ is defined for all $a\in H_{\lambda^+}^{V[G]}$ coding $x,\Q$. Furthermore, these operators determine themselves on generic extensions if $J$ does.
\end{lemma}

\begin{comment}
We end this section with a lemma that was used in the proof of Theorem \ref{the cmi theorem}.
\begin{lemma}\label{closeUnderLp}
Suppose $A\subseteq H_{\omega_3}$ is transitive and $A$ codes $H_{\omega_3}$ and the canonical $Col(\omega,\omega_2)$-name for $\mathbb{R}^{V[G]}$. Let $X\in S$ contain all relevant objects and $X$ is cofinal in Lp$^{^\gTheta\Fop}(A)$ and $\sigma = \mathbb{R}^{V[G]}\cap X[G]$. Then Lp$^{\gTheta\Fop}(\sigma)\subseteq M_X[G]$.
\end{lemma}
\begin{proof}
Let $\bar{A} = \pi_X^{-1}(A)$ and $\lambda_X = \pi_X^{-1}(\omega_3)$. So by assumption on $X$, Lp$^{\gTheta\Fop}(\bar{A})\subset M_X$. Now by elementarily, $\bar{A}$ codes $H^{M_X}_{\lambda_X}$ and a $Col(\omega,\omega_2)$-name for $\sigma$. By $S$-constructions (see \cite{ATHM} and \cite{schindler2009self}), Lp$^{^\gTheta\Fop}(\sigma)\subset M_X[G]$ because Lp$^{^\gTheta\Fop}(\sigma)$ is determined by $\bar{A}, G$, and the initial segments of Lp$^{^\gTheta\Fop}(\bar{A})$ that are iterable in $V[G]$.
\end{comment}

\subsection{THE CORE MODEL INDUCTION THEOREM}

Let $(\P^*,\Sigma), \mathcal{F}$ be as in the previous section. When an $^g\Fop$-premouse $\P$ is $1$-$\Fop$-$\Gamma$-suitable, we simply say $\P$ is $\Gamma$-suitable if $\Fop$ is clear from the context. Recall that under $\sf{AD}$, if $X$ is any set then $\theta_X$ is the least ordinal which isn't a surjective image of $\mathbb{R}$ via an $\mathrm{OD}_X$ function. 

The following is an outline of the proof of the core model induction theorem. We will follow the standard convention and use upper-case Greek letters $\Gamma,\Omega$ etc. to denote lightface pointclasses, bold upper-case Greek letters $\bf{\Gamma}, \bf{\Omega}$ etc. to denote boldface pointclasses. Given a point class $\bf{\Gamma}$, we let $\hat{\bf{\Gamma}}$ denote the dual pointclass of $\bf{\Gamma}$ and $\bf{\Delta_\Gamma}$ denote the pointclass $\bf{\Gamma}\cap \hat{\bf{\Gamma}}$. For more on the envelope $\bf{Env}(\Gamma)$, the notion of $C_\Gamma$ and other relevant descriptive set theoretic notions, see \cite{wilson2012contributions}.

We refer the reader to \cite{trang2013} for the scales analysis in Lp$^{^\gTheta\Fop}(\mathbb{R},\Fop\rest\mathbb{R})$ that we use in the proof of Theorem \ref{the cmi theorem}. We recall some notions which are obvious generalizations of those in \cite{CMI} and \cite{wilson2012contributions}. The following definitions take place in $V[G]$.

\begin{definition}
We say that the coarse mouse witness condition $W^{*,^\g\Fop}_\gamma$ holds if, whenever $U\subseteq \mathbb{R}$ and both $U$ and its complement have scales in Lp$^{^\gTheta\F}(\mathbb{R},\Fop\rest\mathbb{R})|\gamma$, then for all $k< \omega$ and $x \in \mathbb{R}$ there is a coarse $(k,U)$-Woodin $^\g\Fop$-mouse\footnote{This is the same as the usual notion of a $(k,U)$-Woodin mouse, except that we demand the mouse is closed under $^\g\Fop$.} containing $x$ with an $(\omega_1 + 1)$-iteration $^\g\Fop$-strategy whose restriction to $H_{\omega_1}$ is in Lp$^{^\gTheta\F}(\mathbb{R},\Fop\rest\mathbb{R})|\gamma$.
\end{definition}
\begin{remark}
By the proof of \cite[Lemma 3.3.5]{CMI}, $W^{*,^\g\Fop}_\gamma$ implies Lp$^{\gTheta\F}(\mathbb{R},\F\rest\mathbb{R})|\gamma \vDash \sf{AD}$.
\end{remark}
\begin{definition}
An ordinal $\gamma$ is a critical ordinal in Lp$^{^\gTheta\F}(\mathbb{R},\Fop\rest\mathbb{R})$ if there is some $U \subseteq \mathbb{R}$
such that $U$ and $\mathbb{R} \backslash U$ have scales in Lp$^{^\gTheta\F}(\mathbb{R},\Fop\rest\mathbb{R})|(\gamma + 1)$ but not in Lp$^{^\gTheta\F}(\mathbb{R},\Fop\rest\mathbb{R})|\gamma$. In other words, $\gamma$
is critical in Lp$^{^\gTheta\F}(\mathbb{R},\Fop\rest\mathbb{R})$ just in case $W^{*,^\g\Fop}_{\gamma+1}$ does not follow trivially from $W^{*,^\g\Fop}_{\gamma}$.
\end{definition}
\begin{definition}
Let $\rm{sLp}$$^{^\gTheta\F}(\mathbb{R},\Fop\rest\mathbb{R})$ be the initial segment of $\rm{Lp}$$^{^\gTheta\F}(\mathbb{R},\Fop\rest\mathbb{R})$ that is the union of all $\M\lhd \rm{Lp}$$^{^\gTheta\F}(\mathbb{R},\Fop\rest\mathbb{R})$ such that every countable $\M^*$ embeddable into $\M$ has an iteration strategy in $\M$.
\end{definition}
We will prove in the next theorem that $\rm{sLp}$$^{^\gTheta\F}(\mathbb{R},\Fop\rest\mathbb{R})\vDash \sf{AD}^+$; in fact, this is the maximal model of $\sf{AD}^+ + \Theta=\theta$$_\Sigma$ in light of \cite{DMATM}[Theorem 17.1]. We note that 
\begin{center}
$\powerset(\mathbb{R})\cap L(\rm{sLp}$$^{^\gTheta\F}(\mathbb{R},\Fop\rest\mathbb{R}))=\powerset(\mathbb{R})\cap \rm{sLp}$$^{^\gTheta\F}(\mathbb{R},\Fop\rest\mathbb{R})$ 
\end{center}
but don't know if $\rm{sLp}$$^{^\gTheta\F}(\mathbb{R},\Fop\rest\mathbb{R})=\rm{Lp}$$^{^\gTheta\F}(\mathbb{R},\Fop\rest\mathbb{R})$ in general.
\begin{theorem}\label{the cmi theorem} Assume the hypothesis of Theorem \ref{main_technical_theorem} and $(\dag)$. Suppose $(\P, \Sigma)$ is a hod pair below $\kappa$ and $\Sigma$ is a $(\lambda^+,\lambda^+)$-strategy in $V[G]$ with branch condensation. Let $\F$ be the corresponding operator (i.e. $\Fop = \Fop_{\Sigma,\varphi_{\rm{all}}}$). Suppose $\F\rest \mathbb{R}$ is self-scaled. Then in $V[G]$,  $\rm{sLp}$$^{^\gTheta\F}(\mathbb{R},\Fop\rest\mathbb{R})\models \textsf{\sf{AD}$$}^+ + \theta_\Sigma=\Theta$. Hence, $\powerset(\mathbb{R})\cap \rm{sLp}^{^\gTheta\F}(\mathbb{R},\Fop\rest\mathbb{R}) \subseteq \Omega$.
%\begin{enumerate}
%\item 
%\item There is $A\subseteq \mathbb{R}$ such that $\Sigma^g\in L(A, \mathbb{R})$, $L(A, \mathbb{R})\models \sf{AD}$$^++\textsf{MC}(\Sigma^g)+\theta_{\Sigma^g}<\Theta$ and $(\textrm{Lp}^\Sigma(\mathbb{R}))^{L(A, \mathbb{R})}\insegeq \S_{\k, g}^{\l, \Sigma}$.
%\end{enumerate}
\end{theorem}
\begin{proof}[Proof sketch]
As shown in subsection \ref{M1sharp}, our hypothesis implies that for every $\Sigma$-cmi operator $H$, $\M_1^{H,\sharp}$ exists and can be extended to $H_{\lambda^+}^{V[G]}$; furthermore, these operators determine themselves on generic extensions. We will use this fact and refer the reader to subsection \ref{M1sharp} for the proof. Working in $V[G]$, let $\alpha$ be the strict supremum of the ordinals $\gamma$ such that
\begin{enumerate}
\item the coarse mouse witness condition $W^{*,^\g\F}_{\gamma+1}$ holds\footnote{This is defined similarly to $W^*_{\gamma+1}$ but relativized to the operator $^\g\F$. Similarly, we can also define the fine-structural mouse witness condition $W^{^\g\Fop}_{\gamma+1}$.};
\item $\gamma$ is a critical ordinal in $\textrm{sLp}^{^\gTheta\F}(\mathbb{R},\Fop\rest\mathbb{R})$ (i.e. $\gamma+1$ begins a gap in $\textrm{sLp}^{^\gTheta\F}(\mathbb{R},\Fop\rest\mathbb{R})$).
\end{enumerate}
Using the fact that $\M_1^{H,\sharp}$ exists for every $\Sigma$-cmi operator $H$, it's easily seen that $\alpha$ is a limit ordinal. By essentially the same proof, with obvious modifications, as that in \cite{PFA}, we can advance past inadmissible gaps and admissible gaps in $\textrm{sLp}^{^\gTheta\F}(\mathbb{R},\Fop\rest\mathbb{R})|\alpha$. In each case, say $[\gamma,\xi]$ is a gap in $\textrm{sLp}^{^\gTheta\F}(\mathbb{R},\Fop\rest\mathbb{R})|\alpha$ and $W^{*,^\g\F}_\gamma$, the proof in \cite[Sections 1.4, 1.5]{PFA} and the scales analysis in \cite{trang2013}\footnote{If $\gamma$ is such that $M = \rm{sLp}^{^\gTheta\F}(\mathbb{R},\F\rest\mathbb{R})|\gamma$ is inadmissible, then $M$ is passive. Then \cite{trang2013} gives us that $\bf{\Sigma}$$_1^M$ has the scales property assuming $M\vDash \textsf{AD}$. This is the main reason why we analyze scales in $\Theta$-g-organized premice; if $M$ were g-organized, it could be that $\dot{B}^M\neq \emptyset$ and the argument in \cite{trang2013} does not seem to give us the scales property of $\bf{\Sigma}$$_1^M$ from $\sf{AD}^M$.} allow us to construct a nice operator $F$ on a cone above some $x$ in $HC^{V[G]}$ such that $x\in H_{\omega_3}^V$, $F\rest V\in V$. By the previous subsection, we can extend $F$ to a nice operator on $H_{\kappa^+}^{V[G]}$ (also called $F$) with $F\rest V\in V$. Again, by the previous subsection, we can construct a sequence of nice operators $(F_n : n<\omega)$, where $F_0 = F$, $F_{n+1} = \M_1^{F_n,\sharp}$, and these operators witness $W^{*,^\g\F}_{\xi+1}$ (or $W^{*,^\g\F}_{\xi+2}$ if the gap is strong).

Hence, the (lightface) pointclass $\Gamma = \Sigma_1^{\textrm{sLp}^{^\gTheta\F}(\mathbb{R},\Fop\rest\mathbb{R})|\alpha}$ is inductive-like and $\bf{\Delta}_\Gamma$$ = \powerset(\mathbb{R})\cap \textrm{sLp}^{^\gTheta\F}(\mathbb{R},\Fop\rest\mathbb{R})|\alpha$. Since $\Gamma$ is inductive-like and $\bf{\Delta}_\Gamma$ is determined, $\bf{Env}(\Gamma)$ is determined by Theorem 3.2.4 of \cite{wilson2012contributions}. Since whenever $\gamma$ is a critical ordinal in $\textrm{sLp}^{^\gTheta\F}(\mathbb{R},\Fop\rest\mathbb{R})$ and $W^{*,^\g\F}_{\gamma+1}$ holds then $\sf{AD}$ holds in $\textrm{sLp}^{^\gTheta\F}(\mathbb{R},\Fop\rest\mathbb{R})|(\gamma+1)$, we have that $\sf{AD}$ holds in $\textrm{sLp}^{^\gTheta\F}(\mathbb{R},\Fop\rest\mathbb{R})|\alpha$. 

Now we claim that $\bf{Env}(\Gamma)$$ = \powerset(\mathbb{R})\cap \textrm{sLp}^{^\gTheta\F}(\mathbb{R},\Fop\rest\mathbb{R})$. This implies $\textrm{sLp}^{^\gTheta\F}(\mathbb{R},\Fop\rest\mathbb{R})\vDash \sf{AD}$$^+ + \Theta = \theta_\Sigma$ as desired. We first show $\bf{Env}(\Gamma)$$\subseteq \powerset(\mathbb{R})\cap \textrm{sLp}^{^\gTheta\F}(\mathbb{R},\Fop\rest\mathbb{R})$. Let $A\in \bf{Env}(\Gamma)$, say $A\in Env(\Gamma)(x)$ for some $x \in \mathbb{R}$. By definition of $Env$, for each countable $\sigma\subseteq \mathbb{R}$, $A\cap \sigma  = A'\cap \sigma$ for some $A'$ that is $\Delta_1$-definable over $\textrm{sLp}^{^\gTheta\F}(\mathbb{R},\Fop\rest\mathbb{R})|\alpha$ from $x$ and some ordinal parameter. In $V$, let $\tau$ be the canonical name for $x$ and let $X\prec H_{\lambda^{++}}$ be such that $|X|^V = \kappa$, $\kappa\subseteq X$, $X^\omega\subseteq X$, $X$ is cofinal in the ordinal height of sLp$^{^\gTheta\Fop}(\mathbb{R},\Fop\rest\mathbb{R})$.\footnote{Let $\epsilon$ be the name of $\mathbb{R}^{V[G]}$. Note that the ordinal height of sLp$^{^\gTheta\Fop}(\mathbb{R},\Fop\rest\mathbb{R})$ is the ordinal height of sLp$^{^\gTheta\Fop}(\epsilon,\Fop\rest\epsilon)$ and the latter is in $V$. This has cofinality at most $\omega_2$ since the construction in \cite{TrangCoherent} gives a coherent, nonthreadable sequence of length o$(\rm{sLp}^{^\gTheta\Fop}(\epsilon,\Fop\rest\epsilon))$.}  We also assume $\P\cup \{\P,\tau\} \subseteq X$, $\Sigma\cap M_X[G]\in M_X[G]$, and $A\in \textrm{ran}(\pi_X)$. Let $\sigma = \mathbb{R}\cap M_X[G]$. By $\textsf{MC}(\Sigma)$ in $\textrm{sLp}^{^\gTheta\F}(\mathbb{R},\Fop\rest\mathbb{R})|\alpha$, $A\cap \sigma \in \textrm{sLp}^{^\gTheta\F}(\sigma,\Fop\rest\sigma)$ (this is because $\textrm{sLp}^{^\gTheta\F}(\sigma,\Fop\rest\sigma)$ and $\textrm{sLp}^{^\g\F}(\sigma,\Fop\rest\sigma)$ have the same $\powerset(\sigma)$, cf. \cite[Section 5]{trang2013}; also, in applying $\sf{MC}(\Sigma)$, we need that definability is done without referencing the extender sequence and we can do this since we are inside sLp$^{^\gTheta\F}(\sigma,\Fop\rest\sigma)$, where the self-iterability condition helps us define the extender sequence). As shown in Lemma \ref{correctness}, $\textrm{sLp}^{^\gTheta\F}(\sigma,\Fop\rest\sigma)=(\textrm{sLp}^{^\gTheta\F}(\sigma,\Fop\rest\sigma))^{M_X[G]}$. Hence 
\begin{center}
$M_X[G] \vDash A\cap \sigma \in \textrm{sLp}^{^\gTheta\F}(\sigma,\Fop\rest\sigma)$.
\end{center}
By elementarity, the fact that $\pi_X(\P) = \P$, and $\pi_X(\Sigma\cap M_X[G])=\Sigma$, we get $A\in \textrm{sLp}^{^\gTheta\F}(\mathbb{R},\Fop\rest\mathbb{R})$.

Now assume toward a contradiction that $\bf{Env}(\Gamma)$$\subsetneq \powerset(\mathbb{R})\cap \textrm{sLp}^{^\gTheta\F}(\mathbb{R},\Fop\rest\mathbb{R})$. Hence $\alpha < \Theta^{\textrm{sLp}^{^\gTheta\F}(\mathbb{R},\Fop\rest\mathbb{R})}$. Let $\beta^*$ be the end of the gap starting at $\alpha$ in $\textrm{sLp}^{^\gTheta\F}(\mathbb{R},\Fop\rest\mathbb{R})$. Let $\beta = \beta^*$ if the gap is weak and $\beta = \beta^* + 1$ if the gap is strong. Note that $\alpha \leq \beta$, $\powerset(\mathbb{R})^{\textrm{sLp}^{^\gTheta\F}(\mathbb{R},\Fop\rest\mathbb{R})|\beta} = \bf{Env}(\Gamma)$$^{\textrm{sLp}^{^\gTheta\F}(\mathbb{R},\Fop\rest\mathbb{R})}\subseteq \footnote{We get equality in this case but we don't need this fact.}\bf{Env}(\Gamma)$$\subsetneq \powerset(\mathbb{R})\cap \textrm{sLp}^{^\gTheta\F}(\mathbb{R},\Fop\rest\mathbb{R})$. Hence $\beta < \Theta^{\textrm{sLp}^{g\F}(\mathbb{R})}$ and $\textrm{sLp}^{^\gTheta\F}(\mathbb{R},\Fop\rest\mathbb{R})|\beta$ projects to $\mathbb{R}$. Furthermore, $\textrm{Lp}^{^\gTheta\F}(\mathbb{R},\Fop\rest\mathbb{R})|\beta \vDash \sf{AD}$$ + \Gamma$-$\textsf{MC}(\Sigma)$, where $\Gamma$-$\sf{MC}$$(\Sigma)$ is the statement: for any countable transitive $a$, $(\textrm{Lp}^{^\g\F}(a))^\Gamma\cap \powerset(a)=C_\Gamma(a)$. $\textrm{sLp}^{^\gTheta\F}(\mathbb{R},\Fop\rest\mathbb{R})|\beta\vDash \Gamma$-$\textsf{MC}(\Sigma)$ is clear; if $\beta = \beta^*$, $\textrm{sLp}^{^\gTheta\F}(\mathbb{R},\Fop\rest\mathbb{R})|\beta \vDash \sf{AD}$ by the fact that $[\alpha,\beta^*]$ is a $\Sigma_1$-gap; otherwise, $\textrm{sLp}^{^\gTheta\F}(\mathbb{R},\Fop\rest\mathbb{R})|\beta\vDash \sf{AD}$ by Kechris-Woodin transfer theorem (see \cite{kechris1983equivalence}). Since $\textrm{sLp}^{^\gTheta\F}(\mathbb{R},\Fop\rest\mathbb{R})|\beta$ projects to $\mathbb{R}$, every countable sequence from $\bf{Env}(\Gamma)$$^{\textrm{sLp}^{^\gTheta\F}(\mathbb{R},\Fop\rest\mathbb{R})}$ is in $\textrm{sLp}^{^\gTheta\F}(\mathbb{R},\Fop\rest\mathbb{R})|(\beta+1)$. The scales analysis of  \cite{trang2013}, Theorem 4.3.2 and Corollary 4.3.4 of \cite{wilson2012contributions} together imply that there is a self-justifying-system $\mathcal{A} =\{ A_i \ | \ i<\omega\}\subseteq \bf{Env}(\Gamma)$$^{\textrm{sLp}^{^\gTheta\F}(\mathbb{R},\Fop\rest\mathbb{R})}$ containing a universal $\Gamma$ set. By a theorem of Woodin and the fact that $\Sigma\cap V\in V$, we can get a pair $(\N,\Lambda)$ such that $\N\in V$, $|\N|^V \leq \kappa$, $\N$ is a $\Gamma$-suitable $^\g\F$-premouse, and $\Lambda$ is the strategy for $\N$ guided by $\mathcal{A}$\footnote{We get first a pair $(\N^*,\Lambda^*)\in V[G]$ with $\N^*$ being $\Gamma$-suitable, $\Lambda^*$ is an $(\omega_1,\omega_1)$-strategy in $V[G]$ guided by $\mathcal{A}$. By boolean comparisons, we can obtain such a pair $(\N,\Lambda)$. The details are given in \cite{PFA} and \cite{CMI}.}. Arguments in the last subsection allow us to then lift $\Lambda$ to an $(\lambda^+,\lambda^+)$-strategy that condenses well in $V[G]$.

Using the hypothesis of Theorem \ref{main_technical_theorem}, we can get a sequence of nice operators $(F_n : n<\omega)$ where each $F_n$ is in $\textrm{sLp}^{^\gTheta\F}(\mathbb{R},\Fop\rest\mathbb{R})|(\beta+1)$. Namely, let $F_0 = \Fop_\Lambda$, and let $F_{n+1} = \M^{F_n,\sharp}_1$ be the $F_n$-Woodin $\Sigma$-cmi operator. Each $F_n$ is first defined on a cone in $H_{\omega_3}^V$; then using the lemmas in the previous subsection, we can extend $F_n$ to $H_{\kappa^+}^{V[G]}$ and furthermore, each $F_n$ is nice (i.e. condenses and relativizes well and determines itself on generic extensions).  These operators are all projective in $\mathcal{A}$ and are cofinal in the projective-like hierarchy containing $\mathcal{A}$, or equivalently in the Levy hierarchy of sets of reals definable from parameters over $\textrm{sLp}^{^\gTheta\F}(\mathbb{R},\Fop\rest\mathbb{R})|\beta$. Together these model operators can be used to establish the coarse mouse witness condition $W^{*,^\g\F}_{\beta+1}$.  Therefore $\beta < \alpha$ by the definition of $\alpha$, which is a contradiction.
\end{proof}

\subsection{BEYOND ``$\textsf{AD}^+ + \Theta=\Theta_\Sigma$"}
\label{theta>theta0}
Let $(\P^*,\Sigma), \Fop$ be as in Section \ref{M1sharp}. In this section, we prove.
\begin{theorem}
\label{GettingNontame}
Let $G\subseteq$ Col$(\omega,\kappa)$ be $V$-generic. Then in $V[G]$, there is a model $M$ such that $\rm{OR}$$\cup \mathbb{R}\subseteq M$ and $M\vDash ``\sf{AD}$$^+$$ + \Theta > \theta_\Sigma$."
\end{theorem}
The rest of this subsection is devoted to proving Theorem \ref{GettingNontame}. We assume $(\P^*,\Sigma) = (\emptyset,\emptyset)$ (the proof of the general case just involves more notations; in particular, for the general case, we work in the hierarchy Lp$^{^\gTheta\Fop}(\mathbb{R},\Fop\rest\mathbb{R})$ instead of in Lp$(\mathbb{R})$). Suppose the conclusion of the theorem fails. By the results of Subsection \ref{M1sharp} and Theorem \ref{the cmi theorem}, in $V[G]$
\begin{center}
sLp$(\mathbb{R}) \vDash \sf{AD}$$^+ + \Theta = \theta_0$.\footnote{We could have also worked with the hierarchy sLp$^{^\gTheta\F}(\mathbb{R},\F\rest\mathbb{R})$ where $\F$ is associated with the canonical strategy of $\M_1^\sharp$. As mentioned before, these hierarchies construct the same sets of reals.}
\end{center}
Working in $V[G]$, let $\bf{\Omega_0} = \powerset(\mathbb{R})\cap \textrm{sLp}(\mathbb{R})$, $\theta=\Theta^{\Omega_0}$, and $\Gamma = (\Sigma_1)^{\textrm{sLp}(\mathbb{R})}$. 

In $V$, let $\tau\in H_{\lambda^+}$ be a canonical name for $\mathbb{R}^{V[G]}$. Let $S$ be as in Section \ref{M1sharp}. For $X\in S$ such that $X$ is cofinal in o$(\textrm{Lp}(\mathbb{R}))$ and o$(\textrm{sLp}(\mathbb{R}))$,\footnote{One can construct a coherent sequence of length $\theta$ in sLp$(\mathbb{R}^{V[G]})$ as in \cite{TrangCoherent}. Our hypothesis and the properties of the sequence then imply that cof$(o(\textrm{sLp}(\mathbb{R}))\leq \kappa$. Similarly, one can show cof$(o(\textrm{Lp}(\mathbb{R}))\leq \kappa$}  let $\mathbb{R}_X = \pi_X^{-1}(\mathbb{R})$, $\bf{\Omega_X} = \pi_X^{-1}(\Omega_0)$, $\Gamma_X = \pi_X^{-1}(\Gamma)$, and $\theta_X = \pi_X^{-1}(\theta)$; we note that $\theta = \Theta^{\bf{\Omega_0}}=$ o$(\rm{sLp}(\mathbb{R})^{V[G]})$ is the supremum of the Wadge ranks of sets in $\bf{\Omega_0}$. Also, let $\eta = (\undertilde{\delta^2_1})^{\textrm{sLp}(\mathbb{R})}$ and $\eta_X = \pi_X^{-1}(\eta)$; so $\eta$ is the largest Suslin cardinal of sLp$(\mathbb{R})$. Let $T$ be the tree of scales for a universal $\Gamma$-set and $T_X = \pi_X^{-1}(T)$. As usual, $T$ is a tree on $\omega\times \eta$. For $s\in \omega^{<\omega}$, we let $T_s = \{ t \ | \ (s,t)\in T \}$.

Following \cite{wilson2012contributions}, we define.
\begin{definition}
Let $\Gamma, T, \eta$ be as above. 
\begin{enumerate}
\item $\powerset^{\bf{\Gamma}}(\eta)$ is the $\sigma$-algebra consisting of subsets $Y\subseteq \eta$ such that $Y\in L[T,z]$ for some real $z$.
\item meas$^{\bf{\Gamma}}(\eta)$ is the set of countably complete measures on $\powerset^{\bf{\Gamma}}(\eta)$.
\item Using the canonical bijection $\eta \rightarrow \eta^{<\omega}$, we can define $\powerset^{\bf{\Gamma}}(\eta^n)$, meas$^{\bf{\Gamma}}(\eta^n)$, $\powerset^{\bf{\Gamma}}(\eta^{<\omega})$, meas$^{\bf{\Gamma}}(\eta^{<\omega})$ in a similar fashion.
\end{enumerate}
\end{definition}
\begin{lemma}\label{correctness}
\label{MeasureFull}
Suppose $X\in S$. Then in $V[G]$, 
\begin{center}
$Lp(\mathbb{R}_X)^{\Omega_0}\subseteq Lp(\mathbb{R}_X)^{M_X[G]} = \textrm{Lp}(\mathbb{R}_X)$.
\end{center} 
This implies 
\begin{center}
$meas^{\bf{\Gamma_X}}(\eta_X^{<\omega})\subseteq (meas^{\bf{\Gamma_X}}(\eta_X^{<\omega}))^{M_X[G]}$. 
\end{center}
\end{lemma}
\begin{proof}
We first prove in $V[G]$,
\begin{center}
$\textrm{Lp}(\mathbb{R}_X)^{M_X[G]} = \textrm{Lp}(\mathbb{R}_X)$.
\end{center}
If $\M\lhd \textrm{Lp}(\mathbb{R}_X)^{M_X[G]}$ is a sound mouse that projects to $\mathbb{R}_X$, then $\M$ is embeddable into a level of Lp$(\mathbb{R})^{V[G]}$. So $\M\lhd \textrm{Lp}(\mathbb{R}_X)$. To see the converse, let $\M\lhd \textrm{Lp}(\mathbb{R}_X)$, then letting $\M^*\lhd \textrm{Lp}(\tau_X)$ be the $S$-translation of $\M$, then by Lemma \ref{LpClosed}, $\M^*\in M_X$; so $\M\in M_X[G]$. Now in $M_X[G]$, let $\pi:\N\rightarrow \M$ be an elementary embedding and $\N$ is countable, transitive; then in $V[G]$, $\N$ is iterable via a unique iteration strategy; so in $M_X[G]$, $\N$ is iterable via a unique iteration strategy. This means $\M\lhd \textrm{Lp}(\mathbb{R}_X)^{M_X[G]}$.

%Let $\M\lhd \textrm{Lp}(\mathbb{R}_X)^{M_X[G]}$, then by absoluteness, in Lp$(\mathbb{R})^{V[G]}$, there is an elementary embedding from $\M$ into some $\N\lhd \textrm{Lp}(\mathbb{R})^{V[G]}$; so $\M\lhd \textrm{Lp}(\mathbb{R}_X)^{\bf{\Omega_0}}$. 

To see the first inclusion, note that $\textrm{Lp}(\mathbb{R}_X)^{\bf{\Omega_0}}$ can easily be computed from $\textrm{Lp}(\tau_X)$ (using $G$ and the $S$-constructions) and $\textrm{Lp}(\tau_X) \in M_X$ by Lemma \ref{LpClosed}. This gives $\textrm{Lp}^{\bf{\Omega_0}}(\mathbb{R}_X)\subseteq \textrm{Lp}(\mathbb{R}_X)^{M_X[G]}$ because given any $\M\lhd \textrm{Lp}(\mathbb{R}_X)^{\bf{\Omega_0}}$, letting $\M^*$ be a premouse over $\tau_X$ which results from the $S$-constructions of $\M$. $\M^*$ has a unique strategy $\Lambda$ in Lp$(\mathbb{R})^{V[G]}$ because $\M$ does. So $\Lambda\rest V\in V$ by homogeneity. By the proof of Lemma \ref{LiftOperators}, $\Lambda$ can be uniquely extended in $V$ to an $(\lambda^+,\lambda^+)$-strategy (also called $\Lambda$)\footnote{Note that every tree $\T$ according to $\Lambda$ is short and guided by $\Q$-structures.}. The previous subsection also shows that $\M_1^{\Lambda,\sharp}$ exists and is $(\kappa^+,\kappa^+)$-iterable in $V$. By lemma \ref{LiftOperatorsToGenExt}, $\Lambda$ can be extended to an $(\lambda^+,\lambda^+)$ strategy in $V[G]$, but this means $\M\lhd \textrm{Lp}(\mathbb{R}_X)$ in $V[G]$.

We just prove $meas^{\bf{\Gamma_X}}(\eta_X)\subseteq (meas^{\bf{\Gamma_X}}(\eta_X))^{M_X[G]} $  as the proof of the general case is an easy generalization of the proof of the special case. We need to define ``code sets" for measures in $meas^{\bf{\Gamma_X}}(\eta_X)$. Fix a map from $\mathbb{R}_X$ onto $\powerset(\eta_X)^{\Gamma_X}$:  $x \mapsto Y_x$ in $M_X[G]$ such that the relation $\{(x,\alpha) \ | \ \alpha\in Y_x\} \in \Sigma_1^{\textrm{sLp}(\mathbb{R}_X)}$. We then define the code set $C_\mu$ for each $\mu\in meas^{\bf{\Gamma_X}}(\eta_X)$ as $x\in C_\mu \Leftrightarrow Y_x\in \mu$. For each such $\mu$, $C_\mu$ is easily seen to be $OD^{\Omega_0}(\mathbb{R}_X)$ (as each such measure is principal, being a countably complete measure on a countable set in $V[G]$), and so $C_\mu \in \textrm{Lp}(\mathbb{R}_X)^{M_X[G]}$ by $\textsf{MC}$ in $\Omega_0$ and by the first part. So $\mu\in M_X[G]$ and is countably complete there. This proves the second inclusion. %We note also that for every measure $\mu\in (meas^{\bf{\Gamma_X}}(\eta_X))^{M_X[G]}$, $C_\mu$ is also in Lp$(\mathbb{R}_X)$ by \cite{wilson2012contributions}. 
\end{proof}

For $X$ as in the lemma, we can choose a set $C_X\in M_X$ of canonical names for measures in $[meas^{\bf{\Gamma_X}}(\eta_X^{<\omega})]^{M_X[G]}$. Since $M_X^{\omega}\subset M_X$, $M_X$ contains all $\omega$-sequences of its terms for code sets of measures in $[meas^{\bf{\Gamma_X}}(\eta_X^{<\omega})]^{M_X[G]}$.

Now let $\sigma = \pi_X[[meas^{\bf{\Gamma_X}}(\eta_X^{<\omega})]^{M_X[G]}]\subseteq meas^{\bf{\Gamma}}(\eta^{<\omega})$. $\sigma$ is a countable set of measures in $V[G]$. For $\mu\in meas^{\bf{\Gamma}}(\eta^{<\omega})\cap \sigma$, let $\overline{\mu}$ be such that $\pi_X(\overline{\mu})=\mu$. Suppose $\mu$ concentrates on $\eta^n$ and let $\langle\mu_i \ | \ i\leq n \rangle$ be the projections of $\mu$ (that is $A\in \mu_i \Leftrightarrow \{ s\in \eta^n \ | \ s\rest i \in A \} \in \mu$). Note that $\mu_0$ is the trivial measure. Define $\langle\overline{\mu}_i \ | \ i\leq n\rangle$ similarly for $\overline{\mu}$. Let $G^{\sigma,\mu}_T$ be the game defined in Definition 4.1.2 of \cite{wilson2012contributions}. For the reader's convenience, we give the definition of $G^{\sigma,\mu}_T$. I starts by playing $m_0, \cdots, m_n$, $s_n, h_n$; II responds by playing a measure $\mu_{n+1}$. From the second move on, I plays $m_{i},s_i, h_i$ and II plays a measure $\mu_{i+1}$ for all $i > n$. \footnote{The game can be defined over $V$ using the forcing relation and $C_X$.}

\noindent \textbf{Rules for I:}
\begin{itemize}
\item $m_k < \omega$ for all $k<\omega$
\item $T_{(m_0,\cdots, m_{n-1})} \in \mu = \mu_n$
\item $s_i \in j_{\mu_i}(T_{(m_0,\cdots, m_{i-1})})$, in particular $s_i \in j_{\mu_i}(\eta)^{i+1}$ for all $i\geq n$
\item $s_n \supsetneq [\rm{id}]$$_{\mu_n}$
 \item $j_{\mu_i,\mu_{i+1}}(s_i) \subsetneq s_{i+1}$ for all $i\geq n$
\item $h_i \in \rm{OR}$ for all $i\geq n$
\item $j_{\mu_i,\mu_{i+1}}(h_i)>h_{i+1}$ for all $i\geq n$
\end{itemize}
\noindent \textbf{Rules for II:}
\begin{itemize}
 \item $\mu_{i+1}\in \sigma$ is a measure on $\eta^{i+1}$ projecting to $\mu_i$ for all $i\geq n$.
\item $\mu_{i+1}$ concentrates on $T_{(m_0,\cdots,m_i)}\subset \eta^{i+1}$.
\end{itemize}
The first player that violates one of these rules loses, and if both players follow the rules for all $\omega$ moves, then I wins. The game is closed, so is determined.

\begin{lemma}
\label{IIwins}
Player II has a winning strategy for $G^{\sigma,\mu}_T$ for all $\mu\in \sigma$.
\end{lemma}
\begin{proof}
Suppose for contradiction that I has a winning strategy in $G^{\sigma,\mu}_T$, that is if both players follow all the rules of the game, then I can continue playing for $\omega$ moves. Suppose I plays integers $m_0,\dots, m_n$ such that $T_{(m_0,\dots,m_{n-1})}\in\mu = \mu_n$, an $s_n \in j_{\mu_n}(\eta)^{n+1}$ such that $s_n\in j_{\mu_n}(T_{(m_0,\dots, m_n)})$ such that $[id]_{\mu_n}\subsetneq s_n$, and some $h_n\in \textrm{OR}$ on his first move. II then responds with $\mu_{n+1} = \pi_X(\overline{\mu}_{n+1})$, where
\begin{center}
$A\in \overline{\mu}_{n+1} \Leftrightarrow s_n \in j_{\mu_n}(\pi_X(A))$.
\end{center}
We have that $\overline{\mu}_{n+1}\in meas^{\bf{\Gamma_X}}(\eta_X^{<\omega})\in M_X[G]$ (so $\mu_{n+1}$ is defined). Similarly, suppose for $i>n$, I has played $(m_i,s_i,h_i)$ such that 
\begin{itemize}
\item $m_i\in \omega$,
\item $s_i\in j_{\mu_i}(T_{(m_0,\dots,m_{i})})$,
\item $j_{\mu_j,\mu_{j+1}}(s_j)\subsetneq s_{j+1}$ for $j < i$,
\item $h_i\in$OR,
\item $j_{\mu_j,\mu_{j+1}}(h_j) > h_{j+1}$ for all $j<i$.
\end{itemize}
II then responds with $\mu_{i+1} = \pi_X(\overline{\mu}_{i+1})$, where
\begin{center}
$A\in \overline{\mu}_{i+1} \Leftrightarrow s_i\in j_{\mu_i}(\pi_X(A))$.
\end{center}
Again, $\mu_{i+1}$ makes sense since $\bar{\mu}_{i+1}\in M_X[G]$. 

After $\omega$ many moves, the players play a real $x = (m_0,m_1, \dots)$, a tower of measures $(\mu_i \ | \ i<\omega)$, a sequence of ordinals $(h_i \ | \ n\leq i<\omega)$ witnessing the tower $(\mu_i \ | \ n\leq i<\omega)$ is illfounded, and the sequence $(s_i \ | \ n\leq i < \omega)$.  By closure of $M_X$ and the fact that we can find a canonical name for each $\bar{\mu}_i$ in $M_X$,\footnote{In fact, the definition of $\bar{\mu}_i$ only depends on $s_i$ and not on $G$. Furthermore, (the codeset of) $\bar{\mu}_i$ is also $OD$ in $M_X[G]$ from $T_X=\pi_X^{-1}(T)$ (so the $\bar{\mu}_i$'s have symmetric names in $M_X$); hence we can think of the game $G^{\sigma,\mu}_T$ as being defined in $V$ where player II plays finite sequences of ordinals in $X[G]$, which are $\pi_X$-images of the sequences of ordinals that define the $\bar{\mu}_i$'s in $M_X[G]$.} the sequence $(\bar{\mu}_i \ | \ i<\omega) \in M_X[G]$ (and so $(\mu_i \ | \ i<\omega) = \pi_X(\bar{\mu}_i \ | \ i<\omega)$). 

In $M_X[G]$, the tower $(\bar{\mu}_i \ | \ i<\omega)$ is illfounded. By Lemma 3.5.9 of \cite{wilson2012contributions}, there is a tree $\bar{W}\in L[T_X,x]$ for some $x\in \mathbb{R}^{V[G]}$ on $\omega\times \eta_X$ such that the $\bar{\mu}_i$'s concentrate on $\bar{W}$ and the function $\bar{h}(i) = [\textrm{rk}_{\bar{W}}]_{\bar{\mu}_i}$ is a pointwise minimal witness to the illfoundedness of $(\bar{\mu}_i \ | \ i<\omega)$. Let $h = \pi_X(\bar{h})$ and $W = \pi_X(\bar{W})$; since $\bar{\mu}_i$'s concentrate on $\bar{W}$, $s_i\in j_{\mu_i}(W)$ for all $i$.

Let $h'(i) = \textrm{rk}_{j_{\mu_i}(W)}(s_i)$. We have
\begin{enumerate}
\item $j_{\mu_i,\mu_{i+1}}(h'(i)) = j_{\mu_i,\mu_{i+1}}(\textrm{rk}_{j_{\mu_i}(W)}(s_i)) > h'(i+1)$ for all $i$ as $j_{\mu_i,\mu_{i+1}}(s_i)\subsetneq s_{i+1}$;
\item $h'(n) = \textrm{rk}_{j_{\mu_n}(W)}(s_n) < \textrm{rk}_{j_{\mu_n}(W)}([id]_{\mu_n}) = h(n)$ because $[id]_{\mu_n}\subsetneq s_n$.  
\end{enumerate} 
So $h'$ witnesses $(\mu_i \ | \ i<\omega)$ is illfounded and $h'(n)<h(n)$. Contradiction.
\end{proof}
Lemma \ref{IIwins} easily implies that for each $\mu\in meas^{\bf{\Gamma}}(\eta^{<\omega})$, there is a countable set of measures $\sigma\subset meas^{\bf{\Gamma}}(\eta^{<\omega})$ that stabilizes $\mu$ (in the sense of \cite[Section 4]{wilson2012contributions}). By a simple argument using $\textsf{DC}$ and the fact that if $\sigma$ stabilizes $\mu$ then any $\sigma'\supseteq \sigma$ stabilizes $\mu$, we get a countable $\tau\subset meas^\Gamma(\eta^{<\omega})$ that stabilizes every $\mu\in \tau$.

Knowing this, \cite[Sections 4.1, 4.3]{wilson2012contributions} constructs a self-justifying system $\mathcal{A}$ for $\bf{Env}(\Gamma) = \Omega_0$ in $V[G]$. Using the argument in \cite[Section 5.5]{CMI}, we can then find a pair $(\N,\Psi)$ such that $\N\in V$, $|\N|^V \leq \kappa$, $\N$ is $\Gamma$-suitable, and $\Psi$ is a $(\omega_1^{V[G]},\omega_1^{V[G]})$-strategy for $\N$ such that $\Psi$ is $\Gamma$-fullness preserving and has branch condensation (and hence hull condensation by results in \cite{ATHM}); furthermore, $\Psi\rest V\in V$. In fact $\Psi$ is guided by $\mathcal{A}$ and hence $\Psi\notin \Omega_0$. By the lemmas in Section \ref{M1sharp}, we can then extend $\Psi$ to a $(\lambda^+,\lambda^+)$-strategy in $V$ and further to an $(\lambda^+,\lambda^+)$ strategy in $V[G]$ (also called $\Psi$) that has branch condensation. Furthermore, results of the previous section allow us to construct operators $x\mapsto \M_n^{\Psi,\sharp}(x)$ for all $n<\omega$. This means $PD(\Psi)$ holds and since $\Psi$ is guided by a self-justifying system, we can conclude by standard methods that the operator $\Fop = \Fop_{\Psi,\varphi_{\rm{all}}}$ is self-scaled. This allows us to run a core model induction as before to show in $V[G]$
\begin{center}
sLp$^{^\gTheta\Fop}(\mathbb{R},\Fop\rest\mathbb{R}) \vDash ``L(\powerset(\mathbb{R})) \vDash  \sf{AD}$$^+ + \Theta = \theta_1$''.\footnote{Of course, what we showed in the previous section also shows $\M_1^{\Psi,\sharp}$ exists and is $(\lambda^+,\lambda^+)$-iterable.}
\end{center}
The above construction works in general and allows us to show that ``sLp$^{^\gTheta\Fop}(\mathbb{R},\Fop\rest\mathbb{R})\vDash\sf{AD}$$^+ + \Theta = \theta_{\alpha+1}$'' for any hod pair $(\P^*,\Sigma)\in \Omega$ below $\kappa$ and $\Sigma$ is $\Omega$-fullness preserving, and $\Sigma$ is projectively equivalent to a set of Wadge rank $\theta_\alpha$, where $\Fop=\Fop_{\Sigma,\varphi_{\rm{all}}}$. In other words, we have shown that the Solovay sequence of $\Omega$ is of limit length.
%\\
%\noindent \textbf{Remark:} The construction of Lp$^{\Sigma}(\mathbb{R})$ in $V[G]$ in \cite{trang2013} and \cite{trangThesis2013} appears to use more iterability of $\Sigma$ than we have. In defining Lp$^{\Sigma}(\mathbb{R})$, \cite{trang2013} and \cite{trangThesis2013} use the unique strategy (in $V$) $\Lambda$ of $\M_1^{\Sigma,\sharp}$ to make various levels Lp$^\Sigma(\tau)|\alpha$ ``generically generic'', where $\tau \subseteq \omega_3^V$ is the canonical $Col(\omega,\omega_2)$ name for $\mathbb{R}^{V[G]}$; this appears to need $\omega_4+1$ iterability of $\Lambda$. But here we can run that argument in $M = L_{\omega_4}^{\Lambda}[\M_1^{\Sigma,\sharp},\textrm{Lp}^\Sigma(\tau)|\alpha]$, where $\omega^V_4 > \omega_4^M$ (by $\neg \square_{\omega_3}$). Since $M$ is closed under $\Lambda$, the genericity iteration can be carried out in $M$ and since $\omega_4^M < \omega_4^V$, it in fact terminates successfully in $M$. 

%\end{proof}

\subsection{GETTING $\sf{AD}$$_\mathbb{R} + \Theta$ IS REGULAR}
\label{thetareg}
In the previous subsections, we show that the core model induction cannot stop at successor stages, i.e. in $V[G]$, $L(\Omega,\mathbb{R})$ cannot satisfy $\textsf{AD}^+ + \Theta = \theta_{\alpha+1}$ for some $\alpha\geq -1$.\footnote{$\Omega$ can also be characterized as the set of all $A\subseteq \mathbb{R}$ such that $A$ is Wadge reducible to a $\Sigma$-cmi-operator $J$ that determines itself on generic extensions, for some hod pair $(\P,\Sigma)$ below $\omega_2$.} This means the Solovay sequence of $\Omega$ is of limit order type. In this subsection, we show that there is some Solovay initial segment $\Omega^*$ of $\Omega$ such that $L(\Omega^*,\mathbb{R})\vDash ``\textsf{AD}_\mathbb{R} + \Theta$ is regular." This contradicts $(\dag)$. So we get after all that there is a model of ``$\textsf{AD}_\mathbb{R} + \Theta$ is regular".

%We first consider the case cof$^{V[G]}(\Theta^\Omega) = \omega$. Let $(\alpha_n \ | \ n<\omega)$ be cofinal in $\Theta^\Omega$, where each $\alpha_n$ is a successor in the Solovay sequence of $\Omega$. Let $(\P_n,\Sigma_n)$ be a hod pair below $\omega_2$ such that $\Gamma(\P_n,\Sigma_n)=\Omega|\alpha_n$. By comparison, we can obtain a hod pair below $\omega_2$ $(\P,\Sigma)$ such that $\Gamma(\P,\Sigma)=\Omega$. Since $\Sigma$ has branch condensation, we can show $L^\Sigma(\mathbb{R})\vDash \sf{AD}$$^+$. But $\Sigma\notin \Omega$. Contradiction.

%Now we consider the case cof$^{V[G]}(\Theta^\Omega) > \omega$. Let $\mathcal{H} = \textrm{HOD}^\Omega$. 
Let $\langle\theta_\alpha^\Omega \ | \ \alpha\leq \lambda\rangle$ be the Solovay sequence of $\Omega$. We write $\theta_\beta$ for $\theta_\beta^\Omega$ and $\Theta$ for $\Theta^\Omega$ and let $\alpha = \textrm{cof}^{V[G]}(\Theta)$. Note that $\lambda$ is limit and for each $\beta < \lambda$, $L(\Omega\rest\theta_\beta,\mathbb{R})\cap \powerset(\mathbb{R})=\Omega\rest\theta_\beta$. Note also that $\Theta < \kappa^{++}$ since otherwise, we've already reached a model of ``$\sf{AD}$$_\mathbb{R} + \Theta$ is regular'' by the following lemma. 
\begin{lemma}\label{ThetaOmega4}
Suppose $\Theta = \kappa^{++}$. Then in $V[G]$, $\Omega = \powerset(\mathbb{R})\cap L(\Omega,\mathbb{R})$. Consequently, $L(\Omega,\mathbb{R})\vDash ``\sf{AD}_\mathbb{R} + \Theta$ is regular."
\end{lemma}
\begin{proof}
\indent Suppose not. Let $\alpha$ be the least such that $\rho_\omega(L_\alpha(\Omega,\mathbb{R})) = \mathbb{R}$. Hence $\alpha \geq \Theta$ by our assumption. Let $f: \alpha \times \Omega \twoheadrightarrow L_\alpha(\Omega,\mathbb{R})$ be a surjection that is definable over $L_\alpha(\Omega,\mathbb{R})$ (from parameters). 

We first define a sequence $\langle H_i \ | \ i<\omega\rangle$ as follows. Let $H_0 = \mathbb{R}$. By induction, suppose $H_n$ is defined and there is a surjection from $\mathbb{R} \twoheadrightarrow H_n$. Suppose $(\psi, a)$ is such that $a\in H_n$ and $L_\alpha(\Gamma,\mathbb{R}) \vDash \exists x\psi[x,a]$. Let $(\gamma_{a,\psi},\beta_{a,\psi})$ be the $<_{lex}$-least pair such that there is a $B\in \Gamma$ with Wadge rank $\beta_{a,\psi}$ such that
\begin{center}
$L_\alpha(\Omega,\mathbb{R}) \vDash \psi[f(\gamma_{a,\psi},B),a]$.
\end{center}
Let then $H_{n+1} = H_n\cup \{f(\gamma_{a,\psi},B) \ | \ L_\alpha(\Omega,\mathbb{R}) \vDash \exists x\psi[x,a]\wedge w(B) = \beta_{a,\psi}\wedge a\in H_n\}$. It's easy to see that there is a surjection from $\mathbb{R}\twoheadrightarrow H_{n+1}$. This uses the fact that $\Theta^\Omega = \Theta$ is regular, which implies sup$\{\beta_{a,\psi} \ | \ a\in H_n \wedge L_\alpha(\Omega,\mathbb{R})\vDash \exists x\psi[x,a]\} < \Theta$. Let $H = \bigcup_n H_n$. By construction, $H\prec L_\alpha(\Omega,\mathbb{R})$. Finally, let $M$ be the transitive collapse of $H$. 
\\
\indent Say $M = L_\beta(\Omega^*,\mathbb{R})$. By construction, $\Omega^* = \Omega\rest \theta_\gamma$ for some $\gamma$ such that $\theta_\gamma < \Theta$. But then $\rho_\omega(L_\beta(\Omega^*,\mathbb{R})) = \mathbb{R}$. This contradicts that $\Omega^*$ is constructibly closed. This gives $\Omega = \powerset(\mathbb{R})\cap L(\Omega,\mathbb{R})$ and in fact, $L(\Omega,\mathbb{R})\vDash \textsf{AD}_\mathbb{R} + \Theta$ is regular.

\end{proof}

Now let $\mathcal{H}$ be the direct limit of all hod pairs $(\Q,\Lambda)\in \Omega$ such that $\Lambda$ is $\Omega$-fullness preserving and has branch condensation. For each $X\in S$ such that $\{\mathcal{H}\}\subseteq X$, $\Omega\in X[G]$, let $(\Omega_X,\mathcal{H}_X,\Theta_X,\alpha_X) = \pi_X^{-1}(\Omega,\mathcal{H},\Theta,\alpha)$. For each $\beta < \lambda'_X$, where $\lambda'_X$ is the order-type of the closure of the set of Woodin cardinals in $\mathcal{H}_X$, let $\Sigma_{X,\beta}$ be the canonical strategy for $\mathcal{H}_X(\beta)$, which is the tail of a hod pair $(\Q,\Lambda)$ below $\kappa$ (in $M_X[G]$) and $\mathcal{H}_X(\beta)$ is the direct limit of all $\Lambda$-iterates in $M_X[G]$. The fact that $\mathcal{H}$ ($\mathcal{H}_X$, respectively) is the direct limit of hod mice in $\Omega$ follows from our smallness assumption $(\dag)$ and the remarks after it. Then $(H_X(\beta),\Sigma_{X,\beta})$ is a hod pair below $\kappa$. Let $\Sigma_X^- = \oplus_{\beta<\lambda_X}\Sigma_{X,\beta}$ and 
\begin{center}
$\mathcal{H}_X^+ = [\textrm{Lp}^{^\g\Sigma_X^-}(\mathcal{H}_X)]^{\Omega}$. \footnote{Recall that our convention is: $\mathcal{H}_X(\alpha+1)$ is a g-organized $\Sigma_{\mathcal{H}_X(\alpha)}=\Sigma_{X,\alpha}$-mouse for each $\alpha<\lambda_X$. In general, hod mice in this paper are g-organized.}
\end{center}
Finally, let $\mathcal{H}^+$ be the union of all $\M$ such that $\M$ is sound, $\mathcal{H}\lhd \M$, $\rho_\omega(\M)\leq o(\mathcal{H})$, and whenever $\M \in X\in S$, then $\pi_X^{-1}(\M)\lhd \mathcal{H}_X^+$. 
\begin{lemma}
\label{NotProjectAcross}
\begin{enumerate}
\item $\forall^* X\in S \ \mathcal{H}_{X}^+ \in M_X$ and $\pi_X(\mathcal{H}_X^+) = \mathcal{H}^+$.
\item Let $X$ be as in 1). Then no levels of $\mathcal{H}_X^+$ project across $\Theta_X$.
\end{enumerate}
\end{lemma}
\begin{proof}
To prove 1), note that cof$^{V}(o(\mathcal{H}^+))\leq \kappa$. To see this, first note that $o(\mathcal{H}^+) < \kappa^{++}$; this follows from $\neg \square_{\kappa^+}$. We can then rule out cof$^{V}(o(\mathcal{H}^+))=\kappa^+$ using $\neg \square(\kappa^+)$ (since otherwise, the $\square$-sequence constructed in $\mathcal{H}^+$ of length o$(\mathcal{H}^+)$ gives rise to an nonthreadable coherent sequence witnessing $\square(\kappa^+)$, contradicting $\neg \square(\kappa^+)$). This means, $\forall^* X\in S$, the range of $\pi_X$ is cofinal in $o(\mathcal{H}^+)$. This implies 
\begin{center}
$\forall^* X\in S$, $\mathcal{H}^+_X \unlhd\pi_X^{-1}(\mathcal{H}^+)\in X$.
\end{center} 
Otherwise, fix such an $X$ and let $\M_X \lhd \mathcal{H}^+_X$ be such that $\M_X\notin \pi_X^{-1}(\mathcal{H}^+)$, $o(\M_X)$ has cofinality $\omega$, $\rho_1(\M_X)=\Theta_X$ and let $\M = \textrm{Ult}_0(\M_X,E_X)$ where $E_X$ is the $(\lambda_X,\Theta)$-extender derived from $\pi_X$. Since $\pi_X$ is cofinal in $o(\mathcal{H}^+)$, $\M\notin \mathcal{H}^+$. But whenever $Y\in S$ is such that $\M\in Y$, it's easy to see that $\pi_Y^{-1}(\M)\in \mathcal{H}_Y^+.$\footnote{Note that $\pi_Y^{-1}(\M) = \textrm{Ult}_0(\M_X,E_{X,Y})$ where $E_{X,Y}$ is derived from $\pi_{X,Y}$ the same way $E_X$ is derived from $\pi_X$. Let $Z\prec H_{\omega_6}$ be countable and contain all relevant objects and $\pi:M\rightarrow Z$ be the uncollapse map. Write $a^*$ for $\pi^{-1}(a)$ for $a\in Z$. Then it's easy to see using countable completeness of $E_{X,Y}$ that $(\pi_Y^{-1}(\M))^*$ is embeddable into $\M_X$, which in turns gives $(\pi_Y^{-1}(\M))^*\lhd ((\mathcal{H}_Y)^+)^*$.} So $\M\in \mathcal{H}^+$ after all. Contradiction. 

Suppose equality fails. By pressing down, there is some $\M\lhd \mathcal{H}^+$, some stationary set $T\subseteq S$ such that for $X\in T$, $\M\notin \pi_X(\mathcal{H}_X^+)$. But $\forall^* X \ \pi^{-1}_X(\M)\lhd \mathcal{H}_X^+$ by definition of $\mathcal{H}^+$. Contradiction. This completes the proof of 1).

To prove 2), suppose for contradiction that there is a $\P\lhd \mathcal{H}_X^+$ such that $\rho_\omega(\P)<\Theta_X$. Let $\P$ be the least such. By 1), $\P\in M_X$. Let $\beta<\lambda'_X$ be least such that $\rho_\omega(\P)\leq \delta_\beta^\P$ and $\delta_\beta > \textrm{cof}^\P(\lambda^\P)$. $\P$ can be considered a hod premouse over $(\mathcal{H}_X(\beta),\Sigma_{X,\beta})$. Using $\pi_X$, we can define a strategy $\Lambda$ for $\P$ such that $\Lambda$ acts on stacks above $\delta_\beta^\P$ and extends $\oplus_{\alpha<\lambda_X}\Sigma_{X,\alpha}$ (the strategy is simply $\oplus_{\alpha<\lambda_X}\Sigma_{X,\alpha}$ for stacks based on $\mathcal{H}_X$ (above $\delta^\P_\beta$), but the point is that it also acts on all of $\P$ because of $\pi_X$). By a core model induction similar to the previous subsections using the fact that $\Lambda$ has branch condensation and noting that $\Lambda$ can be extended to $H_{\kappa^+}^{V[G]}$, we can show $L^{^\gTheta\Fop}(\mathbb{R})\vDash \textsf{AD}^+$, where $\Fop=\Fop_{\Lambda,\varphi_{\rm{all}}}$, and hence $L(\Lambda\rest\rm{HC}$$,\mathbb{R})\vDash \textsf{AD}^+$. This implies $Code(\Lambda\rest\rm{HC})\in \Omega$ by definition of $\Omega$.

In $\Omega$, let $\mathcal{F}$ be the direct limit system of $\Sigma_{X,\beta}$ hod pairs $(\Q,\Psi)$ Dodd-Jensen equivalent to $(\P,\Lambda)$. $\mathcal{F}$ can be characterized as the direct limit system of $\Sigma_{X,\beta}$ hod pairs $(\Q,\Psi)$ in $\Omega$ such that $\Psi$ is $\Gamma(\P,\Lambda)$-fullness preserving and has branch condensation and $\Gamma(\Q,\Psi)=\Gamma(\P,\Lambda)$. $\mathcal{F}$ only depends on $\Sigma_{X,\beta}$ and the Wadge rank of $\Gamma(\P,\Lambda)$ and hence is $OD^{L(\mathbb{R},C)}_{\Sigma_{X,\beta}}$ for some $C\in\Omega$. 

Fix such a $C$ and note that $L(\mathbb{R},C)\vDash \sf{AD}^+ + \sf{SMC}$. Let $A\subseteq \delta^\P_\beta$ witness $\rho_\omega(\P)\leq \delta_\beta^\P$, that is, there is a formula $\phi$ such that for all $\alpha\in \delta^\P_\beta$,
\begin{center}
$\alpha\in A \Leftrightarrow \P \vDash \phi[\alpha,p]$,
\end{center}
where $p$ is the standard parameter of $\P$. Now $A$ is $OD_{\Sigma_{X,\beta}}$ in $L(\mathbb{R},C)$; this is because letting $\M_\infty$ be the direct limit of $\mathcal{F}$ under iteration maps, then in $L(\mathbb{R},C)$, $\M_\infty \in HOD_{\Sigma_{X,\beta}}$ and $A$ witnesses that $\rho_\omega(\M_\infty)\leq \delta_\beta^\P$. By $\sf{SMC}$ in $L(\mathbb{R},C)$ and the fact that $\mathcal{H}_X(\beta+1)$ is $\Omega$-full, we get that $A\in \P$. This is a contradiction.
\end{proof}
%By the assumption on the cofinality of $\Theta$, we get that
%\begin{center}
%$\mathcal{H}_X^+ \vDash ``\textrm{cof}(\Theta_X)$ is measurable".
%\end{center}
Now let $X$ be as in Lemma \ref{NotProjectAcross}. Using the embedding $\pi_X$ and the construction in \cite[Section 11]{sargsyanCovering2013}, we obtain a strategy $\Sigma_X$ for $\mathcal{H}_X^+$ such that
\begin{enumerate}
\item $\Sigma_X$ extends $\Sigma_X^-$;
\item for any $\Sigma_X$-iterate $\P$ of $\mathcal{H}_X^+$ via a stack $\vec{\mathcal{T}}$ such that $i^{\vec{\mathcal{T}}}$ exists, there is an embedding $\sigma:\P \rightarrow \mathcal{H}^+$ such that $\pi_X = \sigma \circ i^{\vec{\mathcal{T}}}$. Furthermore, letting $\Sigma_\P$ be the $\vec{\mathcal{T}}$-tail of $\Sigma_X$, for all $\alpha< \lambda^\P$, $\Sigma_{\P(\alpha)}\in \Omega$ has branch condensation.
\item $\Sigma_X$ is $\Gamma(\mathcal{H}^+_X,\Sigma_X)$-fullness preserving.
\end{enumerate}
\begin{remark} the construction in \cite{sargsyanCovering2013} is nontrivial in the case that $\mathcal{H}_X^+ \vDash \textrm{cof}(\Theta_X)$ is measurable; otherwise, as mentioned in the proof of Lemma \ref{NotProjectAcross}, $\Sigma_X$ is simply $\Sigma_X^-$ but because of $\pi_X$, it acts on all of $\mathcal{H}^+_X$. 
\end{remark}
We claim that $\Sigma_X\in \Omega$. Let $(\Q,\Lambda)$ be a $\Sigma_X$-iterate of $\mathcal{H}_X^+$ such that 
\begin{enumerate}[a)]
\item $\Q\in V$, $|\Q|^V\leq \kappa$; 
\item $\Lambda\rest V\in V$;
\item $\Lambda$ has branch condensation.
\end{enumerate}
c) follows from results in \cite{ATHM}. a) and b) can be ensured using boolean comparisons (see \cite{ATHM}). Using a), b), c), and arguments in previous subsections, we get that in $V[G]$,
\begin{center}
$L^{^\gTheta\F}(\mathbb{R},\F\rest\mathbb{R}) \vDash \sf{AD}$$^+$,
\end{center}
where $\F = \F_{\Lambda,\varphi_{\rm{all}}}$. This means $\Lambda \in \Omega$, and hence $\Sigma\in \Omega$.

\begin{lemma}
\label{FullnessPreserving}
$\forall^* X\in S$ $\Sigma_X$ is $\Omega$-fullness preserving .
\end{lemma}
\begin{proof}
Suppose not. Let $\vec{\mathcal{T}}_X$ be according to $\Sigma_X$ with end model $\Q_X$ such that $\Q_X$ is not $\Omega$-full. This means there is a strong cut point $\gamma$ such that letting $\alpha\leq \lambda^{\Q_X}$ be the largest such that $\delta^{\Q_X}_\alpha \leq \gamma$, then in $\Omega$, there is a mouse $\M\lhd \textrm{Lp}^{^\g\Sigma_{\Q_X(\alpha)}}(\Q_X|\gamma)$\footnote{The case where $\gamma = \delta_\alpha$ and $\M\lhd \textrm{Lp}^{^\gTheta \oplus_{\beta<\alpha}\Sigma_{\Q_X(\beta)}}(\Q|\gamma)$ is similar.} such that $\M\notin \Q_X$. Let $k: \Q_X\rightarrow \mathcal{H}^+$ be such that $\pi_X = k \circ i^{\vec{\mathcal{T}}_X}$. We use $i$ to denote $i^{\vec{\mathcal{T}}_X}$ from now on. %Using boolean comparison, we can find a pair $((\vec{\mathcal{T}_X},\Q_X),\M)\in V$.

Let $(\P_X,\Sigma_{\P_X})\in V$ be a $\Sigma_X^-$-hod pair such that 
\begin{itemize}
\item $\Gamma(\P_X,\Sigma_{\P_X})\vDash \Q$ is not full as witnessed by $\M$. 
\item $\Sigma_{\P_X}\in \Omega$ is fullness preserving and has branch condensation.
\item $\lambda^{\P_X}$ is limit and cof$^{\P_X}(\lambda^{\P_X})$ is not measurable in $\P_X$.
\end{itemize}
Such a pair $(\P_X,\Sigma_{\P_X})$ exists by boolean comparisons.

By arguments similar to that used in \ref{NotProjectAcross}, for almost all $X\in S$, no levels of $\P_X$ projects across $\mathcal{H}_X$ and in fact, $o(\mathcal{H}_X^+)$ is a cardinal of  $\P_X$. The second clause follows from the following argument. Suppose not and let $\N_X\lhd \P_X$ be least such that $\rho_\omega(\N_X) = \Theta_X$ for stationary many $X\in S$. By minimality of $\N_X$ and an argument similar to that in Lemma \ref{NotProjectAcross}, we may assume for stationary many $X\in S$, $\N_X\in M_X$. Fix such an $X$. Let $f: \kappa^* \rightarrow \Theta_X$ be an increasing and cofinal map in $\mathcal{H}_X^+$, where $\kappa^* = \textrm{cof}^{\mathcal{H}_X^+}(\Theta_X)$. We can construe $\N_X$ as a sequence $g = \langle \N_\alpha \ | \ \alpha < \kappa^*\rangle$, where $\N_\alpha = \N_X\cap \delta^{\mathcal{H}_X^+}_{f(\alpha)}$. Note that $\N_\alpha\in \mathcal{H}_X^+$ for each $\alpha<\kappa^*$. Now let $\R_0 = \textrm{Ult}_{0}(\mathcal{H}_X^+,\mu)$, $\R_1= \textrm{Ult}_{\omega}(\N_X,\mu)$, where $\mu\in \mathcal{H}^+_X$ is the (extender on the sequence of $\mathcal{H}^+_X$ coding a) measure on $\kappa^*$ with Mitchell order $0$. Let $i_0:\mathcal{H}_X^+ \rightarrow \R_0$, $i_1:\N_X\rightarrow \R_1$ be the ultrapower maps. Letting $\delta = \delta_{\lambda^{\mathcal{H}^+_\sigma}} = \Theta_X$, it's easy to see that $i_0\rest (\delta+1)=i_1\rest(\delta+1)$ and $\powerset(\delta)^{\R_0} = \powerset(\delta)^{\R_1}$. This means $\langle i_1(\N_\alpha) \ | \ \alpha<\kappa^*\rangle \in \powerset(\delta)^{\R_0}$. By fullness of $\mathcal{H}_X^+$ in $\Omega$,\footnote{Any $A\subset \delta$ in $\R_0$ is $OD^\Omega_{\Sigma_X^-}$ (as in the proof of Lemma \ref{NotProjectAcross}, this means OD$^{L(\mathbb{R},C)}_{\Sigma_X^-}$ for some $C\in\Omega$) and so by Strong Mouse Capturing ($\sf{SMC}$, see \cite{ATHM}), $A\in \mathcal{H}_X^+$.} $\langle i_1(\N_\alpha) \ | \ \alpha<\kappa^*\rangle \in \mathcal{H}^+_X$. Using $i_0$, $\langle i_1(\N_\alpha) \ | \ \alpha<\kappa^*\rangle \in \mathcal{H}^+_X$, and the fact that $i_0\rest \mathcal{H}_X^+|\Theta_X = i_1\rest \N_X|\Theta_X \in \mathcal{H}^+_X$, we can get $\N_X \in \mathcal{H}^+_X$ as follows. For any $\alpha,\beta < \Theta_X$, $\alpha \in \N_\beta$ if and only if $i_0(\alpha) \in i_1(\N_\beta)= i_0(\N_\beta)$. Since $\mathcal{H}_X^+$ can compute the right hand side of the equivalence, it can compute the sequence $\langle \N_\alpha \ | \ \alpha<\kappa^*\rangle$. Contradiction.

In other words, $\P_X$ thinks $\mathcal{H}^+_X$ is full. For here on, let $\P=\P_X$, $\Sigma_\P=\Sigma_{\P_X}$, $(\vec{\T}_X,\Q_X)=(\vec{\T},\Q)$. Let 
\begin{center}
$\pi_X^*:\P\rightarrow \mathcal{H}^{++}$
\end{center}
be the ultrapower map by the $(\textrm{crt}(\pi_X),\Theta)$-extender $E_{\pi_X}$ induced by $\pi_X$. Note that $\pi_X^*$ extends $\pi_X\rest \mathcal{H}_X^+$ and $\mathcal{H}^{++}$ is wellfounded since $X$ is closed under $\omega$-sequences. Let 
\begin{center}
$i^*: \P\rightarrow \R$
\end{center}
be the ultrapower map by the $(\textrm{crt}(i),\delta^\Q)$-extender induced by $i$. Note that $\Q\lhd \R$ and $\R$ is wellfounded since there is a natural map \begin{center}
$k^*: \R \rightarrow \mathcal{H}^{++}$
\end{center} 
extending $k$ and $\pi_X^* = k^*\circ i^*$. Without loss of generality, we may assume $\M$'s unique strategy $\Sigma_\M \leq_w \Sigma_\P$. Also, let $(\dot{\Q},\dot{\vec{\T}})$ be the canonical $Col(\omega,\omega_2)$-names for $(\Q,\vec{\T})$. Let $K$ be the transitive closure of $H_{\omega_2}\cup (\dot{\Q},\dot{\vec{\T}})$.

Let $\W = \M_\omega^{\Sigma_\P,\sharp}$ and $\Lambda$ be the unique strategy of $\W$. Let $\W^*$ be a $\Lambda$-iterate of $\W$ below its first Woodin cardinal that makes $K$-generically generic. Then in $\W^*[K]$, the derived model $D(\W^*[K])$ satisfies
\begin{center}
$L(\Gamma(\P,\Sigma_\P),\mathbb{R}) \vDash \dot{\Q} \textrm{ is not full}.$\footnote{This is because we can continue iterating $\W^*$ above the first Woodin cardinal to $\W^{**}$ such that letting $\lambda$ be the sup of the Woodin cardinals of $\W^{**}$, then there is a $Col(\omega,<\lambda)$-generic $h$ such that $\mathbb{R}^{V[G]}$ is the symmetric reals for $\W^{**}[h]$. And in $\W^{**}(\mathbb{R}^{V[G]})$, the derived model satisfies that $L(\Gamma(\P,\Sigma_\P)) \vDash \Q$ is not full.}
\end{center}
So the above fact is forced over $\W^*[K]$ for $\dot{\Q}$.

Let $H\prec H_{\lambda^{+++}}$ be countable such that all relevant objects are in $H$. Let $\pi: M\rightarrow H$ invert the transitive collapse and for all $a\in H$, let $\overline{a}=\pi^{-1}(a)$. By the countable completeness of $E_{\pi_X}$ there is a map $\overline{\pi}:\overline{\R}\rightarrow \mathcal{H}_X^{++}$ such that 
\begin{center}
$\pi\rest \overline{\P} = \overline{\pi}\circ \overline{i^*}$.\footnote{This is because $i^*\circ \bar{\pi} = \pi\rest \bar{\R}$ and $\pi\rest i^*\circ \bar{\P} = \pi\rest \bar{\R}\circ \bar{i^*}$. }
\end{center}
Let $\Sigma_{\overline{\P}}$ be the $\pi$-pullback of $\Sigma_\P$ and $\Sigma_{\overline{\R}}$ be the $\overline{\pi}$-pullback of $\Sigma_{\P}$. Note that $\Sigma_{\overline{\P}}$ extends $\pi^{-1}(\Sigma_\P)$ and $\Sigma_{\overline{\P}}$ is also the $\overline{i^*}$-pullback of $\Sigma_{\overline{\R}}$; so in particular, $\Sigma_{\bar{\P}} \leq_w \Sigma_{\bar{\R}}$. We also confuse $\bar{\Lambda}$ with the $\pi$-pullback of $\Lambda$. Hence $\Gamma(\overline{\P},\Sigma_{\overline{\P}})$ witnesses that $\overline{\Q}$ is not full and this fact is forced over $\bar{\W^*}[\bar{K}]$ for the name $\bar{\dot{\Q}}$. This means if we further iterate $\bar{\W^*}$ to $\mathcal{Y}$ such that $\mathbb{R}^{V[G]}$ can be realized as the symmetric reals over $\mathcal{Y}$ then in the derived model $D(\mathcal{Y})$, 
\begin{equation}\label{notFull}
L(\Gamma(\bar{\P},\Sigma_{\bar{\P}})) \vDash \bar{\Q} \textrm{ is not full}. 
\end{equation}
In the above, we have used the fact that the interpretation of the UB-code of the strategy for $\bar{\P}$ in $\mathcal{Y}$ to its derived model is $\Sigma_{\bar{\P}}\rest\mathbb{R}^{V[G]}$; this key fact is proved in \cite[Theorem 3.26]{ATHM}.

Now we iterate $\overline{\R}$ to $\S$ via $\Sigma_{\bar{\R}}$ to realize $\mathbb{R}^{V[G]}$ as the symmetric reals for the collapse $Col(\omega,<\delta^\S)$, where $\delta^\S$ is the sup of $\S$'s Woodin cardinals. By \ref{notFull} and the fact that $\Sigma_{\bar{\P}}\leq_w \Sigma_{\bar{\R}}$, we get that in the derived model $D(\S)$,
\begin{center}
$\bar{\Q}$ is not full as witnessed by $\bar{\M}$.
\end{center}
So $\Sigma_{\bar{\M}}$ is OD$_{\Sigma_{\overline{\Q}}}$ in $D(\S)$ and hence $\overline{\M}\in \overline{\R}$. This contradicts internal fullness of $\bar{\Q}$ in $\bar{\R}$. 
\end{proof}

We continue with a key definition, due to G. Sargsyan. This definition is first formulated in \cite{sargsyanCovering2013} and we reformulate it a bit to fit our situation.

\begin{definition}[Sargsyan]
\label{ACondensation}
Suppose $X\in S$ and $A\in \mathcal{H}^+_X \cap \powerset(\Theta_X)$. We say that $\pi_X$ has \textbf{$A$-condensation} if whenever $\Q$ is such that there are elementary embeddings $\upsilon: \mathcal{H}^+_X \rightarrow \Q$, $\tau:\Q\rightarrow \mathcal{H}^+$ such that $\Q$ is countable in $V[G]$ and $\pi_X = \tau\circ \upsilon$, then $\upsilon(T_{\mathcal{H}^+_X,A}) = T_{\Q,\tau,A}$, where  
\begin{center}
$T_{\mathcal{H}^+_X,A} = \{(\phi,s) \ | \ s\in [\Theta_X]^{<\omega}\wedge \mathcal{H}^+_X \vDash \phi[s,A]\}$,
\end{center}
and
\begin{center}
$T_{\Q,\tau,A} = \{(\phi,s) \ | \ s\in [\delta_\alpha^\Q]^{<\omega} \textrm{ for some } \alpha<\lambda_\Q \wedge \mathcal{H}^+ \vDash \phi[i^{\Sigma^{\tau,-}_{\Q}}_{\Q(\alpha),\infty}(s),\pi_X(A)]\}$,
\end{center}
where $\Sigma^{\tau}_\Q$ is the $\tau$-pullback strategy and $\Sigma^{\tau,-}_\Q = \oplus_{\alpha<\lambda^\Q} \Sigma^{\tau}_{\Q(\alpha)}$. We say $\pi_X$ has \textbf{condensation} if it has $A$-condensation for every $A\in \mathcal{H}^+_X \cap \powerset(\Theta_X)$.
\end{definition}
The following is the key lemma (cf. \cite[Section 11]{sargsyanCovering2013}).
\begin{lemma}
\label{GettingACondensation}
$\forall^* X\in S$ $\pi_X$ has condensation.
\end{lemma}
\begin{proof}
Suppose not. Let $T$ be the set of counterexamples. Hence $T$ is stationary. For each $X\in T$, let $A_X$ be the $\lesssim_X$-least such that $\pi_X$ fails to have $A_X$-condensation, where $\lesssim_X$ is the canonical well-ordering of $\mathcal{H}_X^+$. Recall that if $(\P,\Sigma)$ is a hod pair such that $\delta^\P$ has measurable cofinality then $\Sigma^- = \oplus_{\alpha < \lambda^{\P}} \Sigma_{\P(\alpha)}$. We say that a tuple $\{\langle \P_i,\Q_i,\tau_i,\xi_i,\pi_i,\sigma_i \ | \ i<\omega \rangle, \M_{\infty,Y}\}$ is a \textbf{bad tuple} if
\begin{enumerate}
\item $Y\in S$;
\item $\P_i = \mathcal{H}^+_{X_i}$ for all $i$, where $X_i\in T$;
\item for all $i < j$, $X_i \prec X_j \prec Y$;
\item $\M_{\infty,Y}$ be the direct limit of iterates $(\Q,\Lambda)$ of $(\mathcal{H}_Y^+,\Sigma_Y)$ such that $\Lambda$ has branch condensation; 
\item for all $i$, $\xi_i:\P_i\rightarrow \Q_i$, $\sigma_i:\Q_i \rightarrow \M_{\infty,Y}$,  $\tau_i: \P_{i+1}\rightarrow \M_{\infty,Y}$, and $\pi_i: \Q_i \rightarrow \P_{i+1}$;
\item for all $i$, $\tau_i = \sigma_i\circ \xi_i$, $\sigma_i = \tau_{i+1} \circ \pi_i$, and $\pi_{X_i,X_{i+1}}\rest \P_i=_{\textrm{def}} \phi_{i,i+1} = \pi_i \circ \xi_i$;
\item $\phi_{i,i+1}(A_{X_i}) = A_{X_{i+1}}$;
\item for all $i$, $\xi_i(T_{\P_i,A_{X_i}}) \neq T_{\Q_i,\sigma_i,A_{X_i}}$. 
\end{enumerate}
In (8), $T_{\Q_i,\sigma_i,A_{X_i}}$ is computed relative to $\M_{\infty,Y}$, that is
\begin{center}
$T_{\Q_i,\sigma_i,A_{X_i}} = \{(\phi,s) \ | \ s\in [\delta_\alpha^{\Q_i}]^{<\omega} \textrm{ for some } \alpha<\lambda^{\Q_i} \wedge \M_{\infty,Y} \vDash \phi[i^{\Sigma^{\sigma_i,-}_{\Q_i}}_{\Q_i(\alpha),\infty}(s),\tau_i(A_{X_i})]\}$
\end{center}
\noindent \textbf{Claim: } There is a bad tuple.
\begin{proof}
For brevity, we first construct a bad tuple $\{\langle \P_i,\Q_i,\tau_i,\xi_i,\pi_i,\sigma_i \ | \ i<\omega \rangle, \mathcal{H}^+\}$ with $\mathcal{H}^+$ playing the role of $\M_{\infty,Y}$. We then simply choose a sufficiently large $Y\in S$ and let $i_Y:\mathcal{H}_Y^+\rightarrow \M_{\infty,Y}$ be the direct limit map, $m_Y: \M_{\infty,Y}\rightarrow \mathcal{H}^+$ be the natural factor map, i.e. $m_Y\circ i_Y = \pi_Y$. It's easy to see that for all sufficiently large $Y$, the tuple $\{\langle \P_i,\Q_i,m_Y^{-1}\circ \tau_i,m_Y^{-1}\circ \xi_i,m_Y^{-1}\circ \pi_i,m_Y^{-1}\circ \sigma_i \ | \ i<\omega \rangle, \M_{\infty,Y}\}$ is a bad tuple.

The key point is (6). Let $A_X^* = \pi_X(A_X)$ for all $X\in T$. By Fodor's lemma, there is an $A$ such that $\exists^* X\in T \ A^*_X = A$. So there is an increasing and cofinal sequence $\{X_\alpha \ | \ \alpha < \omega_3 \} \subseteq T$ such that for $\alpha < \beta$, $\pi_{X_\alpha,X_\beta}(A_{X_\alpha}) = A_{X_\beta} = \pi_{X_\beta}^{-1}(A)$. This easily implies the existence of such a tuple $\{\langle \P_i,\Q_i,\tau_i,\xi_i,\pi_i,\sigma_i \ | \ i<\omega \rangle, \mathcal{H}^+\}$.
\end{proof}
Fix a bad tuple $\mathcal{A} = \{\langle \P_i,\Q_i,\tau_i,\xi_i,\pi_i,\sigma_i \ | \ i<\omega \rangle, \M_{\infty,Y}\}$. Let $(\P_0^+,\Pi)$ be a $^\g\Sigma^-_{\P_0}$-hod pair such that 
\begin{center}
$\Gamma(\P_0^+,\Pi) \vDash \mathcal{A}$ is a bad tuple.
\end{center}
We may also assume $(\P_0^+,\Pi\rest V)\in V$, $\lambda^{\P_0^+}$ is limit of nonmeasurable cofinality in $\P_0^+$ and there is some $\alpha<\lambda^{\P_0^+}$ such that $\Sigma_{Y} \leq_w \Pi_{\P_0^+(\alpha)}$. This type of reflection is possible because we replace $\mathcal{H}^+$ by $\M_{\infty,Y}$. Let $\W = \M_{\omega}^{\sharp, \Sigma_Y, \Pi, \oplus_{n<\omega}\Sigma_{X_n}}$ and $\Lambda$ be the unique strategy of $\W$. If $\mathcal{Z}$ is the result of iterating $\W$ via $\Lambda$ to make $\mathbb{R}^{V[G]}$ generic, then letting $h$ be $\mathcal{Z}$-generic for the Levy collapse of the sup of $\mathcal{Z}$'s Woodin cardinals to $\omega$ such that $\mathbb{R}^{V[G]}$ is the symmetric reals of $\mathcal{Z}[h]$, then in $\mathcal{Z}(\mathbb{R}^{V[G]})$,

\begin{center}
$\Gamma(\P_0^+,\Pi) \vDash \mathcal{A}$ is a bad tuple.
\end{center}

Now we define by induction $\xi_i^+: \P_i^+ \rightarrow \Q_i^+$, $\pi_i^+: \Q_i^+ \rightarrow \P_{i+1}^+$, $\phi_{i,i+1}^+: \P_i^+\rightarrow \P_{i+1}^+$ as follows. $\phi_{0,1}^+: \P_0^+\rightarrow \P_{1}^+$ is the ultrapower map by the $(\textrm{crt}(\pi_{X_0,X_1}),\Theta_{X_1})$-extender derived from $\pi_{X_0,X_1}$. Note that $\phi_{0,1}^+$ extends $\phi_{0,1}$. Let $\xi_0^+: \P_0^+ \rightarrow \Q_0^+$ extend $\xi_0$ be the ultrapower map by the $(\textrm{crt}(\xi_0),\delta^{\Q_0})$-extender derived from $\xi_0$. Finally let $\pi_0^+ = (\phi^+_{0,1})^{-1}\circ \xi_0^+$. The maps $\xi_i^+, \pi_i^+, \phi_{i,i+1}^+$ are defined similarly. Let also $\M_Y = \textrm{Ult}(\P_0^+,E)$, where $E$ is the $(\lambda_X,\Theta_Y)$-extender derived from $\pi_{X,Y}$. There are maps $\epsilon_{2i}: \P_i^+ \rightarrow \M_Y$, $\epsilon_{2i+1}:\Q_i^+\rightarrow \M_Y$ for all $i$ such that $\epsilon_{2i} = \epsilon_{2i+1}\circ \xi^+_i$ and $\epsilon_{2i+1} = \epsilon_{2i+2}\circ \pi_i^+$. When $i = 0$, $\epsilon_0$ is simply $i_E$. Letting $\Sigma_i = \Sigma^-_{\P_i}$ and $\Psi=\Sigma_{\Q_i}^-$, $A_i = A_{X_i}$, there is a finite sequence of ordinals $t$ and a formula $\theta(u,v)$ such that in $\Gamma(\P_0^+,\Pi)$
\begin{enumerate}
\setcounter{enumi}{8}
\item for every $i<\omega$, $(\phi,s)\in T_{\P_i,A_i} \Leftrightarrow \theta[i^{\Sigma_i}_{\P_i(\alpha),\infty},t]$, where $\alpha$ is least such that $s\in [\delta_\alpha^{\P_i}]^{<\omega}$;
\item for every $i$, there is $(\phi_i,s_i)\in T_{\Q_i,\xi_i(A_i)}$ such that $\neg \theta[i^{\Psi_i}_{\Q_i(\alpha)}(s_i),t]$ where $\alpha$ is least such that $s_i\in [\delta_\alpha^{\Q_i}]^{<\omega}$.
\end{enumerate}
The pair $(\theta,t)$ essentially defines a Wadge-initial segment of $\Gamma(\P_0^+,\Pi)$ that can define the pair $(\M_{\infty,Y}, A)$, where $\tau_i(A_i)=A$ for some (any) $i$. 

Now let $X\prec H_{\lambda^{+++}}$ be countable that contains all relevant objects and $\pi: M\rightarrow X$ invert the transitive collapse. For $a\in X$, let $\overline{a}=\pi^{-1}(a)$. By countable completeness of the extender $E$, there is a map $\pi^*: \overline{\M_Y} \rightarrow \P_0$ such that $\pi\rest \overline{\M_Y} = \epsilon_0 \circ \pi^*$. Let $\overline{\Pi_i}$ be the $\pi^*\circ \overline{\epsilon_i}$-pullback of $\Pi$. Note that in $V[G]$, $\overline{\Sigma_Y} \leq_w \overline{\Pi_0} \leq_w \overline{\Pi_1} \dots \leq_w \Pi^{\pi^*}$. 

Let $\dot{\mathcal{A}}\in (H_{\bar{\kappa}})^M$ be the canonical name for $\bar{\mathcal{A}}$. It's easy to see (using the assumption on $\W$) that if $\W^*$ is a result of iterating $\bar{\W}$ via $\bar{\Lambda}$ (we confuse $\bar{\Lambda}$ with the $\pi$-pullback of $\Lambda$; they coincide on $M$) in $M$ below the first Woodin of $\bar{\W}$ to make $H$-generically generic, where $H$ is the transitive closure of $H_{\omega_2}^M\cup \dot{A}$, then in $\W^*[H]$, the derived model of $\W^*[H]$ at the sup of $\W^*$'s Woodin cardinals satisfies:
\begin{center}
$L(\bar{\P}_0,\mathbb{R}) \vDash \dot{\mathcal{A}}$ is a bad tuple.
\end{center} 

Now we stretch this fact out to $V[G]$ by iterating $\W^*$ to $\W^{**}$ to make $\mathbb{R}^{V[G]}$-generic. In $\W^{**}(\mathbb{R}^{V[G]})$, letting $i: \W^* \rightarrow \W^{**}$ be the iteration map then
\begin{center}
$\Gamma(\bar{\P_0}^+,\bar{\Pi}) \vDash i(\bar{\mathcal{A}})$\footnote{We abuse the notation slightly here. Technically, $\bar{\mathcal{A}}$ is not in $\W^*$ but $\W^*$ has a canonical name $\dot{{\mathcal{A}}}$ for $\bar{\mathcal{A}}$. Hence by $i(\bar{\mathcal{A}})$, we mean the interpretation of $i(\dot{{\mathcal{A}}})$.} is a bad tuple.
\end{center} 

By a similar argument as in Theorem 3.1.25 of \cite{trangThesis2013}, we can use the strategies $\overline{\Pi_i}^+$'s to simultanously execute a $\mathbb{R}^{V[G]}$-genericity iterations. The last branch of the iteration tree is wellfounded. The process yields a sequence of models $\langle\overline{\P^+_{i,\omega}},\overline{\Q_{i,\omega}^+} \ | \ i<\omega\rangle$ and maps $\overline{\xi^+_{i,\omega}}:\overline{\P^+_{i,\omega}}\rightarrow \overline{\Q^+_{i,\omega}}$, $\overline{\pi^+_{i,\omega}}:\overline{\Q^+_{i,\omega}}\rightarrow \overline{\P^+_{i+1,\omega}}$, and $\overline{\phi^+_{i,i+1,\omega}} = \overline{\pi^+_{i,\omega}}\circ \overline{\pi^+_{i,\omega}}$. Furthermore, each $\overline{\P^+_{i,\omega}}, \overline{\Q^+_{i,\omega}}$ embeds into a $\Pi^{\pi^*}$-iterate of $\overline{\M_Y}$ and hence the direct limit $\P_\infty$ of $(\overline{\P^+_{i,\omega}}, \overline{\Q^+_{j,\omega}} \ | \ i,j<\omega)$ under maps $\overline{\pi^+_{i,\omega}}$'s and $\overline{\xi^+_{i,\omega}}$'s is wellfounded. We note that $\overline{\P^+_{i,\omega}}$ is a $^{\g}\Sigma^\pi_i$-premouse and $\overline{\Q^+_{i,\omega}}$ is a $^{\g}\Psi^\pi_i$-premouse because the genericity iterations are above $\overline{\P_i}$ and $\overline{\Q_i}$ for all $i$ and by \cite[Theorem 3.26]{ATHM}, the interpretation of the strategy of $\bar{\P_i}$ ($\bar{\Q_i}$ respectively) in the derived model of $\bar{\P_{i,\omega}^+}$ ($\bar{\P_{i,\omega}^+}$, respectively) is $^\g\Sigma^\pi_i$ ($^\g\Psi^\pi_i$, respectively). Let $C_i$ be the derived model of $\overline{\P^+_{i,\omega}}$, $D_i$ be the derived model of $\overline{\Q^+_{i,\omega}}$ (at the sup of the Woodin cardinals of each model), then $\mathbb{R}^{V[G]} = \mathbb{R}^{C_i} = \mathbb{R}^{D_i}$. Furthermore, $C_i\cap \powerset(\mathbb{R})\subseteq D_i\cap \powerset(\mathbb{R})\subseteq C_{i+1}\cap \powerset(\mathbb{R})$ for all $i$.

(9), (10) and the construction above give us that there is a $t\in [\textrm{OR}]^{<\omega}$, a formula $\theta(u,v)$ such that
\begin{enumerate}
\setcounter{enumi}{10}
\item for each $i$, in $C_i$, for every $(\phi,s)$ such that $s\in \delta^{\overline{\P_i}}$, $(\phi,s)\in T_{\overline{\P_i},\overline{A_i}}\Leftrightarrow \theta[i^{\overline{\Sigma_i}}_{\overline{\P_i}(\alpha),\infty}(s),t]$ where $\alpha$ is least such that $s\in [\delta_\alpha^{\overline{\P_i}}]^{<\omega}$.
\end{enumerate}
Let $n$ be such that for all $i\geq n$, $\overline{\xi^+_{i,\omega}}(t) = t$. Such an $n$ exists because the direct limit $\P_\infty$ is wellfounded as we can arrange that $\P_\infty$ is embeddable into a $\Pi^{\pi^*}$-iterate of $\bar{\M_Y}$. By elementarity of $\overline{\xi^+_{i,\omega}}$ and the fact that $\overline{\xi^+_{i,\omega}}\rest \P_i = \overline{\xi_i}$,
\begin{enumerate}
\setcounter{enumi}{11}
\item for all $i\geq n$, in $D_i$, for every $(\phi,s)$ such that $s\in \delta^{\overline{\Q_i}}$, $(\phi,s)\in T_{\overline{\Q_i},\overline{\xi_i}(\overline{A_i})}\Leftrightarrow \theta[i^{\overline{\Psi_i}}_{\overline{\Q_i}(\alpha),\infty}(s),t]$ where $\alpha$ is least such that $s\in [\delta_\alpha^{\overline{\Q_i}}]^{<\omega}$.
\end{enumerate}
However, using (10), we get
\begin{enumerate}
\setcounter{enumi}{12}
\item for every $i$, in $D_i$, there is a formula $\phi_i$ and some $s_i\in [\delta^{\overline{\Q_i}}]^{<\omega}$ such that $(\phi_i,s_i)\in T^{\overline{\Q_i},\overline{\xi_i}(\overline{A_i})}$ but $\neg \phi[i^{\overline{\Psi_i}}_{\overline{\Q_i}(\alpha),\infty}(s_i),t]$ where $\alpha$ is least such that $s\in [\delta_\alpha^{\overline{\Q_i}}]^{<\omega}$.
\end{enumerate}
Clearly (12) and (13) give us a contradiction. This completes the proof of the lemma.
\end{proof}
\begin{remark}
The main ideas of the proof above originate from \cite[Lemma 11.15]{sargsyanCovering2013}. The main difference is in the situation of \cite[Lemma 11.15]{sargsyanCovering2013}, there is an elementary embedding $j$ acting on all of $V$, so roughly speaking, the iterability of the $\P_i^+$'s is justified by embedding them into $j(\P_0^+)$. Here we don't have such a $j$, we use pressing down arguments, countable closure of hulls $X\in S$ and reflection arguments instead.
\end{remark}
Fix an $X$ satisfying the conclusion of Lemma \ref{GettingACondensation}. Suppose $(\Q,\vec{\mathcal{T}}) \in I(\mathcal{H}^+_X,\Sigma_X)$ is such that $i^{\vec{\mathcal{T}}}:\mathcal{H}^+_X\rightarrow \Q$ exists. Let $\gamma^{\vec{\mathcal{T}}}$ be the sup of the generators of $\vec{\mathcal{T}}$. For each $x\in \Q$, say $x = i^{\vec{\mathcal{T}}}(f)(s)$ for $f\in \mathcal{H}_X^+$ and $s\in [\delta^\Q_\alpha]^{<\omega}$, where $\delta^\Q_\alpha\leq \gamma^{\vec{\mathcal{T}}}$ is least such, then let $\tau_\Q(x) = \pi_X(f)(i^{\Sigma_{\Q,\vec{\mathcal{T}}}}_{\Q(\alpha),\infty}(s))$. 

\begin{remark}
By Lemma \ref{GettingACondensation} and \cite[Theorem 3.26]{ATHM}, $\tau_\Q$ is elementary and $\tau_{\Q}\rest \delta^\Q = i^{\Sigma_{\Q,\vec{\mathcal{T}}}}_{\Q|\delta^\Q,\infty}\rest \delta^\Q = i^\Lambda_{\Q|\delta^\Q,\infty}\rest \delta^\Q$, where $\Lambda$ is the $\tau_\Q$-pullback strategy of $\Q|\delta^\Q$.
\end{remark}
\begin{lemma}
\label{commutativity}
Fix an $X$ satisfying the conclusion of Lemma \ref{GettingACondensation}. Suppose $(\Q,\vec{\mathcal{T}})\in I(\mathcal{H}^+_X,\Sigma_X)$ and $(\R,\vec{\mathcal{U}})\in I(\Q,\Sigma_{\Q,\vec{\mathcal{T}}})$ are such that $i^{\vec{\mathcal{T}}}, i^{\vec{\mathcal{U}}}$ exist and $\Sigma_{\Q,\vec{\mathcal{T}}}$ and $\Sigma_{\R,\vec{\mathcal{U}}}$ have branch condensation. Then $\tau_\Q = \tau_\R\circ i^{\vec{\mathcal{U}}}$.
\end{lemma}
\begin{proof}
Let $x\in \Q$. There are some $f\in \mathcal{H}^+_X$ and $s\in [\gamma^{\vec{\mathcal{T}}}]^{<\omega}$ such that $x = i^{\vec{\mathcal{T}}}(f)(s)$. So $\tau_\Q(x) = \pi_X(f)(i^{\Sigma_{\Q,\vec{\mathcal{T}}}}_{\Q,\infty}(s))$. On the other hand, $\tau_\R\circ i^{\vec{\mathcal{U}}}(x) = \tau_\R\circ i^{\vec{\mathcal{U}}}(i^{\vec{\mathcal{T}}}(f)(s)) = \pi_X(f)(i^{\Sigma_{\R,\vec{\mathcal{U}}}}_{\R,\infty}\circ i^{\vec{\U}}\circ i^{\Sigma_{\Q,\vec{\mathcal{T}}}}(s)) = \pi_X(f)(i^{\Sigma_{\Q,\vec{\mathcal{T}}}}_{\Q,\infty}(s)) = \tau_\Q(x)$.
\end{proof}

Let $\M_\infty(\mathcal{H}_X^+,\Sigma_X)$ be the direct limit of all hod pairs $(\Q,\Lambda)\in I(\mathcal{H}_X^+,\Sigma_X)$ such that $\Lambda$ has branch condensation. The lemma implies that the map $\sigma: \M_\infty(\mathcal{H}_X^+,\Sigma_X)\rightarrow \mathcal{H}^+$ defined as: 

\begin{adjustwidth}{2cm}{2cm}
$\sigma(x) = y$ iff whenever $(\R,\Lambda)\in I(\mathcal{H}_X^+,\Sigma_X)$ is such that $\Lambda$ has branch condensation, and $i^\Lambda_{\R,\infty}(x^*)=x$ for some $x^*$, then $y = \tau_\R(x^*)$
\end{adjustwidth}
is elementary and crt$(\sigma) = \delta =_{\rm{def}} \delta^{\M_\infty(\mathcal{H}_X^+,\Sigma_X)}$. This implies that $\M_\infty(\mathcal{H}_X^+,\Sigma_X)\vDash ``\delta$ is regular". Let $(\Q,\Lambda)\in I(\mathcal{H}_X^+,\Sigma_X)$ be such that $\Lambda$ has branch condensation. By a similar argument as those used before, we get $\Lambda\in \Omega$ and in fact since $\Q\vDash ``\delta^\Q$ is regular", we easily get that $N=L(\Gamma(\Q,\Lambda),\mathbb{R})\vDash ``\Theta$ is regular" (note that $\Theta^N$ is the image of $\delta^\Q$ under the direct limit map into the direct limit of all $\Lambda$-iterates). This contradicts the assumption that there is no model $M$ satisfying ``$\sf{AD}$$_\mathbb{R} + \Theta$ is regular". Such an $M$ has to exist after all. This finishes this subsection and the proof of Theorem \ref{main_technical_theorem}.
\begin{remark}\label{InV}
In the above, there are $(\Q, \Lambda)\in V$ that are in $I(\mathcal{H}_X^+,\Sigma_X)$ (this is via a standard boolean comparison argument, cf. \cite{ATHM}). By taking a countable hull, we can find a countable hod pair $(\Q^*,\Lambda^*)$ that generates in $V$ a model of ``$\sf{AD}_\mathbb{R}$$ + \Theta$ is regular" by an $\mathbb{R}$-genericity iteration argument using the fact that $\Lambda^*$ has branch condensation and is $\kappa$-universally Baire. 
\end{remark}
%\begin{proof}[Proof of Theorem \ref{consecutiveSquareFailures}] The proof is the same as that of Theorem \ref{main_technical_theorem}, noting the following changes:
%\begin{enumerate}[(i)]
%\item The core model induction is carried out in $V^{\rm{Col(\omega,\kappa)}}$.
%\item For $A\subseteq \delta$, for any nice operator $\Fop$ such that Lp$^\Fop(A)$ is defined, cof$(\textrm{o}(\rm{Lp}^\Fop(A))) \leq \kappa$.
%\item We use models $X\prec H_{\delta^{++}}$ of size $\kappa$ for which $X^\omega\subseteq X$ (in place of $X\prec H_{\kappa^{++}}$ of size $\aleph_2$ for which $X^\omega\subseteq X$ in the proof of Theorem \ref{main_technical_theorem}).
%\item The hypothesis $2^\kappa = \delta$ is used to ensure that a canonical name of $\mathbb{R}^{V[G]}$ is in $H_{\delta^+}$.
%\end{enumerate}
%\end{proof}

\section{QUESTIONS AND OPEN PROBLEMS}
We conjecture the following (the proof of which will settle Conjecture \ref{ConLSA}).
\begin{conjecture}\label{LSAFromPFA}
Suppose $\kappa$ is a cardinal such that $\kappa^\omega=\kappa$. Let $\lambda = 2^{\kappa}$. Suppose for all cardinal $\alpha\in [\kappa^+,\lambda^+]$, $\neg \square(\alpha)$. Then in $V^{Col(\omega,\kappa)}$, there are models $M$ containing $\mathbb{R}\cup \rm{OR}$ such that $M \vDash \sf{LSA}$.
\end{conjecture}
We're hopeful that the conjecture have a positive answer. This is because we believe it's possible to construct hod mice generating models of $\textsf{LSA}$ from the hypothesis of Conjecture \ref{LSAFromPFA}. 

We end the paper with the following technical questions, whose solution seems to require new core model induction techniques for working with hulls that are not closed under countable sequences. Note that in most interesting cases (e.g. under $\textsf{PFA}$) $\omega_1$-guessing models of size $\aleph_1$ cannot be closed under $\omega$-sequences.
\begin{question}
Let $\kappa=2^{\aleph_2}$. Can one construct a model of ``$\sf{AD}$$_\mathbb{R} + \Theta$ is regular'' from the existence of stationary many $\omega_1$-guessing models $X\prec H_{\kappa^{++}}$ such that  $|X|=\aleph_1$?
\end{question}
\begin{question}
Let $\kappa=2^{\aleph_2}$. Can one construct a model of ``$\sf{AD}$$_\mathbb{R} + \Theta$ is regular'' from $\neg \square(\alpha)$ for all $\alpha\in [\omega_2,\kappa^+]$?
\end{question}
\bibliographystyle{plain}
\small
\bibliography{Rmicebib}

\begin{thebibliography}{10}

\bibitem{JSSS}
R.~Jensen, E.~Schimmerling, R.~Schindler, and J.~R. Steel.
\newblock Stacking mice.
\newblock {\em J. Symbolic Logic}, 74(1):315--335, 2009.

\bibitem{jensenSteelCoreModel}
R.~Jensen and J.~R. Steel.
\newblock ${K}$ without the measurable.
\newblock 2013.
\newblock To appear in the Journal of Symbolic Logic.

\bibitem{kechris1983equivalence}
A.~S. Kechris and W.~H. Woodin.
\newblock Equivalence of partition properties and determinacy.
\newblock {\em Proceedings of the National Academy of Sciences},
  80(6):1783--1786, 1983.

\bibitem{koellner2010large}
P.~Koellner and W.H. Woodin.
\newblock Large cardinals from determinacy.
\newblock {\em Handbook of Set Theory}, pages 1951--2119, 2010.

\bibitem{krueger2008general}
J.~Krueger.
\newblock A general {M}itchell style iteration.
\newblock {\em Mathematical Logic Quarterly}, 54(6):641--651, 2008.

\bibitem{FSIT}
W.~J. Mitchell and J.~R. Steel.
\newblock {\em Fine structure and iteration trees}, volume~3 of {\em Lecture
  Notes in Logic}.
\newblock Springer-Verlag, Berlin, 1994.

\bibitem{ATHM}
G.~Sargsyan.
\newblock {\em Hod mice and the mouse set conjecture}, volume 236 of {\em
  Memoirs of the {A}merican {M}athematical {S}ociety}.
\newblock American Mathematical Society, 2014.

\bibitem{sargsyanCovering2013}
G.~Sargsyan.
\newblock Covering with universally {B}aire operators.
\newblock {\em Advances in Mathematics}, 268:603--665, 2015.

\bibitem{sargsyan2014Rmice}
G.~Sargsyan and J.~R. Steel.
\newblock The {M}ouse {S}et {C}onjecture for sets of reals, available at
  http://www.math.rutgers.edu/$\sim$gs481/papers.html.
\newblock 2014.
\newblock To appear in the Journal of Symbolic Logic.

\bibitem{schimmerling2007coherent}
Ernest Schimmerling.
\newblock Coherent sequences and threads.
\newblock {\em Advances in Mathematics}, 216(1):89--117, 2007.

\bibitem{CMI}
R.~Schindler and J.~R. Steel.
\newblock {\em The core model induction}.
\newblock available at http://www.math.uni-muenster.de/logik/Personen/rds/.
  2013.

\bibitem{schindler2009self}
Ralf Schindler and J.~R. Steel.
\newblock The self-iterability of ${L} [{E}]$.
\newblock {\em The Journal of Symbolic Logic}, 74(03):751--779, 2009.

\bibitem{trang2013}
F.~Schlutzenberg and N.~Trang.
\newblock Scales in hybrid mice over $\mathbb{R}$.
\newblock {\em arXiv preprint arXiv:1210.7258}, 2014.

\bibitem{solovay1978independence}
R.~Solovay.
\newblock The independence of $\textsf{DC}$ from $\textsf{AD}$.
\newblock In {\em Cabal Seminar 76--77}, pages 171--183. Springer, 1978.

\bibitem{solovay1974strongly}
Robert~M Solovay et~al.
\newblock Strongly compact cardinals and the $\sf{GCH}$.
\newblock In {\em Proceedings of the Tarski symposium}, volume~25, pages
  365--372. Amer. Math. Soc. Providence, RI, 1974.

\bibitem{CMIP}
J.~R. Steel.
\newblock {\em The core model iterability problem}, volume~8 of {\em Lecture
  Notes in Logic}.
\newblock Springer-Verlag, Berlin, 1996.

\bibitem{PFA}
J.~R. Steel.
\newblock {$\textsf{PFA}$} implies {$\textsf{AD}^{L(\mathbb{R})}$}.
\newblock {\em J. Symbolic Logic}, 70(4):1255--1296, 2005.

\bibitem{DMATM}
J.~R. Steel.
\newblock Derived models associated to mice.
\newblock In {\em Computational prospects of infinity. {P}art {I}.
  {T}utorials}, volume~14 of {\em Lect. Notes Ser. Inst. Math. Sci. Natl. Univ.
  Singap.}, pages 105--193. World Sci. Publ., Hackensack, NJ, 2008.

\bibitem{K(R)}
J.~R. Steel.
\newblock Scales in {$K(\Bbb R)$}.
\newblock In {\em Games, scales, and {S}uslin cardinals. {T}he {C}abal
  {S}eminar. {V}ol. {I}}, volume~31 of {\em Lect. Notes Log.}, pages 176--208.
  Assoc. Symbol. Logic, Chicago, IL, 2008.

\bibitem{Scalesendgap}
J.~R. Steel.
\newblock Scales in {$K(\Bbb R)$} at the end of a weak gap.
\newblock {\em J. Symbolic Logic}, 73(2):369--390, 2008.

\bibitem{steelderived}
J.~R. Steel and N.~Trang.
\newblock $\textsf{AD}^+$, derived models, and {$\Sigma_1$}-reflection.
\newblock {\em available at
  http://math.berkeley.edu/$\sim$steel/papers/Publications.html}, 2010.

\bibitem{DFSR}
John Steel and Stuart Zoble.
\newblock Determinacy from strong reflection.
\newblock {\em Transactions of the American Mathematical Society},
  366(8):4443--4490, 2014.

\bibitem{steel2010outline}
J.R. Steel.
\newblock An outline of inner model theory.
\newblock {\em Handbook of set theory}, pages 1595--1684, 2010.

\bibitem{TrangCoherent}
N.~Trang.
\newblock A construction of a coherent sequence in ${K}(\mathbb{R})$.
\newblock {\em available at math.cmu.edu/$\sim$namtrang}, 2011.

\bibitem{trangThesis2013}
N.~Trang.
\newblock {\em Generalized Solovay Measures, the HOD Analysis, and the Core
  Model Induction}.
\newblock PhD thesis, UC Berkeley, 2013.

\bibitem{viale2011consistency}
M.~Viale and C.~Wei{\ss}.
\newblock On the consistency strength of the proper forcing axiom.
\newblock {\em Advances in Mathematics}, 228(5):2672--2687, 2011.

\bibitem{weiss2012combinatorial}
C.~Wei{\ss}.
\newblock The combinatorial essence of supercompactness.
\newblock {\em Annals of Pure and Applied Logic}, 163(11):1710--1717, 2012.

\bibitem{wilson2012contributions}
T.~Wilson.
\newblock {\em Contributions to descriptive inner model theory}.
\newblock PhD thesis, UC Berkeley, 2012.

\bibitem{Woodin}
W.~H. Woodin.
\newblock {\em The axiom of determinacy, forcing axioms, and the nonstationary
  ideal}, volume~1 of {\em de Gruyter Series in Logic and its Applications}.
\newblock Walter de Gruyter \& Co., Berlin, 1999.

\end{thebibliography}
\end{document}